\documentclass[11pt,reqno]{amsproc}
\usepackage{tikz}
\usetikzlibrary{matrix,arrows}
\usepackage{amsmath,amscd}
\usepackage{cite}
\usepackage{color}
\usepackage{graphicx}
\usepackage{psfrag}
\usepackage{multirow}
\usepackage{rotating}
\usepackage{fullpage}
\usepackage{hyperref}
\usepackage{subfigure}
\usepackage{enumerate}
\usepackage{algorithmic}
\usepackage{algorithm}
\usepackage[small]{caption}
\newtheorem{theorem}{Theorem}[section]
\newtheorem{lemma}[theorem]{Lemma}

\newtheorem{remark}[theorem]{Remark}

\numberwithin{equation}{section}

\title[Diffusive-reactive systems]{A numerical framework for 
diffusion-controlled bimolecular-reactive systems to enforce 
maximum principles and the non-negative constraint}

\author{K.~B.~Nakshatrala} \author{M.~K.~Mudunuru} \author{A.~J.~Valocchi} 

\address{Correspondence to: Dr. Kalyana Babu Nakshatrala, 
  Department of Civil and Environmental Engineering, 
  Engineering Building \#1, Room \#135, University of 
  Houston, Houston, Texas 77204-4003, USA. 
  TEL:+1-713-743-4418} \email{knakshatrala@uh.edu} 
\address{Maruti Kumar Mudunuru, Graduate Student, Department 
  of Civil and Environmental Engineering, University of Houston, 
  Houston, Texas 77204-4003, USA.} 
\address{Dr.~Albert Valocchi, Department of Civil and 
Environmental Engineering, Newmark Civil Engineering 
  Laboratory, University of Illinois at Urbana-Champaign, 
  Urbana, Illinois 61801, USA.} 
\date{\today}

\begin{document}

\begin{abstract}
We present a novel computational framework for diffusive-reactive 
systems that satisfies the non-negative constraint and maximum 
principles on general computational grids. The governing equations 
for the concentration of reactants and product are written in terms 
of tensorial diffusion-reaction equations. 
We restrict our studies to fast irreversible bimolecular reactions. 
If one assumes that the reaction is diffusion-limited and all 
chemical species have the same diffusion coefficient, one 
can employ a linear transformation to rewrite the governing 
equations in terms of \emph{invariants}, which are unaffected 
by the reaction. This results in two uncoupled tensorial diffusion 
equations in terms of these invariants, which are solved using a 
novel non-negative solver for tensorial diffusion-type equations. 
The concentrations of the reactants and the product are then 
calculated from invariants using algebraic manipulations. The 
novel aspect of the proposed computational framework is that 
it will always produce physically meaningful non-negative 
values for the concentrations of all chemical species. 
Several representative numerical examples are presented to 
illustrate the robustness, convergence, and the numerical 
performance of the proposed computational framework. We 
will also compare the proposed framework with other popular 
formulations. In particular, we will show that the Galerkin 
formulation (which is the standard single-field formulation) 
does not produce reliable solutions, and the reason can be 
attributed to the fact that the single-field formulation does 
not guarantee non-negative solutions. We will also show 
that the clipping procedure (which produces non-negative 
solutions but is considered as a variational crime) does not 
give accurate results when compared with the proposed 
computational framework. 
\end{abstract}
  
\maketitle


\section{INTRODUCTION AND MOTIVATION}
\label{Sec:S1_Introduction}  
Mixing of chemical species across plume boundaries has a major 
influence on the fate of reactive pollutants in subsurface flows. 
In many practical cases, the intrinsic rate of reaction is 
fast compared with other relevant time scales and hence the 
reaction may be assumed instantaneous (e.g., see references 
\cite{McCarty_Criddle,Dentz_Borgne_Englert_Bijeljic_JHC_2011_v120_p1}).
Mixing is commonly modeled as an anisotropic Fickian diffusion 
process, with the effective diffusion coefficient aligned with 
the flow velocity (termed the longitudinal hydrodynamic dispersion 
coefficient) being much larger than the transverse components. 
Moreover, the small-scale spatial variability of permeability 
in real aquifers leads to highly heterogeneous velocity fields, 
which means that the principal directions of the diffusion tensor 
will vary in space and will not be aligned with the numerical 
grid. It has also been demonstrated that accurate treatment 
of dispersive mixing is crucial for computing the large-time 
spatial extent of the contaminant plume, which is a very 
important measure of the contaminant risk and thus of great 
interest to government regulators and the general public 
(e.g., see Reference \cite{Cirpka_Valocchi_AWR_2007_v30}).
It is well-known that heterogeneity and anisotropy lead to 
irregular plume boundaries, which enhance mixing-controlled 
reactions through increasing the interfacial area of the 
plume. It is, therefore, crucial to capture heterogeneity 
and anisotropy in order to properly model reactive transport 
in hydrogeological systems. 

One can capture heterogeneity computationally through adequate 
mesh refinement or by employing multiscale methods (for example, 
see references \cite{Arbogast_LNP_2000_v552_p35,
Chen_Hou_MathComput_2003_v72_p541,Aarnes_MMS_2004_v2_p421,
Aarnes_Kippe_Lie_AWR_2005_v28_p257}). Although simulating flow 
and transport in highly heterogeneous porous media is still an 
active area of research, we shall not address this aspect in 
this paper. We shall employ meshes that are fine enough to be 
able to resolve heterogeneity. This paper focuses on resolving 
anisotropy to be able to accurately predict the fate of chemical 
species. We shall model the spatial and temporal variation of 
chemical species through diffusion-reaction equations. 

Diffusive-reactive equations arise naturally in modeling biological 
\cite{Murray,Farkas}, chemical \cite{Epstein_Pojman,Gray_Scott}, and 
physical \cite{Walgraef} systems. The areas of application of 
diffusion-reaction systems range from contaminant transport 
\cite{Pinder_Celia}, semiconductors \cite{Scholl}, combustion 
theory \cite{Williams} to biological population dynamics 
\cite{Fisher}. It is well-known that these types of equations 
can exhibit complex and well-ordered structures/patterns 
\cite{Cross_Hohenberg_RMP_1993_v65_p851,Walgraef}. A lot of 
effort has also been put into the mathematical analysis of 
diffusive-reactive systems. Many qualitative mathematical 
properties (e.g., comparison principles, maximum principles, 
the non-negative constraint) have been obtained for these 
systems \cite{Pao,Evans_PDE,Gilbarg_Trudinger}. A detailed 
discussion on mathematical aspects of diffusive-reactive 
systems is beyond the scope of this paper and is not central 
to the present study. 

\subsection{Fast bimolecular reactions}
\emph{The main aim of this paper is to present a robust 
  computation framework to obtain physically meaningful 
  numerical solutions for diffusive-reactive systems. We 
  shall restrict to bimolecular fast reactions for which the 
  rate of reaction is controlled by diffusion of the two 
  chemical species.}
Fast irreversible bimolecular diffusion-reaction equations 
come under the class of diffusion-controlled reactions 
\cite{Rice_Stephen,Kotomin_Kuzovkov}. Examples of such 
systems include ionic reactions occurring in aqueous 
solutions  such as acid-base reactions 
\cite{2001_Cohen_Huppert_Agmon_JPCA_v105_p7165_7173,
1997_Pines_Magnes_Lang_Fleming_v281_p413_420}, polymer 
chain growth kinetics \cite{1962_Benson_North_JACS_v84_p935_940}, 
catalytic reactions \cite{1969_Barnett_Jencks_JACS_v91_p6758_6765} 
and enzymatic reactions \cite{1958_Alberty_Hammes_JPC_v62_p154_159}. 
In such systems the formation of the product is much faster than the 
diffusion of the reactants. In this paper we are only concerned about 
homogeneous reactions (i.e., the reactants are all in the same phase). 
Modeling diffusive-reactive systems in which reactions are heterogeneous 
(i.e., the reactants are in different phases) by incorporating the 
surface effects of the reactants (for example, see references 
\cite{1973_Wilemski_Fixman_TJCP_v58_p4009_4019,
1982_Keizer_JPC_v86_p5052_5067,1974_Chen_Ping_SS_v17_p664_680,
1982_Chou_Zhou_JACS_v104_p1409_1413,1986_Cukier_JSP_v42_p69_82,
1983_Calef_Deutch_ARPC_v34_p493_524,2007_Goldstein_Levine_Torney_SIAMJAM_v67_p1147_1165}) 
is beyond the scope of the present paper but will be considered in our future works. 

Some main challenges in solving a system of diffusion-reaction 
equations are as follows:
\begin{enumerate}[(a)]
\item \emph{Anisotropy:} Developing robust computational 
  frameworks for highly anisotropic diffusive-reactive systems 
  is certainly gaining prominence. However, caution needs to be 
  exercised in selecting numerical formulations to avoid negative 
  values for the concentration of the chemical species. Many 
  popular numerical schemes such as the standard single-field 
  formulation \cite{Hughes}, the lowest order Raviart-Thomas 
  formulation \cite{Raviart_Thomas_MAFEM_1977_p292}, and the 
  variational multiscale mixed formulation 
  \cite{Masud_Hughes_CMAME_2002_v191_p4341,
    Nakshatrala_Turner_Hjelmstad_Masud_CMAME_2006_v195_p4036} 
  give unphysical negative values for the concentration even 
  for pure diffusion equations 
  \cite{Nakshatrala_Valocchi_JCP_2009_v228_p6726}. 
  Furthermore, mesh refinement (either $h$-refinement 
  \cite{Nagarajan_Nakshatrala_IJNMF_2011_v67_p820} or $p$-refinement 
  \cite{Payette_Nakshatrala_Reddy_IJNME_2012}) will not 
  overcome the problem of negative values for the concentration. 
  Keeping in mind about these concerns and developing a robust 
  framework for diffusive-reactive systems is a challenging task.
\item \emph{Nonlinearity:} The equations governing 
  for these systems are coupled and nonlinear (see equations 
  \eqref{Eqn:Mixture_for_A}--\eqref{Eqn:NN_species_IC}, which 
  are presented in a subsequent section). The solutions to 
  these diffusive-reactive systems can exhibit steep gradients 
  \cite{Logan_NonlinearPDEs}, and a robust numerical solution 
  procedure should be able to handle such features.
\item \emph{Scale dependence:} These types of systems can 
  exhibit multiple spatial and temporal scales. For example, 
  the diffusion process can be much slower than the chemical 
  reactions. Therefore, the numerical techniques should be 
  able to resolve these multiple scales. Compared to fast 
  reactions, different solution strategies are required 
  for moderate and slow reactions, which are typically 
  easier to solve due to smaller gradients. Construction of 
  adaptive numerical techniques that take the advantage of 
  specific reaction kinetics and simultaneously satisfying 
  the underlying mathematical properties is still in its 
  infancy. 
\item \emph{Bifurcations, physical instabilities and pattern 
  formations:} These systems are capable of exhibiting physical 
  instabilities, and even chaos \cite{Pao}.
\end{enumerate}

We shall overcome the first challenge by employing a novel non-negative 
solver that ensures physically meaningful non-negative values for the 
concentration of chemical species. We employ a transformation of variables, 
which will overcome the second and third challenges. We solve problems that 
do not exhibit physical instabilities and chaos. Although considerable 
progress has been made in developing numerical solutions of diffusive-reactive 
systems \cite{Hundsdorfer_Verwer,Mei}, none of these studies addressed 
the difficulties in obtaining non-negative solutions especially under 
strong anisotropy, which is the main focus of this paper. 
The following systematic approach has been employed to achieve the 
desired goal. The reactions are assumed to be fast and bimolecular. 
The concentration of reactants and product are governed by tensorial 
diffusion-reaction equations, and by an appropriate stoichiometric 
relationship. The three coupled tensorial diffusion-reaction equations 
are rewritten in terms of invariants, which are unaffected by the 
reaction. This procedure results in two uncoupled tensorial diffusion 
equations in terms of the invariants. A robust non-negative solver is 
employed to solve these two uncoupled tensorial diffusion equations 
following the ideas outlined in references 
\cite{Nakshatrala_Valocchi_JCP_2009_v228_p6726,
Nagarajan_Nakshatrala_IJNMF_2011_v67_p820}.
The concentrations of the reactants and product are then 
calculated using simple algebraic relations given by the 
transformation that is employed.

\subsection{An outline of the paper}
The remainder of the paper is organized as follows. In Section 
\ref{Sec:S2_Bimolecular_GE}, we first present the governing 
equations for a diffusive-reactive system in a general setting. 
We later restrict the study to fast bimolecular reactions, and 
deduce the corresponding governing equations. In Section 
\ref{Sec:S3_Bimolecular_FEM}, we present a novel non-negative 
computational framework for solving diffusive-reactive systems. 
Section \ref{Sec:S4_NN_Theoretical} provides a theoretical 
discussion on the optimization-based solver for enforcing 
maximum principles. Representative numerical results are 
presented in Section \ref{Sec:S5_NN_NR}, and conclusions 
are drawn in Section \ref{Sec:S6_NN_Closure}. 

The standard symbolic notation is adopted in this paper. We shall 
denote scalars by lowercase English alphabet or lowercase Greek 
alphabet (e.g., concentration $c$ and density $\rho$). We shall 
make a distinction between vectors in the continuum and finite 
element settings. Similarly, a distinction is made between 
second-order tensors in the continuum setting versus matrices 
in the context of the finite element method. The continuum 
vectors are denoted by lower case boldface normal letters, 
and the second-order tensors will be denoted using upper 
case boldface normal letters (e.g., vector $\mathbf{x}$ 
and second-order tensor $\mathbf{D}$). In the finite element 
context, we shall denote the vectors using lower case boldface 
italic letters, and the matrices are denoted using upper case 
boldface italic letters (e.g., vector $\boldsymbol{v}$ and 
matrix $\boldsymbol{K}$). It should be noted that repeated 
indices do not imply summation. (That is, Einstein's summation 
convention is not employed in this paper.) Other notational 
conventions adopted in this paper are introduced as needed.

\section{GOVERNING EQUATIONS: DIFFUSIVE-REACTIVE SYSTEMS}
\label{Sec:S2_Bimolecular_GE}
We now present governing equations of a simple diffusive-reactive 
system in a rigid porous medium. Let $\Omega \subset \mathbb{R}^{nd}$ 
be an open bounded domain, where ``$nd$" denotes the number of spatial 
dimensions. The boundary is denoted by $\partial \Omega$, which is 
assumed to be piecewise smooth. Mathematically, $\partial \Omega := 
\bar{\Omega} - \Omega$, where $\bar{\Omega}$ is the set closure of 
$\Omega$. A spatial point in $\bar{\Omega}$ is denoted by $\mathbf{x}$. 
The divergence and gradient operators with respect to $\mathbf{x}$ are, 
respectively, denoted by $\mathrm{div}[\cdot]$ and $\mathrm{grad}[\cdot]$. 
The unit outward normal to the boundary is denoted by $\mathbf{n}
(\mathbf{x})$. Let $t \in \, ]0, \mathcal{I}[$ denote the time, 
where $\mathcal{I}$ is the length of time of interest. 
We shall write the governing equations for the individual chemical 
species and for the mixture on the whole using the mathematical 
framework provided by the theory of interacting continua 
\cite{Bowen}. 

We shall assume that there are $n$ chemical species, which 
include both reactants and product(s). We shall denote the 
molar mass of the $i$-th species by $\zeta_i$. The mass 
concentration of the $i$-th species at the spatial position 
$\mathbf{x}$ and time $t$ is denoted by $\varphi_i(\mathbf{x},
t)$. The molar concentration of the $i$-th species is denoted 
by $c_i(\mathbf{x},t)$ and is defined as follows:
\begin{align}
  c_i(\mathbf{x},t) := \frac{\varphi_i(\mathbf{x},t)}{\zeta_i}
\end{align}
For each chemical species $(i = 1, \cdots, n)$, the boundary 
is divided into two complementary parts: $\Gamma^{\mathrm{D}}_{i}$ 
and $\Gamma^{\mathrm{N}}_{i}$. $\Gamma^{\mathrm{D}}_{i}$ is that part 
of the boundary on which concentration for the $i$-th species 
is prescribed, and $\Gamma^{\mathrm{N}}_{i}$ is the part of the 
boundary on which the flux for the $i$-th chemical species is 
prescribed. For well-posedness and uniqueness of the solution, 
we require $\Gamma^{\mathrm{D}}_i \cap \Gamma^{\mathrm{N}}_i = 
\emptyset$, $\Gamma^{\mathrm{D}}_i \cup \Gamma^{\mathrm{N}}_i = 
\partial \Omega$, and $\mathrm{meas}(\Gamma^{\mathrm{D}}_i) 
> 0$. 
The governing equations for the fate of the $i$-th 
chemical species can be written as follows:
\begin{subequations}
  \label{Eqn:Mixture_for_A_B_C}
  \begin{align}
    \label{Eqn:Mixture_for_A}
    &\frac{\partial \varphi_i}{\partial t} - \mathrm{div}[\mathbf{D}
    (\mathbf{x}) \, \mathrm{grad}[\varphi_i]] =  m_i(\mathbf{x},t) + 
    g_i(\mathbf{x},\varphi_1, \cdots, \varphi_n, t) \quad \mathrm{in} 
    \; \Omega \times ]0, \mathcal{I}[ \\
    \label{Eqn:Mixture_for_Dirchlet}
    &\varphi_i(\mathbf{x},t) = \varphi^{\mathrm{p}}_i(\mathbf{x},t) 
    \quad \mathrm{on} \; \Gamma^{\mathrm{D}}_{i} \times 
    ]0, \mathcal{I}[ \\
    \label{Eqn:Mixture_for_Neumann}
    &\mathbf{n}(\mathbf{x}) \cdot \mathbf{D} (\mathbf{x}) \,  
    \mathrm{grad}[\varphi_i] = \vartheta^{\mathrm{p}}_i(\mathbf{x},t) 
    \quad \mathrm{on} \; \Gamma^{\mathrm{N}}_{i} \times 
    ]0, \mathcal{I}[ \\
    \label{Eqn:NN_species_IC}
    &\varphi_i(\mathbf{x},t=0) = \varphi^{0}_i(\mathbf{x}) 
    \quad \mathrm{in} \; \Omega 
  \end{align}
\end{subequations}
where $\mathbf{D}(\mathbf{x})$ is the diffusivity tensor, 
$\varphi^{\mathrm{p}}_i(\mathbf{x},t)$ is the prescribed mass 
concentration on the boundary, $\vartheta^{\mathrm{p}}_i(\mathbf{x},
t)$ is the prescribed mass concentration flux on the boundary, 
$\varphi^{0}_i(\mathbf{x})$ is the (prescribed) initial condition, 
$g_i(\mathbf{x}, \varphi_1, \cdots, \varphi_n,t)$ is the rate of 
production/depletion of the $i$-th species due to chemical 
reactions, and $m_i(\mathbf{x},t)$ is the prescribed external 
volumetric mass supply for the $i$-th species. 
Equations \eqref{Eqn:Mixture_for_A}--\eqref{Eqn:NN_species_IC} are, 
respectively, the balance of mass, the Dirichlet boundary condition, 
the Neumann boundary condition, and the initial condition for the 
$i$-th chemical species. The governing equation for the balance 
of mass for the mixture on the whole takes the following form: 
\begin{align}
  \sum_{i=1}^{n} g_i(\mathbf{x},\varphi_1,\cdots, 
  \varphi_n,t) = 0 \quad \forall \mathbf{x} \in 
  \Omega, \; \forall t \in ]0, \mathcal{I}[
\end{align}
In Engineering and Applied Sciences, the above system of equations 
\eqref{Eqn:Mixture_for_A}--\eqref{Eqn:NN_species_IC} is commonly 
referred to as time-dependent or non-stationary diffusion-reaction 
equations. In the theory of partial differential equations, the 
above system of equations is referred to as coupled \emph{semilinear 
  parabolic} partial differential equations \cite{Pao}. These types 
of equations are capable of exhibiting many interesting and complex 
features like instabilities, bifurcations, and even chaos (for 
example, see Pao \cite{Pao}). However, in this paper we shall 
focus only on three interrelated mathematical properties that 
these types of equations satisfy: maximum principles, comparison 
principles, and the non-negative constraint.

For completeness, we shall document the assumptions 
in obtaining the governing equations 
\eqref{Eqn:Mixture_for_A}--\eqref{Eqn:NN_species_IC} 
under the mathematical framework offered by the theory 
of interacting continua. The main assumptions are as 
follows:
\begin{enumerate}[(a)]
\item The porous medium in which the reaction takes 
  place is assumed to be rigid. 
\item The flow aspects in the porous media are 
neglected. Hence, the advection velocity of each 
  chemical species in the mixture is zero.
\item The chemical species do not take partial stresses.
\item The diffusion process is based on the Fickian model.
\item The diffusivity tensor $\mathbf{D}(\mathbf{x})$ is 
  the \emph{same} for all the chemical species, which is 
  the case with many diffusive-reactive systems. (However, 
  it should be noted that in some chemical and biological 
  diffusive-reactive systems the diffusivity tensor 
  $\mathbf{D}(\mathbf{x})$ will be different for 
  different species.)  
\end{enumerate} 
One can obtain a hierarchy of more complicated models by 
relaxing/generalizing one or more of the aforementioned 
assumptions, which will be considered in our future works. 
We shall now restrict our study to bimolecular reactions. 

\subsection{Bimolecular reactions} 
Bimolecular reactions are reactions involving two reactant 
molecules. To this end, consider two chemical species $A$ 
and $B$ that react irreversibly to give the product $C$ 
according to the following stoichiometric relationship:
\begin{align}
  \label{Eqn:Bimolecular_fast_reaction}
    n_{A} \, A \, + \, n_{B} \, B \rightarrow n_{C} \, C
\end{align}
where $n_A$, $n_B$ and $n_C$ are (positive) stoichiometric 
coefficients. The rate of the chemical reaction is denoted 
by $r$, which (in general) can be a nonlinear function of 
the concentrations of the chemical species involved in the 
reaction. For a bimolecular reaction, one can write the 
rate of the reaction at the spatial point $\mathbf{x}$ 
and time $t$ as follows: 
\begin{align}
  \label{Eqn:Reaction_Rate_Law}
    r(\mathbf{x},c_A(\mathbf{x},t),c_B(\mathbf{x},t),
    c_C(\mathbf{x},t),t) = -\frac{g_A}{n_A \zeta_A} = 
    -\frac{g_B}{n_B \zeta_B} = +\frac{g_C}{n_C \zeta_C}
\end{align}

For bimolecular reactions, the governing equations 
\eqref{Eqn:Mixture_for_A}--\eqref{Eqn:NN_species_IC} 
can be rewritten as follows: 
\begin{subequations}
  \label{Eqn:DRs_for_A_B_C}
  \begin{align}
    \label{Eqn:DRs_for_A}
    &\frac{\partial c_A}{\partial t} - \mathrm{div}[\mathbf{D}
    (\mathbf{x}) \, \mathrm{grad}[c_A]] = f_A(\mathbf{x},t) - 
    n_{\small{A}} \, r \quad \mathrm{in} \; \Omega \times ]0, 
    \mathcal{I}[ \\
    \label{Eqn:DRs_for_B} 
    &\frac{\partial c_B}{\partial t} - \mathrm{div}[\mathbf{D}
    (\mathbf{x}) \, \mathrm{grad}[c_B]] = f_B(\mathbf{x},t) - 
    n_{\small{B}} \, r \quad \mathrm{in} \; \Omega \times ]0, 
    \mathcal{I}[ \\
    \label{Eqn:DRs_for_C} 
    &\frac{\partial c_C}{\partial t} - \mathrm{div}[\mathbf{D}
    (\mathbf{x}) \, \mathrm{grad}[c_C]] = f_C(\mathbf{x},t) + 
    n_{\small{C}} \, r \quad \mathrm{in} \; \Omega \times ]0, 
    \mathcal{I}[ \\
    \label{Eqn:DRs_for_Dirchlet}
    &c_i(\mathbf{x},t) = c^{\mathrm{p}}_i(\mathbf{x},t) \quad 
    \mathrm{on} \; \Gamma^{\mathrm{D}}_{i} \times ]0, 
    \mathcal{I}[ \quad (i = A, \, B, \; C) \\
    \label{Eqn:DRs_for_Neumann}
    &\mathbf{n}(\mathbf{x}) \cdot \mathbf{D} (\mathbf{x}) \,  
    \mathrm{grad}[c_i] = h^{\mathrm{p}}_i(\mathbf{x},t) \quad \mathrm{on} 
    \; \Gamma^{\mathrm{N}}_{i} \times ]0, \mathcal{I}[ 
    \quad (i = A, \, B, \; C) \\
    \label{Eqn:DRs_for_IC}
    &c_i(\mathbf{x},t=0) = c^{0}_i(\mathbf{x}) \quad 
    \mathrm{in} \; \Omega \quad (i = A, \, B, \; C)
  \end{align}
\end{subequations}
where the prescribed molar supply and the prescribed molar 
initial condition for the $i$-th species are, respectively, 
defined as follows:
\begin{align}
  f_i(\mathbf{x},t) &= \frac{m_i(\mathbf{x},t)}{\zeta_i} \\
  c_i^{0}(\mathbf{x}) &= \frac{\varphi^{0}_i(\mathbf{x})}{\zeta_i} 
\end{align}
and the prescribed molar concentration and the prescribed molar 
concentration flux for the $i$-th species on the boundary are, 
respectively, defined as follows: 
\begin{align}
  c^{\mathrm{p}}_i(\mathbf{x},t) &= 
  \frac{\varphi^{\mathrm{p}}_i(\mathbf{x},t)}{\zeta_i} \\
  h^{\mathrm{p}}_i(\mathbf{x},t) &= 
  \frac{\vartheta^{\mathrm{p}}_i(\mathbf{x},t)}{\zeta_i}
\end{align}

Before we present the proposed computational 
framework, few remarks are in order.
\begin{remark}
  A bimolecular equation, in its most general form, 
  takes the following form:
  \begin{align}
    \label{Eqn:Bimolecular_general_bimolecular}
    n_A A + n_B B \rightarrow n_C C + n_D D + \cdots
  \end{align}
  We considered bimolecular reactions of the form given 
  in equation \eqref{Eqn:Bimolecular_fast_reaction} just 
  to make the presentation simpler. However, it should 
  be emphasized that with a simple and straightforward 
  extension the framework can handle bimolecular 
  reactions of the form given by equation 
  \eqref{Eqn:Bimolecular_general_bimolecular}.
\end{remark}

\begin{remark}
  One can handle stoichiometry using a stoichiometric matrix 
  (for example, see references \cite{Bowen,Erdi_Toth}). This 
  approach will be convenient and appropriate for complicated 
  chemical reactions that involve multiple stages, multiple 
  chemical species, and intermediate complexes. Herein, 
  we shall not take such an approach, as the reaction is 
  relatively simpler.
\end{remark} 

\begin{remark}
  The mathematical model presented in this section, 
  and the computational framework proposed in this 
  paper can be used to obtain numerical solutions and 
  perform predictive simulations for various practical 
  problems proposed in the literature. 
  To name a few: transverse mixing-limited chemical reactions 
  in groundwater and aquifers \cite{2008_Willingham_etal_EST_v42_p3185_p3193,
    1997_Davies_etal_BJ_v70_p1006_p1016,
    2010_Willingham_etal_v33_p525_p535}, and mixing-controlled 
  bioreactive transport in heterogeneous porous media for 
  engineered bioremediation and contaminant degradation 
  scenarios \cite{Cirpka_Valocchi_AWR_2007_v30,
    2009_Cirpka_Valocchi_AWR_v32_p298_p301,
    2004_Ham_etal_AWR_v27_p803_p813,
    2005_Chu_etal_AWR_v41,
    1986_Borden_Bedient_AWR_v22_p_1973_p1982,
    1986_Borden_Bedient_AWR_v22_p_1983_p1990} 
    Other practically important examples include acid-base 
    reactions \cite{2006_Arratia_Gollub_v96}, atmospheric 
    chemical transport \cite{1997_Thuburn_Tan_JPR_v102_p13037_p13049}, 
    chemistry of marine boundary layer \cite{2006_Glasow_ACP_v6_p3571_p3581}, 
    drinking water treatment \cite{2003_Huber_EST_v37_p1016_p1024}, 
    and reaction enhancement through chaotic flows 
    \cite{2009_Tsang_PRE_v80,2004_Crimaldi_Browning_JMS_v49_p3_p18,
      2004_Crimaldi_Cadwell_Weiss_PoF_v20}.
\end{remark}

\section{PROPOSED COMPUTATIONAL FRAMEWORK}
\label{Sec:S3_Bimolecular_FEM}
We now present a robust numerical framework to solve 
bimolecular diffusive-reactive systems given by equations 
\eqref{Eqn:DRs_for_A}--\eqref{Eqn:DRs_for_IC} that will 
produce physically meaningful non-negative solutions for 
the concentration of chemical species on general computational 
grids. As mentioned earlier, we shall restrict our studies to 
\emph{fast reactions}. By fast reactions we mean that the rate 
of the chemical reaction is much faster than the rate of the 
diffusion process. 

\subsection{A non-negative invariant set, and uncoupled equations} 
Using an appropriate linear transformation, one can obtain two 
independent invariants, which are unaffected by the chemical 
reaction. As we shall see later that these invariants will 
be governed by uncoupled linear diffusion equations. For the 
bimolecular diffusive-reactive system \eqref{Eqn:DRs_for_A_B_C}, 
one can construct three different sets of independent invariants. 
Herein, we shall employ the set of independent invariants that 
inherits the non-negative property, and the corresponding linear 
transformation can be written as follows: 
\begin{subequations}
  \label{Eqn:Definitions_of_F_G_Set2}
  \begin{align}
    \label{Eqn:Definitions_of_F}
    c_F &:= c_A + \left( \frac{n_A}{n_C} \right) c_C \\ 
    \label{Eqn:Definitions_of_G}
    c_G &:= c_B + \left( \frac{n_B}{n_C} \right) c_C 
  \end{align}
\end{subequations}

\begin{remark}
  Let us denote the other two sets of independent 
  invariants by $(U, V)$ and $(M, N)$, which can be 
  devised as follows: 
  \begin{subequations}
    \label{Eqn:Definitions_of_F_G_Set1}
    \begin{align}
      \label{Eqn:Definitions_U_V}
      c_{U} &:= c_A - \left( \frac{n_A}{n_B} \right) c_B, \; 
      c_{V}  := c_A + \left( \frac{n_A}{n_C} \right) c_C \\
      \label{Eqn:Definitions_M_N}
      c_{M} &:= c_B - \left( \frac{n_B}{n_A} \right) c_A, \; 
      c_{N}  := c_B + \left( \frac{n_B}{n_C} \right) c_C
    \end{align}
  \end{subequations}
  As one can see from the above expressions, $c_U$ and $c_M$ need 
  not be always non-negative, which is not unphysical as $U$ and 
  $M$ do not refer to any chemical species. They are just convenient 
  mathematical quantities, and the possibility of negative values 
  for these quantities is just a consequence of their mathematical 
  definitions. We shall not use the above invariants as they do not 
  possess the non-negative property, and hence are not appropriate 
  to be used in a non-negative solver for diffusion-type equations.  
\end{remark}

In the remainder of this paper, we shall assume that $\Gamma^{\mathrm{D}}_{A} 
= \Gamma^{\mathrm{D}}_{B} = \Gamma^{\mathrm{D}}_{C}$, and shall denote them by 
$\Gamma^{\mathrm{D}}$ for notational convenience. Similarly, we shall assume 
that $\Gamma^{\mathrm{N}}_{A} = \Gamma^{\mathrm{N}}_{B} = \Gamma^{\mathrm{N}}_{C}$, 
and shall denote them by $\Gamma^{\mathrm{N}}$. These assumptions will 
facilitate application of the linear transformation to the prescribed boundary 
conditions $c^{\mathrm{p}}_i (\mathbf{x},t)$ and $h^{\mathrm{p}}_i(\mathbf{x},t)$ 
($i = A, \, B, \, C$).
Using straightforward manipulations on equations 
\eqref{Eqn:DRs_for_A}--\eqref{Eqn:DRs_for_IC} based on the 
non-negative invariant set \eqref{Eqn:Definitions_of_F_G_Set2} 
and appealing to the above assumptions, one can show that the 
governing equations for the invariant $F$ can be written as 
follows:
\begin{subequations}
  \begin{align}
    \label{Eqn:Diffusion_for_F}
    &\frac{\partial c_F}{\partial t} - \mathrm{div}[\mathbf{D}
    (\mathbf{x}) \, \mathrm{grad}[c_F]] = f_F(\mathbf{x},t) \quad 
    \mathrm{in} \; \Omega \times ]0, \mathcal{I}[ \\
    \label{Eqn:Diffusion_for_Dirchlet_F}
    &c_F(\mathbf{x},t) = c_F^{\mathrm{p}}(\mathbf{x},t) := 
    c^{\mathrm{p}}_A(\mathbf{x},t) + \left( \frac{n_A}{n_C} 
    \right) c^{\mathrm{p}}_C(\mathbf{x},t) \quad \mathrm{on} 
    \; \Gamma^{\mathrm{D}} \times ]0, \mathcal{I}[ \\
    \label{Eqn:Diffusion_for_Neumann_F}
    &\mathbf{n}(\mathbf{x}) \cdot \mathbf{D} (\mathbf{x}) \, 
    \mathrm{grad}[c_F] =  h^{\mathrm{p}}_F(\mathbf{x},t) := 
    h^{\mathrm{p}}_A(\mathbf{x},t) + \left( \frac{n_A}{n_C} 
    \right) h^{\mathrm{p}}_C(\mathbf{x},t) \quad \mathrm{on} 
    \; \Gamma^{\mathrm{N}} \times ]0, \mathcal{I}[ \\
    \label{Eqn:Diffusion_for_IC_F}
    &c_F(\mathbf{x},t=0) = c^{0}_F(\mathbf{x}) := 
    c^{0}_A(\mathbf{x}) + \left( \frac{n_A}{n_C} \right) 
    c^{0}_C(\mathbf{x}) \quad \mathrm{in} \; \Omega
  \end{align}
\end{subequations} 
Similarly, the governing equations for the invariant 
$G$ take the following form: 
\begin{subequations}
  \begin{align}
    \label{Eqn:Diffusion_for_G} 
    &\frac{\partial c_G}{\partial t} - \mathrm{div}[\mathbf{D}
    (\mathbf{x}) \, \mathrm{grad}[c_G]] = f_G(\mathbf{x},t) \quad 
    \mathrm{in} \; \Omega \times ]0, \mathcal{I}[ \\
    \label{Eqn:Diffusion_for_Dirchlet_G}
    &c_G(\mathbf{x},t) = c^{\mathrm{p}}_G(\mathbf{x},t) := c^{\mathrm{p}}_B
    (\mathbf{x},t) + \left( \frac{n_B}{n_C} \right) c^{\mathrm{p}}_C
    (\mathbf{x},t) \quad \mathrm{on} \; \Gamma^{\mathrm{D}} \times 
      ]0, \mathcal{I}[ \\
    \label{Eqn:Diffusion_for_Neumann_G}
    &\mathbf{n}(\mathbf{x}) \cdot \mathbf{D} (\mathbf{x}) \, 
    \mathrm{grad}[c_G] = h^{\mathrm{p}}_G(\mathbf{x},t) := h^{\mathrm{p}}_B(\mathbf{x},t) + 
    \left( \frac{n_B}{n_C} \right) h^{\mathrm{p}}_C(\mathbf{x},t) 
    \quad \mathrm{on} \; \Gamma^{\mathrm{N}} \times ]0, \mathcal{I}[ \\
    \label{Eqn:Diffusion_for_IC_G}
    &c_G(\mathbf{x},t=0) = c^{0}_G(\mathbf{x}) := c^{0}_B
    (\mathbf{x}) + \left( \frac{n_B}{n_C} \right) c^{0}_C(\mathbf{x})
    \quad \mathrm{in} \; \Omega 
  \end{align}
\end{subequations}
In the above equations, the prescribed molar supplies for the 
invariants $F$ and $G$ in terms of the chemical species $A$, 
$B$, and $C$ take the following form:
\begin{subequations}
  \begin{align}
    \label{Eqn:Volumetric_Source_for_F}
    &f_F(\mathbf{x},t) = f_A(\mathbf{x},t) + 
    \left( \frac{n_A}{n_C} \right) f_C(\mathbf{x},t) \\
    \label{Eqn:Volumetric_Source_for_G}
    &f_G(\mathbf{x},t) = f_B(\mathbf{x},t) + 
    \left( \frac{n_B}{n_C} \right) f_C(\mathbf{x},t)
  \end{align}
\end{subequations}

\subsection{Specializing on fast bimolecular reactions}
\label{SubSec:Specializing_fast_bimolecular_reactions}
Recall that by a fast reaction we mean that the time 
scale of the reaction is much faster than the time 
scales associated with the diffusion process of the 
chemical species. 
Hence, for bimolecular reactions that are fast, it is a 
good approximation to assume that the species $A$ 
and $B$ cannot co-exist at any given instance of time at 
the same location $\mathbf{x}$. This assumption will be 
exact in the limit of infinitely fast / instantaneous 
reactions.
Using equations 
\eqref{Eqn:Definitions_of_F}--\eqref{Eqn:Definitions_of_G}, 
one can rewrite the whole problem for fast irreversible 
bimolecular fast reactions in terms of the conserved 
quantities $c_F(\mathbf{x},t)$ and $c_G(\mathbf{x},t)$ 
as follows:
\begin{subequations}
  \label{Eqn:Fast_A_B_C}
  \begin{align}
    \label{Eqn:Fast_A}
    &c_A(\mathbf{x},t) = \mathrm{max} \left[c_F(\mathbf{x},t) 
      - \left(\frac{n_A}{n_B}\right) c_G(\mathbf{x},t), \, 0 
      \right] \\
    \label{Eqn:Fast_B}
    &c_B(\mathbf{x},t) = \left( \frac{n_B}{n_A} \right) \; 
    \mathrm{max}\left[- c_F(\mathbf{x},t) + 
      \left(\frac{n_A}{n_B} \right) c_G(\mathbf{x},t), \, 0 \right] \\
    \label{Eqn:Fast_C}
    &c_C(\mathbf{x},t) = \left( \frac{n_C}{n_A} \right) \; 
    \left(c_F(\mathbf{x},t) - c_A(\mathbf{x},t) \right)
  \end{align}
\end{subequations}
It should be noted that the original coupled diffusive-reactive system 
\eqref{Eqn:DRs_for_A_B_C} is nonlinear and the solution procedure to 
obtain the concentrations of $A$, $B$, and $C$ is \emph{still} nonlinear. 
This is because of the fact that $\mathrm{max}[\cdot,\cdot]$ is a nonlinear 
operator. However, the obvious advantage is that the resulting equations 
are much simpler to solve.

Before we discuss about the non-negative constraint, we shall introduce 
relevant notation. Let $D_T := \Omega \times ]0, \mathcal{I}]$, and $S_T 
:= \partial \Omega \times ]0, \mathcal{I}]$. Let $C^{m}(K)$ denote the 
set of all $m$-times continuously differentiable functions on an open 
set $K$. The set of all continuous functions on the (set) closure of 
$K$ is denoted by $C(\bar{K})$. The set of all functions that are 
$m$-times continuously differentiable in $\mathbf{x} \in \Omega$ 
and $l$-times continuously differentiable in $t \in ]0, \mathcal{I}]$ 
is denoted by $C^{m,l}(D_T)$. From the theory of partial differential 
equations we have the following non-negative lemma for (pure) 
transient diffusion-type equations. 
\begin{lemma}[The non-negative constraint for transient diffusion equation]
  \label{Lemma:Non_negative_Lemma}
  Let $c(\mathbf{x},t) \in C^{2,1}(D_T) \cap C(\bar{D}_T)$ 
  such that 
  \begin{subequations}
    \begin{align}
      &\frac{\partial c}{\partial t} - \mathrm{div}[\mathbf{D}
        (\mathbf{x})\mathrm{grad}[c]] \geq 0 \quad \mathrm{in} 
      \; D_T \\
      &\alpha_0 \mathbf{n}(\mathbf{x}) \cdot \mathbf{D}(\mathbf{x}) 
      \mathrm{grad}[c]  + \beta_0 c(\mathbf{x},t) \geq 0 \quad 
      \mathrm{on} \; S_T \\
    &c(\mathbf{x},t=0) \geq 0 \quad \mathrm{in} \; \Omega
    \end{align}
  \end{subequations}
  where $\alpha_0 \geq 0$, $\beta_0 \geq 0$, and $\alpha_0 + 
  \beta_0 > 0$ on $S_T$. Then $c(\mathbf{x},t) > 0$ in 
  $D_T$ unless it is identically zero. 
\end{lemma}
\begin{proof}
  A proof can be found in McOwen \cite{McOwen}, 
  Evans \cite{Evans_PDE} or Pao \cite{Pao}.
\end{proof}
Using the above lemma, we now prove that all the chemical species in 
a bimolecular fast reaction also satisfy the non-negative constraint, 
which will be the main result of the paper on the mathematical front. 
\begin{theorem}[The non-negative constraint for transient bimolecular 
    fast reactions]
  \label{Theorem:Non_negative_Theorem}
  Let $c_{i}(\mathbf{x},t) \in C^{2,1}(D_T)\cap C(\bar{D}_T)$ 
  for $i = A, \, B, \, C$. If $f_i(\mathbf{x},t) \geq 0$, 
  $c_i^{\mathrm{p}}(\mathbf{x},t) \geq 0$, $c_i^{0}(\mathbf{x}) 
  \geq 0$  and $h_i^{\mathrm{p}} (\mathbf{x},t) \geq 0$ 
  $(i = A, \, B, \, C)$ then for bimolecular fast reactions 
  we have $c_i(\mathbf{x},t) \geq 0$ $\forall \mathbf{x} 
  \in \bar{\Omega}$ and $t \in ]0, \mathcal{I}]$.
\end{theorem}
\begin{proof}
  Since $c_{i}(\mathbf{x},t) \in C^{2,1}(D_T) \cap C(\bar{D}_T)$ 
  $(i = A, \, B, \, C)$, using equations 
  \eqref{Eqn:Definitions_of_F}--\eqref{Eqn:Definitions_of_G}, 
  it is evident that $c_{F}(\mathbf{x},t), \, c_{G}(\mathbf{x},t) 
  \in C^{2,1}(D_T) \cap C(\bar{D}_T)$. Based on the hypothesis 
  given in the theorem and using equations
  \eqref{Eqn:Diffusion_for_Dirchlet_F}--\eqref{Eqn:Diffusion_for_IC_F}, 
  \eqref{Eqn:Diffusion_for_Dirchlet_G}--\eqref{Eqn:Diffusion_for_IC_G} 
  and \eqref{Eqn:Volumetric_Source_for_F}--\eqref{Eqn:Volumetric_Source_for_G}, 
  one can conclude that 
  \begin{subequations} 
    \begin{align}
      f_F(\mathbf{x},t) \geq 0, \; c^{\mathrm{p}}_F(\mathbf{x},t) 
      \geq 0, \; h^{\mathrm{p}}_F(\mathbf{x},t) \geq 0, \; c^{0}_F
      (\mathbf{x},t) \geq 0 \\
      f_G(\mathbf{x},t) \geq 0, \; c^{\mathrm{p}}_G(\mathbf{x},t) 
      \geq 0, \; h^{\mathrm{p}}_G(\mathbf{x},t) \geq 0, \; c^{0}_G
      (\mathbf{x},t) \geq 0
    \end{align}
  \end{subequations}
  Using Lemma \ref{Lemma:Non_negative_Lemma}, one can 
  conclude that 
  \begin{align}
    c_{F}(\mathbf{x},t) \geq 0 \; \mathrm{and} \; 
    c_{G}(\mathbf{x},t) \geq 0 \quad \forall \mathbf{x} 
    \in \bar{\Omega}, \; t \in ]0, \mathcal{I}]
  \end{align}
  From the relations given in equations \eqref{Eqn:Fast_A} 
  and \eqref{Eqn:Fast_B}, it is evident that
  \begin{align}
    c_{A}(\mathbf{x},t) \geq 0 \; \mathrm{and} \; 
    c_{B}(\mathbf{x},t) \geq 0 \quad \forall \mathbf{x} 
    \in \bar{\Omega}, \; t \in ]0, \mathcal{I}]
  \end{align}
  From equations \eqref{Eqn:Fast_A} and \eqref{Eqn:Fast_C}, 
  one can conclude that $c_{C}(\mathbf{x},t)$ will be equal 
  to either $\frac{n_C}{n_A} c_{F}(\mathbf{x},t)$ or 
  $\frac{n_C}{n_B} c_{G}(\mathbf{x},t)$.
  In either case, we have 
  \begin{align}
    c_{C}(\mathbf{x},t) \geq 0 \quad \forall \mathbf{x} 
    \in \bar{\Omega}, \; t \in ]0, \mathcal{I}]
  \end{align}
  This concludes the proof. 
\end{proof}
\begin{remark}
  One can show that even the steady-state solution of bimolecular fast 
  reactions satisfies the non-negative constraint. A non-negative result 
  similar to Lemma \ref{Lemma:Non_negative_Lemma} exists for steady-state 
  scalar diffusion equation (for example, see Gilbarg and Trudinger 
  \cite{Gilbarg_Trudinger}). Using such a result and following the 
  arguments presented in Theorem \ref{Theorem:Non_negative_Theorem}, 
  one can show that the concentrations of the chemical species in a 
  steady-state bimolecular fast reaction also satisfy the non-negative 
  constraint. 
\end{remark}

\subsection{Numerical solution strategy}
\label{SubSec:Numerical_Solution_Strategy}
The proposed numerical solution strategy for simulating 
\emph{fast} irreversible bimolecular reactions boils down 
to solving two uncoupled tensorial diffusion equations. 
The procedure can be outlined as follows: 
\begin{itemize}
\item Solve the uncoupled tensorial diffusion equations given by 
  \eqref{Eqn:Diffusion_for_F}--\eqref{Eqn:Diffusion_for_IC_F} and 
  \eqref{Eqn:Diffusion_for_G}--\eqref{Eqn:Diffusion_for_IC_G} to 
  obtain $c_F$ and $c_G$. 
\item Using the relations given in equations 
  \eqref{Eqn:Fast_A}--\eqref{Eqn:Fast_C}, compute 
  the concentrations of reactants $A$ and $B$, and 
  product $C$.   
\end{itemize}
In the next subsection we shall outline a non-negative numerical 
solver for tensorial diffusion equation, which will be used to 
obtain the invariants $F$ and $G$. It should be noted that if 
one obtains non-negative values for the invariants, based on 
the second step in the aforementioned procedure, it is obvious 
that the concentrations for $A$, $B$ and $C$ will be non-negative 
even in the numerical setting. 

\subsection{A non-negative solver}
In references \cite{Liska_Shashkov_CiCP_2008_v3_p852,
  Nakshatrala_Valocchi_JCP_2009_v228_p6726,
  Nagarajan_Nakshatrala_IJNMF_2011_v67_p820,
  Nakshatrala_Nagarajan_arXiv_2012} numerical methodologies 
have been proposed for diffusion-type equations based on 
optimization techniques. In this paper, we shall extend 
these methodologies to both steady-state and transient 
solutions for diffusion-controlled bimolecular-reactive 
systems. 

\subsubsection{Steady-state analysis:} 
The weak statements based on the standard Galerkin formulation 
to obtain the steady-state solutions for the invariants $F$ 
and $G$ can be written as follows: Find $c_F(\mathbf{x}) \in 
\mathcal{P}_F$ and $c_G(\mathbf{x}) \in \mathcal{P}_G$ such 
that we have
\begin{subequations}
  \begin{align}
    \label{Eqn:SteadyState_single_field_formulation_diffusion}
    \mathcal{B}(q;c_F) = L_{F}(q) \quad \forall  
    q(\mathbf{x}) \in \mathcal{Q} \\
    \mathcal{B}(q;c_G) = L_{G}(q) \quad \forall  
    q(\mathbf{x}) \in \mathcal{Q}
  \end{align}
\end{subequations}
where the bilinear form and the linear functionals are 
defined as follows: 
\begin{subequations}
  \begin{align}
    \mathcal{B}(q;c) &:= \int_{\Omega} \mathrm{grad}[q(\mathbf{x})] 
    \cdot \mathbf{D}(\mathbf{x}) \; \mathrm{grad}[c(\mathbf{x})] \; 
    \mathrm{d} \Omega \\
    L_{F}(q) &:= \int_{\Omega} q(\mathbf{x}) \; f_F(\mathbf{x}) \; 
    \mathrm{d} \Omega + \int_{\Gamma^{\mathrm{N}}} q(\mathbf{x}) 
    \; h^{\mathrm{p}}_F(\mathbf{x}) \; \mathrm{d} \Gamma \\
    L_{G}(q) &:= \int_{\Omega} q(\mathbf{x}) \; f_G(\mathbf{x}) \; 
    \mathrm{d} \Omega + \int_{\Gamma^{\mathrm{N}}} q(\mathbf{x}) 
    \; h^{\mathrm{p}}_G(\mathbf{x}) \; \mathrm{d} \Gamma 
  \end{align}
\end{subequations}
In the above discussion, the following function spaces have been used: 
\begin{subequations}
\begin{align}
  \mathcal{P}_F &:= \left\{c(\mathbf{x}) \in H^{1}(\Omega) \; \big| \; 
  c(\mathbf{x}) = c^{\mathrm{p}}_F(\mathbf{x}) \; \mathrm{on} \; 
  \Gamma^{\mathrm{D}}\right\} \\
  \mathcal{P}_G &:= \left\{c(\mathbf{x}) \in H^{1}(\Omega) \; \big| \; 
  c(\mathbf{x}) = c^{\mathrm{p}}_G(\mathbf{x}) \; \mathrm{on} \; 
  \Gamma^{\mathrm{D}}\right\} \\
  \label{Eqn:Function_Space_TestFunction}
  \mathcal{Q} &:= \left\{q(\mathbf{x}) \in H^{1}(\Omega) \; \big| \; 
  q(\mathbf{x}) = 0 \; \mathrm{on} \; \Gamma^{\mathrm{D}}\right\} 
\end{align}
\end{subequations}
where $H^{1}(\Omega)$ is a standard Sobolev space \cite{Brezzi_Fortin}. 
However, it should be noted that the Galerkin formulation does 
not satisfy the non-negative constraint and maximum principles 
for diffusion-type equations (see references 
\cite{Nakshatrala_Valocchi_JCP_2009_v228_p6726,
  Nagarajan_Nakshatrala_IJNMF_2011_v67_p820,
  Nakshatrala_Nagarajan_arXiv_2012}). These violations 
will be exacerbated in the case of diffusive-reactive 
systems (see Section \ref{Sec:S5_NN_NR} in this paper). 
We now modify the Galerkin formulation to meet the 
non-negative constraint and maximum principles for 
fast bimolecular diffusive-reactive systems. 

From Vainberg's theorem \cite{Vainberg,Hjelmstad}, the weak form 
\eqref{Eqn:SteadyState_single_field_formulation_diffusion} has 
an equivalent minimization problem, which can be written as follows:
\begin{subequations} 
\begin{align}
  \label{Eqn:Minimzation_Problem_Statement}
  \mathop{\mathrm{minimize}}_{c_F(\mathbf{x}) \in \mathcal{P}_F} & 
  \quad \frac{1}{2} \mathcal{B}(c_F;c_F) - L_{F}(c_F) \\
  \label{Eqn:Minimzation_Problem_Constraint}
  \mbox{subject to} & \quad c_F^{\mathrm{min}} \leq 
  c_F(\mathbf{x}) \leq c_F^{\mathrm{max}}
\end{align}
\end{subequations}
where $c_F^{\mathrm{min}}$ and $c_F^{\mathrm{max}}$ are, 
respectively, the lower and upper bounds for the 
invariant $F$. These bounds arise due to the 
non-negative constraint and the maximum principle. 
For example, to enforce the non-negative constraint 
one can take $c_F^{\mathrm{min}} = 0$.

After performing finite element discretization using 
low-order finite elements on the above minimization 
problem, a numerical methodology that satisfies the 
non-negative constraint and the maximum principle 
for the invariant $F$ can be written as follows:
\begin{subequations}
  \label{Eqn:NonNegative_Solver_Invariant_F}
  \begin{align}
    \label{Eqn:NonNegative_Solver_Invariant_F_OF}
    \mathop{\mbox{minimize}}_{\boldsymbol{c}_F \in \mathbb{R}^{ndofs}} & 
    \quad \frac{1}{2} \langle \boldsymbol{c}_F;\boldsymbol{K} 
    \boldsymbol{c}_F\rangle - \langle \boldsymbol{c}_F; \boldsymbol{f}_F 
    \rangle \\
    \label{Eqn:NonNegative_Solver_Invariant_F_Con}
    \mbox{subject to} & \quad c_F^{\mathrm{min}} \boldsymbol{1} 
    \preceq \boldsymbol{c}_F \preceq c_F^{\mathrm{max}} 
    \boldsymbol{1} 
  \end{align}
\end{subequations}
where $\langle \cdot;\cdot \rangle$ represents the standard 
inner-product on Euclidean spaces, $\boldsymbol{1}$ denotes 
a vector of ones of size $ndofs \times 1$, and the symbol 
$\preceq$ represents the component-wise inequality for 
vectors. That is, 
\begin{align}
  \boldsymbol{a} \preceq \boldsymbol{b} \quad 
  \mathrm{implies} \quad a_i \leq b_i \quad \forall i
\end{align} 
The other notational conventions employed in the above equations 
\eqref{Eqn:NonNegative_Solver_Invariant_F_OF}--\eqref{Eqn:NonNegative_Solver_Invariant_F_Con} are as follows: $\boldsymbol{c}_F$ is the vector containing 
the nodal concentrations of the invariant $F$, $\boldsymbol{f}_F$ 
is the nodal volumetric source vector, and $\boldsymbol{K}$ is the 
coefficient matrix, which will be symmetric and positive definite. 
The number of degrees-of-freedom in the nodal concentration 
vector is denoted by $ndofs$. Hence the size of each of 
the vectors $\boldsymbol{c}_F$ and $\boldsymbol{f}_F$ will 
be $ndofs \times 1$, and the size of the matrix $\boldsymbol{K}$ 
will be $ndofs \times ndofs$. A numerical methodology for the 
invariant $G$ can be developed in a similar manner by replacing 
$\boldsymbol{f}_F$ with the corresponding volumetric source 
vector $\boldsymbol{f}_G$, $c_F^{\mathrm{min}}$ with $c_G^{\mathrm{min}}$ 
and $c_F^{\mathrm{max}}$ with $c_G^{\mathrm{max}}$. 

It should be noted that the coefficient matrix $\boldsymbol{K}$ 
(which is sometimes referred to ``stiffness'' matrix) for the 
invariant $G$ is the same as that of the invariant $F$. This 
is due to the fact that the diffusivity tensor $\mathbf{D}
(\mathbf{x})$ and the computational domain are the same for 
both the invariants. 
Since the coefficient matrix $\boldsymbol{K}$ is 
positive definite, the constrained optimization 
problem \eqref{Eqn:NonNegative_Solver_Invariant_F} 
belongs to \emph{convex quadratic programming} and 
has a unique global minimizer \cite{Nocedal_Wright}.
The first-order optimality conditions corresponding to the 
optimization problem \eqref{Eqn:NonNegative_Solver_Invariant_F} 
take the following form: 
\begin{subequations}
  \label{Eqn:Bimolecular_FOOC}
  \begin{align}
  \label{Eqn:Bimolecular_FOOC_1}
    &\boldsymbol{K} \boldsymbol{c}_F = \boldsymbol{f}_F 
    + \boldsymbol{\lambda}_F^{\mathrm{min}} - 
    \boldsymbol{\lambda}_F^{\mathrm{max}} \\
    \label{Eqn:Bimolecular_FOOC_2}
    &c_{F}^{\mathrm{min}} \boldsymbol{1} \preceq 
    \boldsymbol{c}_{F} \preceq c_{F}^{\mathrm{max}} 
    \boldsymbol{1} \\
    \label{Eqn:Bimolecular_FOOC_3}
    &\boldsymbol{\lambda}_{F}^{\mathrm{min}} \succeq \boldsymbol{0} \\
    \label{Eqn:Bimolecular_FOOC_4}
    &\boldsymbol{\lambda}_{F}^{\mathrm{max}} \succeq \boldsymbol{0} \\
    \label{Eqn:Bimolecular_FOOC_5}
    &\left(\boldsymbol{c}_F - c_F^{\mathrm{min}} \boldsymbol{1} 
    \right) \cdot \boldsymbol{\lambda}_{F}^{\mathrm{min}} = 0 \\
    \label{Eqn:Bimolecular_FOOC_6}
    &\left(c_F^{\mathrm{max}} \boldsymbol{1} - \boldsymbol{c}_F 
    \right) \cdot \boldsymbol{\lambda}_{F}^{\mathrm{max}} = 0 
  \end{align}
\end{subequations}
where $\boldsymbol{\lambda}_{F}^{\mathrm{min}}$ is the vector 
containing Lagrange multipliers corresponding to the constraint 
$c_{F}^{\mathrm{min}} \boldsymbol{1} \preceq \boldsymbol{c}_{F}$, 
and $\boldsymbol{\lambda}_{F}^{\mathrm{max}}$ is the vector 
containing Lagrange multipliers corresponding to the 
constraint $\boldsymbol{c}_{F} \preceq c_{F}^{\mathrm{max}} 
\boldsymbol{1}$. If the optimization problem for the 
invariant $F$ involves only the non-negative constraint, 
then one has to set $c_{F}^{\mathrm{min}} = 0$ and there 
will be only one set of Lagrange multipliers (i.e., the 
vector $\boldsymbol{\lambda}_F^{\mathrm{min}}$). 
In the optimization literature, the aforementioned 
first-order optimality conditions are commonly 
referred to as the Karush-Kuhn-Tucker (KKT) conditions 
\cite{Boyd_convex_optimization}. One could write similar 
equations for the invariant $G$.
The various steps in the proposed methodology to 
obtain steady-state solutions for bimolecular fast 
diffusive-reactive systems are summarized in Algorithm 
\ref{Algo:SteadyState_Bimolecular}, which will serve as 
a quick reference for computer implementation. 

\begin{remark}
  It should be noted that the shape (or interpolation) 
  functions for low-order finite elements (i.e., two-node 
  line element in 1D; three-node triangle element and 
  four-node quadrilateral element in 2D; and four-node 
  tetrahedron element, six-node wedge element and 
  eight-node brick element in 3D) are non-negative 
  within the element. Therefore, enforcing non-negative 
  constraints on the nodal concentrations will ensure 
  non-negativeness within the element, and hence in the 
  entire computational domain. It should also be noted 
  that the above methodology cannot be extended to 
  high-order finite elements, as high-order interpolation 
  functions can be negative within the finite element. 
  For further details see discussions in references 
  \cite{Nagarajan_Nakshatrala_IJNMF_2011_v67_p820,
    Payette_Nakshatrala_Reddy_IJNME_2012}. 
\end{remark}

\begin{algorithm}
  \caption{Numerical framework for steady-state fast bimolecular diffusive-reactive systems}
  \label{Algo:SteadyState_Bimolecular}
  \begin{algorithmic}[1]
    \STATE Input: Stoichiometric coefficients; volumetric sources; 
    and boundary conditions for the chemical species $A$, $B$, and $C$
    \STATE Calculate corresponding volumetric sources and boundary 
    conditions for the invariants
    \STATE Calculate $c_F^{\mathrm{min}}$, $c_F^{\mathrm{max}}$, 
    $c_G^{\mathrm{min}}$ and $c_G^{\mathrm{max}}$ based on 
    maximum principles and the non-negative constraint
    %
    \STATE Call optimization-based steady-state diffusion 
    solver to obtain $\boldsymbol{c}_F$: 
    \begin{align*}
      \mathop{\mbox{minimize}}_{\boldsymbol{c}_F \in \mathbb{R}^{ndofs}} & 
      \quad \frac{1}{2} \langle \boldsymbol{c}_F;\boldsymbol{K} 
      \boldsymbol{c}_F \rangle - \langle \boldsymbol{c}_F; \boldsymbol{f}_F 
      \rangle \\
      \mbox{subject to} & \quad c_F^{\mathrm{min}} \boldsymbol{1} \preceq 
      \boldsymbol{c}_F \preceq c_F^{\mathrm{max}} \boldsymbol{1} 
    \end{align*}
    %
    \STATE Call optimization-based steady-state diffusion 
    solver to obtain $\boldsymbol{c}_G$: 
    \begin{align*}
      \mathop{\mbox{minimize}}_{\boldsymbol{c}_G \in \mathbb{R}^{ndofs}} & 
      \quad \frac{1}{2} \langle \boldsymbol{c}_G;\boldsymbol{K} 
      \boldsymbol{c}_G \rangle - \langle \boldsymbol{c}_G; \boldsymbol{f}_G 
      \rangle \\
      \mbox{subject to} & \quad c_G^{\mathrm{min}} \boldsymbol{1} \preceq 
      \boldsymbol{c}_G \preceq c_G^{\mathrm{max}} \boldsymbol{1} 
    \end{align*}
    %
    \FOR{$i = 1, 2, \cdots$, Number of nodes in the mesh}
    %
    \IF {$ - \epsilon_{\mathrm{mach}} \leq \boldsymbol{c}_F(i) - 
      \left(\frac{n_A}{n_B}\right) \boldsymbol{c}_G(i) \leq + 
      \epsilon_{\mathrm{mach}}$} 
    \STATE $\boldsymbol{c}_A(i) = 0, \; \boldsymbol{c}_B(i) = 0, \; 
    \boldsymbol{c}_C(i) = \left(\frac{n_C}{n_A}\right) \boldsymbol{c}_F(i)$ 
    \ENDIF
    %
    \IF {$\boldsymbol{c}_F(i) - \left(\frac{n_A}{n_B}\right) \boldsymbol{c}_G(i) 
      < - \epsilon_{\mathrm{mach}}$}
    \STATE $\boldsymbol{c}_A(i) = 0, \; \boldsymbol{c}_B(i) = - \left(\frac{n_B}
           {n_A}\right) \boldsymbol{c}_F(i) + \boldsymbol{c}_G(i), \; 
           \boldsymbol{c}_C(i) = \left(\frac{n_C}{n_A}\right) \boldsymbol{c}_F(i)$ 
           \ENDIF
           %
           \IF {$\boldsymbol{c}_F(i) - \left(\frac{n_A}{n_B}\right) \boldsymbol{c}_G(i) 
             > +\epsilon_{\mathrm{mach}}$}
           \STATE $\boldsymbol{c}_A(i) = \boldsymbol{c}_F(i) - \left(\frac{n_A}{n_B}\right) 
           \boldsymbol{c}_G(i), \; \boldsymbol{c}_B(i) = 0, \; \boldsymbol{c}_C(i) = 
           \left(\frac{n_C}{n_B}\right) \boldsymbol{c}_G(i)$ 
           \ENDIF 
           \ENDFOR  
  \end{algorithmic}
\end{algorithm}

\begin{remark}
  A naive implementation of the numerical strategy proposed in 
  subsection \ref{SubSec:Specializing_fast_bimolecular_reactions} 
  will lead to propagation of numerical noise from the invariant 
  set to reactants and product. So care should be exercised while 
  implementing the proposed numerical strategy to find the 
  concentration of $A$, $B$, and $C$. As outlined in Algorithm 
  \ref{Algo:SteadyState_Bimolecular}, first the numerical noise 
  from invariants $F$ and $G$ should be removed so that product 
  $C$ does not get affected adversely. 
\end{remark}

\subsubsection{Transient analysis:}
We will first perform temporal discretization using the method 
of horizontal lines and then followed by spatial discretization. 
To this end, the time interval $[0, \mathcal{I}]$ is discretized 
into $N$ non-overlapping sub-intervals: 
\begin{align}
  [0, \mathcal{I}] = \bigcup_{n = 1}^{N} [t_{n-1}, t_{n}]
\end{align}
where $t_0 = 0$ and $t_N = \mathcal{I}$. For simplicity, 
we shall assume uniform time step, which will be denoted 
by $\Delta t$. That is,
\begin{align}
  \Delta t := t_{n+1} - t_n \quad \forall n
\end{align}
However, it should be noted that a straightforward modification 
can handle non-uniform time steps. The following notation is 
employed to obtain time discretized version of quantities 
$c_F(\mathbf{x},t)$ and $\frac{\partial c_F(\mathbf{x},t)}
{\partial t}$ at $t = t_n$:
\begin{subequations}
  \begin{align}
    c^{(n)}_F(\mathbf{x}) &:= c_F(\mathbf{x},t_n) \\
    v^{(n)}_F(\mathbf{x}) &:= \frac{\partial c_F(\mathbf{x},t_n)}{\partial t} 
  \end{align}
\end{subequations}
  
Using the method of horizontal lines along with the backward 
Euler time stepping scheme, the governing equations 
\eqref{Eqn:Diffusion_for_F}--\eqref{Eqn:Diffusion_for_IC_F} 
at the time level $t_{n + 1}$ can be rewritten as the following 
anisotropic diffusion with decay:
\begin{subequations}
  \begin{align}
    \label{Eqn:Decay_Diffusion_for_F}
    &\left(\frac{1}{\Delta t} \right) c^{(n + 1)}_F(\mathbf{x})  
    - \mathrm{div} \left[\mathbf{D}(\mathbf{x}) \, \mathrm{grad}
    \left[c^{(n + 1)}_F(\mathbf{x}) \right] \right] = 
    f_F(\mathbf{x},t_{n + 1}) + \left( \frac{1}{\Delta t} \right)
    c^{(n)}_F(\mathbf{x}) \quad \mathrm{in} \; \Omega \\
    \label{Eqn:Decay_Diffusion_for_Dirchlet_F}
    &c^{(n + 1)}_F(\mathbf{x}) = c_F^{\mathrm{p}}(\mathbf{x},
    t_{n + 1}) := c^{\mathrm{p}}_A(\mathbf{x},t_{n + 1}) + 
    \left( \frac{n_A}{n_C} \right) c^{\mathrm{p}}_C(\mathbf{x},t_{n + 1}) 
    \quad \mathrm{on} \; \Gamma^{\mathrm{D}} \\
    \label{Eqn:Decay_Diffusion_for_Neumann_F}
    &\mathbf{n}(\mathbf{x}) \cdot \mathbf{D} (\mathbf{x}) \, 
    \mathrm{grad} \left[c^{(n + 1)}_F(\mathbf{x}) \right] =  
    h^{\mathrm{p}}_F(\mathbf{x},t_{n + 1}) := 
    h^{\mathrm{p}}_A(\mathbf{x},t_{n + 1}) + \left( \frac{n_A}{n_C} 
    \right) h^{\mathrm{p}}_C(\mathbf{x},t_{n + 1}) \quad \mathrm{on} 
    \; \Gamma^{\mathrm{N}} \\
    \label{Eqn:Decay_Diffusion_for_IC_F}
    &c_F(\mathbf{x},t_0) = c^{0}_F(\mathbf{x}) := 
    c^{0}_A(\mathbf{x}) + \left( \frac{n_A}{n_C} \right) 
    c^{0}_C(\mathbf{x}) \quad \mathrm{in} \; \Omega
  \end{align}
\end{subequations} 
where $f_F(\mathbf{x},t_{n + 1}) := f_A(\mathbf{x},t_{n + 1}) 
+ \left(\frac{n_A}{n_C} \right) f_C(\mathbf{x},t_{n + 1})$.

\begin{remark}
  From Section \ref{Sec:S2_Bimolecular_GE}, we have assumed that the 
  volumetric sources, initial and boundary conditions for $A$, $B$, 
  and $C$ are non-negative. Correspondingly the the volumetric source 
  $f_F(\mathbf{x},t_{n + 1})$, initial condition $c^{0}_F(\mathbf{x})$,
  the Dirichlet boundary condition $c_F^{\mathrm{p}}(\mathbf{x},t_{n + 1})$, 
  and the Neumann boundary condition $h^{\mathrm{p}}_F(\mathbf{x},t_{n + 1})$ 
  are non-negative. As the decay coefficient $\left(\frac{1}{\Delta t}
  \right)$ in front of $c^{(n + 1)}_F(\mathbf{x})$ is always greater than 
  zero and the diffusivity tensor $\mathbf{D}(\mathbf{x})$ is assumed 
  to be symmetric, bounded, and uniformly elliptic, a maximum principle 
  under these regularity assumptions for this type of anisotropic 
  diffusion with decay exists \cite{Nagarajan_Nakshatrala_IJNMF_2011_v67_p820}. 
\end{remark}

The standard single-field formulation for equations 
\eqref{Eqn:Decay_Diffusion_for_F}--\eqref{Eqn:Decay_Diffusion_for_IC_F} 
is given as follows: Find $c^{(n + 1)}_F(\mathbf{x}) \in \mathcal{P}^t_F$ 
such that we have
\begin{align}
  \label{Eqn:Transient_single_field_formulation_diffusion}
  \mathcal{B}^{t} \left(q;c^{(n + 1)}_F \right) = L^{t}_{F}(q) 
  \quad \forall q(\mathbf{x}) \in \mathcal{Q}
\end{align}
where the bilinear form and linear functional are, respectively, 
defined as 
\begin{subequations}
  \begin{align}
    \mathcal{B}^{t} \left(q;c^{(n + 1)}_F \right) &:= 
    \frac{1}{\Delta t} \int_{\Omega} q(\mathbf{x}) \; c^{(n + 1)}_F(\mathbf{x}) \; 
    \mathrm{d} \Omega + \int_{\Omega} \mathrm{grad}[q(\mathbf{x})] 
    \cdot \mathbf{D}(\mathbf{x}) \; \mathrm{grad}[c^{(n + 1)}_F(\mathbf{x})] 
    \; \mathrm{d} \Omega \\
    L^{t}_{F}(q) &:= \int_{\Omega} q(\mathbf{x}) \; 
    \left(f_F(\mathbf{x},t_{n + 1}) + \frac{1}{\Delta t} 
    c^{(n)}_F(\mathbf{x})\right) \; \mathrm{d} \Omega + 
    \int_{\Gamma^{\mathrm{N}}} q(\mathbf{x}) \; 
    h^{\mathrm{p}}_F(\mathbf{x},t_{n + 1}) \; \mathrm{d} \Gamma 
  \end{align}
\end{subequations}
The function space $\mathcal{P}^t_F$ is defined as follows (for more 
rigorous definition see \cite[Section 7.1]{Evans_PDE}):
\begin{align}
  \mathcal{P}^t_F &:= \left\{c_F(\cdot ,t) \in H^{1}(\Omega) \; \big| \; 
  c_F(\mathbf{x},t) = c^{\mathrm{p}}_F(\mathbf{x},t) \; \mathrm{on} \; 
  \Gamma^{\mathrm{D}}\right\}
\end{align}
but the function space $\mathcal{Q}$ defined by equation 
\eqref{Eqn:Function_Space_TestFunction} remains the same. 
In a similar manner, one can obtain a numerical methodology 
for transient solutions for the invariant $G$.
 
An equivalent minimization formulation for the weak form 
\eqref{Eqn:Transient_single_field_formulation_diffusion} 
can be written as follows:
\begin{subequations} 
\begin{align}
  \label{Eqn:Minimzation_Problem_Statement_DiffusionDecay}
  \mathop{\mathrm{minimize}}_{c^{(n + 1)}_F(\mathbf{x}) \in 
  \mathcal{P}^t_F} & \quad \frac{1}{2} \mathcal{B}^{t} 
  \left(c^{(n + 1)}_F; c^{(n + 1)}_F \right) - L^t_{F} 
  \left(c^{(n + 1)}_F \right) \\
  \label{Eqn:Minimzation_Problem_Constraint_DiffusionDecay}
  \mbox{subject to} & \quad c_F^{\mathrm{min}} \leq 
  c^{(n + 1)}_F(\mathbf{x}) \leq c_F^{\mathrm{max}}
\end{align}
\end{subequations}
After finite element discretization by means of low-order finite elements, 
a non-negative formulation based on the minimization problem can be devised 
as follows:
\begin{subequations}
  \label{Eqn:NonNegative_Solver_Invariant_F_DiffusionDecay}
  \begin{align}
    \mathop{\mbox{minimize}}_{\boldsymbol{c}^{(n + 1)}_F \in 
    \mathbb{R}^{ndofs}} & \quad \frac{1}{2}  \left \langle 
    \boldsymbol{c}^{(n + 1)}_F; \widetilde{\boldsymbol{K}} 
    \boldsymbol{c}^{(n + 1)}_F \right \rangle - \left \langle 
    \boldsymbol{c}^{(n + 1)}_F; \boldsymbol{f}^{(n + 1)}_F  
    \right \rangle - \frac{1}{\Delta t}  \left \langle 
    \boldsymbol{c}^{(n + 1)}_F; \boldsymbol{c}^{(n)}_F
    \right \rangle \\
    \mbox{subject to} & \quad c_F^{\mathrm{min}} \boldsymbol{1} 
    \preceq \boldsymbol{c}^{(n + 1)}_F \preceq c_F^{\mathrm{max}}
    \boldsymbol{1} 
  \end{align}
\end{subequations}
where $\boldsymbol{c}^{(n + 1)}_F$ and $\boldsymbol{f}^{(n + 1)}_F$ are 
the nodal concentration and volumetric source vectors for invariant $F$ at 
time level $t_{n + 1}$. Correspondingly, the nodal concentration vector 
$\boldsymbol{c}^{(n)}_F$ obtained at time level $t_{n}$ also acts as a 
volumetric source similar to $\boldsymbol{f}^{(n + 1)}_F$. As we have 
employed low-order finite elements for spatial discretization, we have 
$\boldsymbol{c}^{(0)}_F \succeq \boldsymbol{0}$ and $\boldsymbol{f}^{(n + 1)}_F 
\succeq \boldsymbol{0}$. Hence at each time level, the net volumetric source 
vector $\boldsymbol{f}^{(n + 1)}_F + \frac{1}{\Delta t} \boldsymbol{c}^{(n)}_F 
\succeq \boldsymbol{0}$. 
Note that the symmetric positive definite matrix $\widetilde{\boldsymbol{K}}$ 
is different to that of $\boldsymbol{K}$, which is the coefficient matrix in 
steady-state analysis. Since the coefficient matrix $\widetilde{\boldsymbol{K}}$ 
is positive definite, the discrete minimization problem 
\eqref{Eqn:NonNegative_Solver_Invariant_F_DiffusionDecay} 
belongs to \emph{convex quadratic programming}, and from 
the optimization theory a unique global minimizer exists. 
One could write the corresponding first-order optimality 
conditions corresponding to the optimization problem 
\eqref{Eqn:NonNegative_Solver_Invariant_F_DiffusionDecay}, 
which will be similar to equations 
\eqref{Eqn:Bimolecular_FOOC_1}--\eqref{Eqn:Bimolecular_FOOC_6}.
Algorithm \ref{Algo:Transient_Bimolecular} outlines the 
steps involved in the implementation of the proposed 
methodology for transient diffusive-reactive systems. 

\begin{algorithm}
  \caption{Numerical framework for transient fast bimolecular diffusive-reactive systems}
  \label{Algo:Transient_Bimolecular}
  \begin{algorithmic}[1]
    \STATE Input: Time step $\Delta t$; total time of interest $\mathcal{I}$; 
    stoichiometric coefficients; volumetric sources; initial and boundary 
    conditions for the chemical species $A$, $B$, and $C$
    \STATE Calculate corresponding volumetric sources, initial 
    and boundary conditions for the invariants, and form the 
    vectors $\boldsymbol{c}_F^{(0)}$ and $\boldsymbol{c}_G^{(0)}$ 
    \STATE Calculate $c_F^{\mathrm{min}}$, $c_F^{\mathrm{max}}$, 
    $c_G^{\mathrm{min}}$ and $c_G^{\mathrm{max}}$ based on 
    maximum principles and the non-negative constraint
    %
    \FOR{$n = 0, 1, \cdots, N - 1$}
    %
    \STATE Call optimization-based steady-state diffusion 
    with decay solver to obtain $\boldsymbol{c}_F^{(n+1)}$:    
    \begin{align*}
      \mathop{\mbox{minimize}}_{\boldsymbol{c}^{(n + 1)}_F \in 
        \mathbb{R}^{ndofs}} & \quad \frac{1}{2}  \left \langle 
      \boldsymbol{c}^{(n + 1)}_F; \widetilde{\boldsymbol{K}} 
      \boldsymbol{c}^{(n + 1)}_F \right \rangle - \left \langle 
      \boldsymbol{c}^{(n + 1)}_F; \boldsymbol{f}^{(n + 1)}_F  
      \right \rangle - \frac{1}{\Delta t}  \left \langle 
      \boldsymbol{c}^{(n + 1)}_F; \boldsymbol{c}^{(n)}_F
      \right \rangle \\
      \mbox{subject to} & \quad c_F^{\mathrm{min}} \boldsymbol{1} \preceq 
      \boldsymbol{c}^{(n + 1)}_F \preceq c_F^{\mathrm{max}} \boldsymbol{1} 
    \end{align*}
    \STATE Call optimization-based steady-state diffusion 
    with decay solver to obtain $\boldsymbol{c}_G^{(n+1)}$:    
    \begin{align*}
      \mathop{\mbox{minimize}}_{\boldsymbol{c}^{(n + 1)}_G \in 
        \mathbb{R}^{ndofs}} & \quad \frac{1}{2}  \left \langle 
      \boldsymbol{c}^{(n + 1)}_G; \widetilde{\boldsymbol{K}} 
      \boldsymbol{c}^{(n + 1)}_G \right \rangle - \left \langle 
      \boldsymbol{c}^{(n + 1)}_G; \boldsymbol{f}^{(n + 1)}_G  
      \right \rangle - \frac{1}{\Delta t}  \left \langle 
      \boldsymbol{c}^{(n + 1)}_G; \boldsymbol{c}^{(n)}_G
      \right \rangle \\
      \mbox{subject to} & \quad c_G^{\mathrm{min}} \boldsymbol{1} \preceq 
      \boldsymbol{c}^{(n + 1)}_G \preceq c_G^{\mathrm{max}} \boldsymbol{1} 
    \end{align*}
    %
    \FOR{$i = 1, 2, \cdots$, Number of nodes in the mesh}
    %
    \IF {$ - \epsilon_{\mathrm{mach}} \leq \boldsymbol{c}^{(n+1)}_F(i) - \left(\frac{n_A}{n_B}\right) 
      \boldsymbol{c}^{(n+1)}_G(i) \leq + \epsilon_{\mathrm{mach}}$}
    \STATE $\boldsymbol{c}_A^{(n+1)}(i) = 0, \; \boldsymbol{c}_B^{(n+1)}(i) = 0, \; 
    \boldsymbol{c}_C^{(n+1)}(i) = \left(\frac{n_C}{n_A}\right) \boldsymbol{c}^{(n+1)}_F(i)$
    \ENDIF
    %
    \IF {$\boldsymbol{c}^{(n+1)}_F(i) - \left(\frac{n_A}{n_B}\right) \boldsymbol{c}^{(n+1)}_G(i) 
      < - \epsilon_{\mathrm{mach}}$}
    \STATE $\boldsymbol{c}_A^{(n+1)}(i) = 0, \; \boldsymbol{c}_B^{(n+1)}(i) = - \left(\frac{n_B}
           {n_A}\right) \boldsymbol{c}^{(n+1)}_F(i) + \boldsymbol{c}^{(n+1)}_G(i), \; 
           \boldsymbol{c}_C^{(n+1)}(i) = \left(\frac{n_C}{n_A}\right) \boldsymbol{c}^{(n+1)}_F(i)$ 
           \ENDIF
           %
           \IF {$\boldsymbol{c}^{(n+1)}_F(i) - \left(\frac{n_A}{n_B}\right) 
             \boldsymbol{c}^{(n+1)}_G(i) > +\epsilon_{\mathrm{mach}}$}
           \STATE $\boldsymbol{c}_A^{(n+1)}(i) = \boldsymbol{c}^{(n+1)}_F(i) 
           - \left(\frac{n_A}{n_B}\right) \boldsymbol{c}^{(n+1)}_G(i), \; 
           \boldsymbol{c}_B^{(n+1)}(i) = 0, \; \boldsymbol{c}_C^{(n+1)}(i) = 
           \left(\frac{n_C}{n_B}\right) \boldsymbol{c}^{(n+1)}_G(i)$  
           \ENDIF 
           \ENDFOR   
    \ENDFOR
  \end{algorithmic}
\end{algorithm}

\section{PROPERTIES OF THE PROPOSED COMPUTATIONAL 
  FRAMEWORK: A THEORETICAL STUDY}
\label{Sec:S4_NN_Theoretical} 
In a continuous setting, it is well-known from the theory 
of partial differential equations that the comparison 
principle, the maximum principle and the non-negative 
principle satisfy the following commutative diagram:

\begin{tikzpicture}[description/.style={fill=white,inner sep=2pt}]
\matrix (m) [matrix of math nodes, row sep=3em,
column sep=2.5em, text height=1.5ex, text depth=0.25ex]
{\mbox{uniqueness of the solution} & \mbox{comparison principle} 
& \mbox{maximum principle} \\
& \mbox{non-negative constraint} & & \\ };
\path[->,font=\scriptsize]
(m-1-2) edge node[thick,auto] {} (m-1-1) 
(m-1-2) edge node[thick,auto] {} (m-1-3) 
(m-2-2) edge node[thick,auto] {} (m-1-2)
(m-1-3) edge node[thick,auto] {} (m-2-2);
\end{tikzpicture}

\noindent
For linear problems, one can \emph{also directly} prove 
the uniqueness of the solution using the maximum principle 
(instead of the comparison principle). However, for semilinear 
and quasilinear elliptic and parabolic partial differential 
equations, the comparison principle is more convenient to 
prove directly the uniqueness of the solution. As mentioned 
earlier, the diffusive-reactive system given by equation 
\eqref{Eqn:Mixture_for_A_B_C} is a semilinear partial 
differential equation. This is the reason why we have 
indicated in the above commutative diagram that the 
uniqueness of the solution is a direct consequence 
of the comparison principle. 

It is natural to ask which of these properties are inherited 
in discrete setting by the optimization-based solver for 
solving diffusion-type equations. Prior studies on the 
optimization-based methodologies (e.g., references 
\cite{Liska_Shashkov_CiCP_2008_v3_p852,
Nakshatrala_Valocchi_JCP_2009_v228_p6726,
Nagarajan_Nakshatrala_IJNMF_2011_v67_p820,
Nakshatrala_Nagarajan_arXiv_2012}) did not address whether 
such a commutative diagram is satisfied in the discrete 
setting. In particular, these studies did not discuss 
the discrete version of the comparison principle. 
The standard comparison principle takes the following 
form:
\emph{
  Let $\mathcal{L}[u] := \alpha(\mathbf{x}) u(\mathbf{x}) - 
  \mathrm{div}[\mathbf{D}(\mathbf{x})\mathrm{grad}[u]]$. 
  Let $u_1(\mathbf{x}), \; u_2(\mathbf{x}) \in C^{2}(\Omega) 
  \cap C^{0}(\overline{\Omega})$. If $\mathcal{L}[u_1] 
  \geq \mathcal{L}[u_2] \; \mathrm{in} \; \Omega$ and 
  $u_1(\mathbf{x}) \geq u_2(\mathbf{x}) \; on \; \partial 
  \Omega$ then $u_1(\mathbf{x}) \geq u_2(\mathbf{x}) \; 
  in \; \overline{\Omega}$.}
We shall say that a numerical formulation inherits the comparison 
principle (or possess a discrete comparison principle) if for any 
two finite dimensional vectors satisfying $\boldsymbol{f}_1 \succeq 
\boldsymbol{f}_2$ implies the finite dimensional solutions under 
the numerical formulation satisfy $\boldsymbol{c}_1 \succeq 
\boldsymbol{c}_2$.

In discrete setting, the optimization-based methodology 
has the following properties:
\begin{enumerate}[(i)]
\item The methodology is based on the finite element method.
\item The methodology satisfies both the non-negative principle 
  and the maximum principle on general computational grids and 
  with no additional restrictions on the time step.
\item The methodology does \emph{not} satisfy the comparison principle. 
  A counterexample is provided in Figure \ref{Fig:NN_violation_of_DCP}. 
  However, if $\boldsymbol{f}_1 = \eta \boldsymbol{f}_2$ with $\eta 
  \geq 0$, then $\boldsymbol{c}_1 = \eta \boldsymbol{c}_2$ under the 
  optimization-based solver. 
\item Although a discrete version of the comparison principle 
  does not hold, a unique solution always exists under the 
  optimization-based solver. The existence and the uniqueness 
  of the solution in the discrete setting stems from the 
  results in the convex programming with bounds on the 
  variables. 
  Of course, the proofs from the optimization theory do not 
  require a discrete version of the comparison principle 
  to hold but use alternate approaches from Matrix Algebra 
  and Analysis \cite{Boyd_convex_optimization}. 
\item The solution procedure is nonlinear, as a quadratic 
  programming optimization problem with inequalities is 
  nonlinear.
\item One can employ the consistent capacity matrix that 
  stems from the variational structure of the underling 
  weak formulation. If needed, one can also employ a lumped 
  capacity matrix, which is considered to be a variational 
  crime. 
\end{enumerate}
In conclusion, the proposed optimization methodology does not possess a 
commutative diagram similar to the one given above. But it should be 
noted that the optimization-based method is the only known methodology 
that satisfies the non-negative principle and the maximum principle on 
general computational grids with no additional restrictions on the time 
step. Some other noteworthy efforts on non-negative solvers for 
diffusion-type equations are by Le Potier 
\cite{LePotier_CRM_2005_v341_p787} and Lipnikov \emph{et al.} 
\cite{Lipnikov_Shashkov_Svyatskiy_Vassilevski_JCP_2007_v227_p492,
Lipnikov_Svyatskiy_Vassilevski_JCP_2009_v228_p703}, which have 
the following attributes: 
\begin{enumerate}[(i)]
\item The methodology is based on the finite volume method.
\item The methodology satisfies the non-negative principle 
  on computational grids with triangles. The method does 
  \emph{not} satisfy the maximum principle. 
\item The methodology does not satisfy the comparison principle. 
\item The solution procedure is nonlinear, as the collocation 
  points (i.e., the location of the unknown concentrations) 
  in turn depend on the concentration. 
\item One can employ only a lumped capacity matrix, which 
  is consistent with the finite volume method. Also, one 
  has to employ the backward Euler time stepping scheme, 
  as any other time stepping scheme can result in an 
  algorithmic failure or can produce negative concentrations.
\end{enumerate}
Some recent finite volume methods \cite{2009_LePotier_IJFV_v6,
2011_Droniou_LePotier_SIAMJNA_v49_p459_p490} and finite difference 
methods \cite{2011_Lipnikov_Manzini_Svyatskiy_JCP_v230_p2620_p2642} 
have succeeded in overcoming the aforementioned limitation (ii) 
(i.e., these methods satisfy maximum principles).

\section{NUMERICAL PERFORMANCE OF THE PROPOSED COMPUTATIONAL FRAMEWORK}
\label{Sec:S5_NN_NR}
In this section we shall illustrate the performance of the proposed 
computational framework for diffusive-reactive systems given by 
equations \eqref{Eqn:DRs_for_A}--\eqref{Eqn:DRs_for_IC}. We shall 
also perform numerical $h$-convergence using hierarchical meshes. 
We shall solve three representative problems: plume development 
from a boundary, plume formation due to continuous point source 
emitters, and decay of a chemically reacting slug. These types 
of problems have many practical applications. For example, these 
problems naturally arise from situations in which one may want to 
regulate/control an induced contaminant in an ecological system, 
which is often achieved by introducing a suitable chemical or 
biological species that reacts with the contaminant to form a 
harmless product. 
Some specific examples include chemical dispersion methods to 
control oil spills and airborne contaminants, and usage of 
bioremediators for regulating industrial emissions into 
water bodies. 

In all the numerical simulations reported in this paper, 
the resulting convex quadratic programming problems are 
solved using the MATLAB's \cite{MATLAB_2012a} built-in 
function handler \textsf{quadprog}, which offer robust 
solvers for solving quadratic programming problems. In 
particular, we have used the specific algorithm that can 
be invoked using the \textsf{interior-point-convex} option 
under \textsf{quadprog}, which is based on the numerical 
methods presented in references  
\cite{2004_Gould_Toint_MPSB_v100_p95_p132,
1992_Mehrotra_SIAMJO_v2_p575_p601,
1996_Gondzio_COA_v6_p137_p156}. For further details, 
see MATLAB's documentation \cite{MATLAB_2012a}.
The tolerance in the stopping criterion in solving 
convex quadratic programming is taken as $100 
\epsilon_{\mathrm{mach}}$, where $\epsilon_{\mathrm{mach}} 
\approx 2.22 \times 10^{-16}$ is the machine precision 
for a 64-bit machine.

\subsection{Numerical $h$-convergence study}
We shall use the method of manufactured solutions to perform 
numerical $h$-convergence. We shall restrict the current study 
to steady-state response. The computational domain is a rectangle 
of size $L_x \times L_y$. The diffusivity tensor for this numerical 
study is taken as follows:
\begin{align}
  \mathbf{D}(\mathbf{x}) = \mathbf{R} 
  \mathbf{D}_{\mathrm{diagonal}} \mathbf{R}^{\mathrm{T}}
\end{align}
where 
\begin{align}
  \label{Eqn:NN_rotation_tensor}
  &\mathbf{R} = \left(\begin{array}{cc}
    \cos(\theta) & -\sin(\theta) \\
    \sin(\theta) & \cos(\theta) \\
  \end{array}\right) \\
  &\mathbf{D}_{\mathrm{diagonal}} = \left(\begin{array}{cc}
    d_1 & 0 \\
    0 & d_2 \\
  \end{array}\right) 
\end{align}
The assumed analytical solutions for $c_F(\mathbf{x})$ and 
$c_G(\mathbf{x})$ are taken as follows:
\begin{subequations}
  \begin{align}
    \label{Eqn:Analytical_Invariant_F}
    c_F(\mathbf{x}) &= \sin\left(\frac{\pi \mathrm{x}}{2 L_x}\right)
    \sin\left(\frac{\pi \mathrm{y}}{2 L_y}\right) \\
    \label{Eqn:Analytical_Invariant_G}
    c_G(\mathbf{x}) &= \cos\left(\frac{\pi \mathrm{x}}{2 L_x}\right)
    \cos\left(\frac{\pi \mathrm{y}}{2 L_y}\right) 
  \end{align}
\end{subequations}
where $0 \leq \mathrm{x} \leq L_x$ and $0 \leq \mathrm{y} 
\leq L_y$. Using the above expressions, the source terms, 
boundary conditions, and expressions for the concentrations 
of $A$, $B$, and $C$ are in turn calculated. The corresponding 
volumetric sources for invariants $F$ and $G$ are given as 
follows:
\begin{subequations}
  \begin{align}
    \label{Eqn:Volumetric_Source_Invariant_F}
    f_F(\mathrm{x}, \mathrm{y}) &= \frac{\pi^2}{4} \sin
    \left(\frac{\pi \mathrm{x}}{2 L_x}\right) 
    \sin\left(\frac{\pi \mathrm{y}}{2 L_y}\right) 
    \left(d_1 \left(\frac{\cos^2(\theta)}{L_x^2} + 
    \frac{\sin^2(\theta)}{L_y^2}\right) + d_2 
    \left(\frac{\sin^2(\theta)}{L_x^2} + 
    \frac{\cos^2(\theta)}{L_y^2} \right)\right) \nonumber \\
    &- \frac{\pi^2}{2 L_x L_y} \left(d_1 - d_2 \right) \sin(\theta) 
    \cos(\theta) \cos\left(\frac{\pi \mathrm{x}}{2 L_x}\right) \cos
    \left(\frac{\pi \mathrm{y}}{2 L_y}\right) \\
    \label{Eqn:Volumetric_Source_Invariant_G}
    f_G(\mathrm{x}, \mathrm{y}) &= \frac{\pi^2}{4} \cos
    \left(\frac{\pi \mathrm{x}}{2 L_x}\right) 
    \cos\left(\frac{\pi \mathrm{y}}{2 L_y}\right) 
    \left(d_1 \left(\frac{\cos^2(\theta)}{L_x^2} + 
    \frac{\sin^2(\theta)}{L_y^2}\right) + d_2 
    \left(\frac{\sin^2(\theta)}{L_x^2} + 
    \frac{\cos^2(\theta)}{L_y^2} \right)\right) \nonumber \\
    &- \frac{\pi^2}{2 L_x L_y} \left(d_1 - d_2 \right) \sin(\theta) 
    \cos(\theta) \sin\left(\frac{\pi \mathrm{x}}{2 L_x}\right) \sin
    \left(\frac{\pi \mathrm{y}}{2 L_y}\right)
  \end{align}
\end{subequations}
A pictorial description of the boundary value problem is given 
in Figure \ref{Fig:Convergence_Study_SteadyState}. In this paper, 
we have taken the following values for the parameters to perform 
the steady-state numerical $h$-convergence: 
\begin{align}
  L_x = 2, \; L_y = 1, \; \theta = \pi/3, \; d_1 = 
  1000, \; d_2 = 1, \; n_A = 2, \; n_B = 3, \; 
  n_C = 1
\end{align} 
Figure \ref{Fig:Hierarchical_Meshes_Convergence} shows the 
typical computational meshes employed in this study. Figure 
\ref{Fig:Bimolecular_h_convergence_of_F_G} illustrates the 
convergence of the concentration of the invariants, and 
the convergence for the concentrations of the chemical 
species $A$, $B$, and $C$ is shown in Figure 
\ref{Fig:Bimolecular_h_convergence_of_A_B_C}. For invariants, 
we have shown the convergence in both $L_2$ norm and $H^{1}$ 
seminorm. However, for the chemical species $A$, $B$, and 
$C$, we have shown the convergence only in $L_2$ norm as 
$\mathrm{max}[\cdot]$ operator is non-smooth. 


\subsection{Plume development from boundary in a reaction tank}
The test problem is pictorially described in Figure \ref{Fig:Reaction_tank}. 
We restrict the present study to steady-state analysis. The stoichiometric 
coefficients are taken as $n_A = 1$, $n_B = 1$ and $n_C = 2$. The diffusivity 
tensor is taken from the subsurface literature \cite{Pinder_Celia}, and can 
be written as follows:
\begin{align}
  \label{Eqn:NN_subsurface_D}
  \mathbf{D}_{\mathrm{subsurface}}(\mathbf{x}) = 
    \alpha_{T} \|\mathbf{v}\| \mathbf{I} + 
  \frac{\alpha_L - \alpha_T}{\|\mathbf{v}\|} \mathbf{v} \otimes 
  \mathbf{v}
\end{align}
where $\otimes$ is the tensor product, $\mathbf{v}$ is velocity 
vector field, and $\alpha_T$ and $\alpha_L$ are, respectively, 
transverse and longitudinal diffusivity coefficients. It should be emphasized that we have neglected the advection, and 
the velocity is used to define the diffusion tensor.  
The velocity field is defined through a multi-mode stream 
function, which takes the following form: 
\begin{align}
  \label{Eqn:NN_Plume_subsurface_stream_function}
  \psi(\mathrm{x},\mathrm{y}) = -\mathrm{y} - \sum_{k=1}^{3} A_k \cos 
  \left(\frac{{p}_k \pi \mathrm{x}}{L_x} - \frac{\pi}{2}\right) \sin 
  \left( \frac{q_k \pi \mathrm{y}}{L_y}\right)
\end{align}
The corresponding components of the velocity are given by 
\begin{subequations}
  \label{Eqn:NN_Plume_subsurface_velocity_vector_field}
  \begin{align}
    \mathrm{v}_{\mathrm{x}}(\mathrm{x},\mathrm{y}) &= 
    -\frac{\partial \psi}{\partial \mathrm{y}} = 1 
    + \sum_{k=1}^{3} A_k \frac{q_k \pi}{L_y} \cos 
    \left(\frac{p_k \pi \mathrm{x}}{L_x} - \frac{\pi}{2}
    \right) \cos\left(\frac{q_k \pi \mathrm{y}}{L_y}\right) \\
    \mathrm{v}_{\mathrm{y}}(\mathrm{x},\mathrm{y}) &= 
    +\frac{\partial \psi}{\partial \mathrm{x}} = \sum_{k=1}^{3} 
    A_k \frac{p_k \pi}{L_x} \sin\left(\frac{p_k \pi \mathrm{x}}{L_x} 
    - \frac{\pi}{2}\right) \sin\left(\frac{q_k \pi \mathrm{y}}{L_y}\right)
  \end{align}
\end{subequations}
It is noteworthy that the above velocity vector field is solenoidal 
(i.e., $\mathrm{div}[\mathbf{v}] = 0$). Although realistic aquifers 
exhibit spatial variability on a hierarchy of scales, periodic or 
quasi-periodic models similar to the one outlined above (given by 
equations 
\eqref{Eqn:NN_subsurface_D}--\eqref{Eqn:NN_Plume_subsurface_velocity_vector_field}) 
have often been used due to their simplicity 
\cite{Kapoor_Kitanidis_TPM_1996_v22_p91,
Chrysikopoulos_Kitanidis_Roberts_WRR_1992_v28_p1517}. 

In this paper we shall take $(L_x,L_y) = (2, 1)$, 
$(p_1, p_2, p_3) = (4, 5, 10)$, $(q_1, q_2, q_3) 
= (1, 5, 10)$, $(A_1, A_2, A_3) = (0.08, 0.02, 
0.01)$, and $(\alpha_L, \alpha_T) = (1, 0.0001)$. 
The contours of the stream function, and the 
corresponding velocity vector field are shown 
in Figure \ref{Fig:NN_multi_mode_stream_function}. 
The computational domain is meshed using four-node quadrilateral 
elements. The numerical simulation is carried out on various 
structured meshes with XSeed and YSeed ranging from 97 to 259, 
where XSeed and YSeed are, respectively, the number of nodes 
along the x-direction and y-direction. That is, a structured mesh 
with XSeed = YSeed = 259 has over 67,000 unknown nodal 
concentrations.

The values for prescribed concentrations on the boundary are 
taken as $c_A^{\mathrm{p}} = 1$ and $c_B^{\mathrm{p}} = 10$. Figures 
\ref{Fig:NN_Plume_subsurface_F_Q4_contours}--\ref{Fig:NN_Plume_subsurface_C_Q4_contours} 
compare the contours of the concentrations of the invariants, 
reactants and product under the Galerkin formulation, the 
clipping procedure and the proposed non-negative formulation. 
Figure \ref{Fig:NN_Plume_subsurface_FGC_Q4_contours} shows 
the regions in the domain that have violated the lower and 
upper bounds on the concentrations of the invariants and 
the product under the Galerkin formulation. To provide 
more quantitative inference on the violations, the variation 
of the concentration of the product along two sections given 
by $\mathrm{y} = 0.45$ and $\mathrm{x} = 1.75$ is plotted 
in Figure \ref{Fig:NN_Plume_subsurface_cC_sectional_views}.

Figures \ref{Fig:NN_Plume_subsurface_integral_FG_over_y} and 
\ref{Fig:NN_Plume_subsurface_integral_C_over_y} show the variation 
of the total concentration, $\int c(\mathrm{x},\mathrm{y}) \; 
\mathrm{d} \mathrm{y}$, along the x-direction for the invariants 
and the product. The total concentration for a given cross-section 
is a useful measure in the studies on reactive-transport. The results 
are shown for various mesh refinements, and the Galerkin formulation 
predicts unphysical negative values even for this important global 
measure. On the other hand, the proposed computational framework 
predicts physically meaningful for the considered measure. 
Figures \ref{Fig:NN_Plume_subsurface_F_Lagrange_contours} and 
\ref{Fig:NN_Plume_subsurface_G_Lagrange_contours} show the 
contours of the Lagrange multipliers in solving quadratic 
programming problems to obtain the invariants $F$ and $G$. 
The Lagrange multipliers arise due to the enforcement of 
the lower and upper bounds due to the non-negative constraint 
and the maximum principle in calculating the invariants. It 
should be noted that, based on the Karush-Kuhn-Tucker results 
in the theory of optimization, the Lagrange multipliers are 
non-negative \cite{Boyd_convex_optimization}. 

\subsection{Plume formation due to multiple stationary point sources}
The problem is pictorially described in Figure 
\ref{Fig:Plume_Multiple_PointSources}. The following 
rotated heterogeneous diffusivity tensor is employed:
\begin{align}
  \mathbf{D}(\mathbf{x}) = \mathbf{R} \mathbf{D}_{0}
  (\mathbf{x)} \mathbf{R}^{\mathrm{T}}
\end{align}
where $\mathbf{D}_0(\mathbf{x})$ is 
\begin{align}
  \mathbf{D}_0(\mathbf{x}) = 
  \left(\begin{array}{cc}
    \mathrm{y}^2 + \epsilon \mathrm{x}^2 & - (1 - \epsilon) 
    \mathrm{x} \mathrm{y} \\
    -(1 - \epsilon) \mathrm{x} \mathrm{y} & \epsilon 
    \mathrm{y}^2 + \mathrm{x}^2 \\
  \end{array}\right) 
\end{align}
and the rotation tensor is same as before (given by equation 
\eqref{Eqn:NN_rotation_tensor}) with angle $\theta = \pi/3$. 
For this problem, the parameter $\epsilon$ is taken as $0.001$. 
The stoichiometry coefficients are taken to be $n_A = 1$, $n_B 
= 1$ and $n_C = 2$. 
Figures \ref{Fig:NN_Plume_PointSources_contour_F}--\ref{Fig:NN_Plume_PointSources_contour_C} show the contours of the concentrations of the 
invariants and the product under the Galerkin formulation and 
the proposed computational framework. 
Figure \ref{Fig:NN_Plume_PointSources_integral_C_over_y} shows 
the variation of the integrated concentration of the product 
with respect to $\mathrm{y}$ (i.e., $\int c(\mathrm{x},\mathrm{y}) 
\; \mathrm{d} \mathrm{y}$) along $\mathrm{x}$, and the Galerkin 
formulation dramatically predicts negative values even for this 
average. Table \ref{Table:NN_PointSources_min_max} shows that 
small violations in the non-negative constraint for invariants 
can lead to bigger violations in the non-negative constraint 
for the product $C$.
Figure \ref{Fig:NN_Plume_PointSources_CPU_timing} compares 
the CPU time taken by the Galerkin formulation and the 
proposed computational framework. Even for a mesh with 
more than 40,000 nodes, the CPU time taken by the proposed 
framework is only 1.07 times the CPU time taken by the 
Galerkin formulation. For this mesh, 50.51\% of the nodes 
have violated the non-negative constraint for the product 
under the Galerkin formulation. On the other hand, the 
nodal concentrations of the product under the proposed 
computational framework are physically meaningful 
non-negative values. 

\begin{table}[ht]
  \centering
  \caption{Plume formulation due to multiple stationary point sources: 
    This table shows the minimum concentration, maximum concentration, 
    and \% of nodes violated the non-negative constraint under the 
    Galerkin formulation. \emph{The values clearly indicate that small 
    violations in meeting the non-negative constraint for invariants 
    can result in much bigger violation of the non-negative 
    constraint for the product. It is important to note that 
    Ciarlet and Raviart \cite{Ciarlet_Raviart_CMAME_1973_v2_p17} 
    have shown that the Galerkin single-field formulation, in 
    general, does not converge \emph{uniformly} for diffusion-type 
    equations. That is, the numerical solutions under the Galerkin 
    formulation for diffusion-type equations converge in $L_2$ 
    norm but need not converge in $L^{\infty}$ norm.} 
    \label{Table:NN_PointSources_min_max}}
  \begin{tabular}{|c|c|c|c|c|} \hline
    \multirow{2}{*}{Mesh} & \multicolumn{2}{|c|}{Invariant F} & 
    \multicolumn{2}{|c|}{Invariant G} \\ 
    \cline{2-3} \cline{4-5}
    & $\frac{\mbox{Min. conc.}}{\mbox{Max. conc.}} \times 100$ 
    & nodes violated & $\frac{\mbox{Min. conc.}}{\mbox{Max. conc.}} 
    \times 100$ & nodes violated  \\ \hline
    $21 \times 21$ & $\frac{-1.50\times10^{-3}}{1.42\times10^{0}} \times 100 = -0.11\%$ 
    & 8.62\% & $\frac{-2.07\times10^{-3}}{3.00\times10^{-1}} \times 100 = -0.69\%$ & 23.36\% \\
    $51 \times 51$ & $\frac{-1.38\times10^{-2}}{2.37\times10^{0}} \times 100 = -0.58\%$ 
    & 24.14\% & $\frac{-8.54\times10^{-3}}{4.99\times10^{-1}} \times 100 = -1.71\%$ & 43.71\% \\
    $101 \times 101$ & $\frac{-3.13\times10^{-2}}{3.38\times10^{0}} \times 100 = -0.93\%$ 
    & 25.30\%  & $\frac{-1.31\times10^{-2}}{7.11\times10^{-1}} \times 100 = -1.84\%$ & 46.53\%  \\
    $161 \times 161$ & $\frac{-4.45\times10^{-2}}{4.24\times10^{0}} \times 100 = -1.05\%$ 
    & 25.77\%  & $\frac{-1.52\times10^{-2}}{8.96\times10^{-1}} \times 100 = -1.70\%$ & 45.20\%  \\
    $201 \times 201$ & $\frac{-4.98\times10^{-2}}{4.71\times10^{0}} \times 100 = -1.06\%$ 
    & 27.29\%  & $\frac{-1.58\times10^{-2}}{9.98\times10^{-1}} \times 100 = -1.58\%$ & 40.83\%  \\ \hline
  \end{tabular}
  \begin{tabular}{|c|c|c|} \hline
    \multirow{2}{*}{Mesh} & \multicolumn{2}{|c|}{Product C} \\ 
    \cline{2-3}
    & $\frac{\mbox{Min. conc.}}{\mbox{Max. conc.}} \times 100$ 
    & nodes violated  \\ \hline
    $21 \times 21$ & $\frac{-4.13 \times 10^{-3}}{1.22 \times 10^{-1}} \times 100 = -3.38\%$ & 31.29\% \\ 
    $51 \times 51$ & $\frac{-2.75 \times 10^{-3}}{1.83 \times 10^{-1}} \times 100 = -15.05\%$ & 53.94\% \\ 
    $101 \times 101$ & $\frac{-6.25 \times 10^{-2}}{2.26 \times 10^{-1}} \times 100 = -27.70\%$ & 51.02\% \\ 
    $161 \times 161$ & $\frac{-8.90 \times 10^{-2}}{2.48 \times 10^{-1}} \times 100 = -35.93\%$ & 54.94\% \\ 
    $201 \times 201$ & $\frac{-9.96 \times 10^{-2}}{2.56 \times 10^{-1}} \times 100 = -38.90\%$ & 50.51\% \\ \hline
  \end{tabular}
\end{table}

\subsection{Diffusion and reaction of an initial slug}
The test problem is pictorially described in Figure \ref{Fig:Plume_Slug}. 
The initial conditions within the reaction slug are assumed to be 
$c_A^{0} = 10$, $c_B^{0} = 0$, and $c_C^{0} = 0$. In the other parts 
of the rectangular domain we have $c_A^{0} = c_B^{0} = c_C^{0} = 0$. 
The boundary conditions for the chemical species on the rectangular 
domain are assumed to be $c_A^{\mathrm{p}} = 0$, $c_B^{\mathrm{p}} = 
1 - \mathrm{e}^{-t}$, and $c_C^{\mathrm{p}} = 0$. The problem is solved 
within the time interval $t \in [0,1]$. The diffusivity tensor 
within the domain is taken as $\mathbf{D}(\mathbf{x}) = \mathbf{R} 
\mathbf{D}_{\mathrm{subsurface}} \mathbf{R}^{\mathrm{T}}$ with $\theta 
= \frac{\pi}{6}$. The stoichiometric coefficients are $n_A = 2$, $n_B 
= 2$, and $n_C = 1$. Numerical simulations are performed using 
structured mesh with $101$ nodes along each side of the domain 
using four-node quadrilateral elements. The volumetric source 
of each species $A$, $B$, and $C$ is equal to zero inside the 
rectangular domain. 
Two different time steps are employed in the numerical 
simulations: $\Delta t = 0.05 \; \mathrm{s}$ and $\Delta 
t = 1.0 \; \mathrm{s}$. 
Figures \ref{Fig:NN_Slug_C_t_dot05} and \ref{Fig:NN_Slug_C_t_1} 
compare the concentrations of the product $C$ for these time 
steps under the Galerkin single-field formulation and the 
proposed numerical framework. 
As evident from these figures, even for this transient problem, 
the proposed numerical framework produced physically meaningful 
non-negative values for the concentration of the product at all 
time levels. The violation of the non-negative constraint under 
the Galerkin single-field formulation is greater for smaller 
time steps and for smaller time levels. 
Figure \ref{Fig:NN_Slug_QP_iterations} shows the number of 
iterations taken by the quadratic programming solver at each 
time level to obtain concentrations of the invariants $F$ and 
$G$. For this transient problem, the CPU times taken by the 
proposed computational framework and the Galerkin formulation 
are $1516.83$ and $1486.19$. That is, the computational 
overhead of the proposed computational framework is $2\%$ 
when compared with the Galerkin formulation. The elapsed 
CPU time is measured using MATLAB's built-in \textsf{tic-toc} 
feature.

\section{CLOSURE}
\label{Sec:S6_NN_Closure} 
We have presented a robust and computationally efficient non-negative 
framework for fast bimolecular diffusive-reactive systems. We have 
rewritten the governing equations for the concentration of reactants 
and product in terms of invariants which are unaffected by the 
underlying reaction. Because of this algebraic transformation, 
instead of solving three coupled (nonlinear) diffusion-reaction 
equations, we need to solve only two uncoupled linear diffusion 
equations. This will considerably decrease the computational 
cost and also avoid numerical challenges due to nonlinear 
terms. One of the main findings of this paper is that a small 
violation of the non-negative constraint and maximum 
principle in the calculation of invariants can result 
in bigger errors in the prediction of the concentration 
of the product.  

Using several numerical experiments, it has been shown that the 
standard single-field formulation (which is also known as the 
Galerkin formulation) gives unphysical negative values for the 
concentration of the invariants, reactants, and the product. 
The formulation even gives unphysical negative value for 
the average of the concentration of the product, which is 
a useful measure in the prediction of the fate of transport 
of chemical species. 
It is also shown that the clipping procedure (which is 
considered as a variational crime) is not a viable fix 
to obtain the non-negative concentrations, and does not 
give accurate results for diffusive-reactive systems. 
On the other hand, the  proposed computational framework 
provided accurate results, and in all cases the framework 
provided non-negative values for the concentrations of the 
chemical species. The proposed computational framework can 
serve as a robust simulator for anisotropic heterogeneous 
geomodels. A next logical step is to extend the proposed 
computational framework to slow reactions, and 
to more complicated reactions.

\section*{ACKNOWLEDGMENTS}
The first author (K.B.N.) acknowledges the support of the 
National Science Foundation under Grant no. CMMI 
1068181.
The authors also acknowledge the support from the Department of 
Energy through Subsurface Biogeochemical Research Program, 
and SciDAC-2 Project (Grant No. DOE DEFC02-07ER64323). 
Neither the United States Government nor any agency thereof, 
nor any of their employees, makes any warranty, express or 
implied, or assumes any legal liability or responsibility for the 
accuracy, completeness, or usefulness of any information. The 
opinions expressed in this paper are those of the authors and 
do not necessarily reflect that of the sponsors. 

\emph{Author contributions:} Mathematical theory, proofs, 
and algorithm design (K.B.N.), computer implementation 
(K.B.N. \& M.K.M.), designing test problems (K.B.N., M.K.M. 
\& A.J.V.), generation of numerical results and visualization 
(K.B.N. \& M.K.M.), paper writing (K.B.N. \& M.K.M.), and proof 
reading (K.B.N., M.K.M. \& A.J.V.).

\bibliographystyle{unsrt}
\bibliography{Master_References/Master_References,Master_References/Books}
\newpage
\clearpage

\begin{figure}
  \centering
  \includegraphics[clip,scale=0.32]{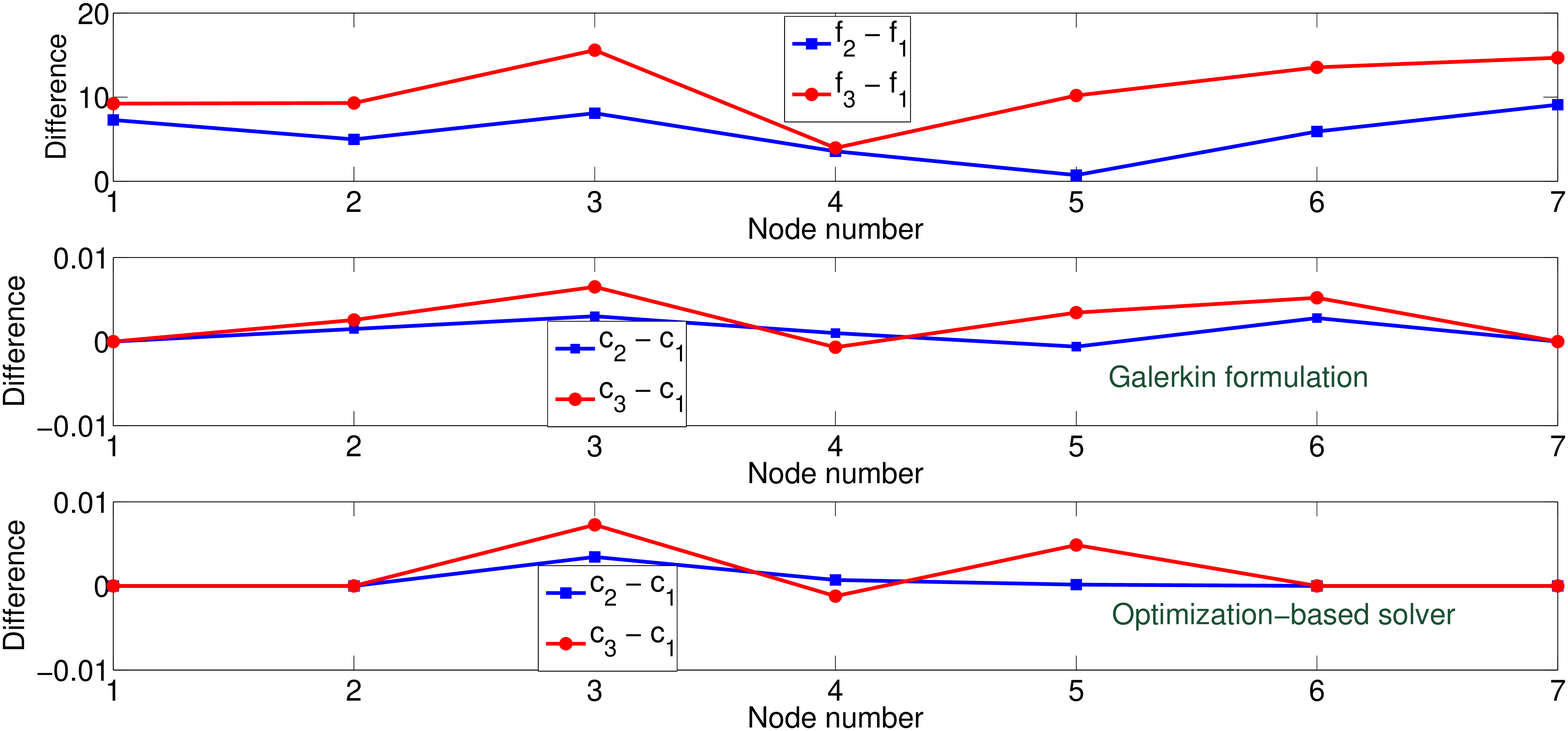}
  \caption{This figure illustrates the violation of the comparison 
    principle by the optimization-based solver. Note that the 
    Galerkin formulation also violates the comparison principle. 
    The problem is a one-dimensional problem with homogeneous 
    Dirichlet boundary conditions. The decay coefficient is 
    taken as $2 \times 10^{4}$, the diffusivity coefficient 
    is taken as unity, and the domain is meshed using six 
    linear finite elements. As evident from the figure, we 
    have $\boldsymbol{f}_3 \succeq \boldsymbol{f}_2 \succeq 
    \boldsymbol{f}_1$ but the obtained numerical solution 
    does \emph{not} obey $\boldsymbol{c}_3 \succeq 
    \boldsymbol{c}_2 \succeq \boldsymbol{c}_1$. 
    \label{Fig:NN_violation_of_DCP}}
\end{figure}

\begin{figure}
  \psfrag{F}{Volumetric source of invariant-$F$ is given by equation 
    \ref{Eqn:Volumetric_Source_Invariant_F}}
  \psfrag{G}{Volumetric source of invariant-$G$ is given by equation 
    \ref{Eqn:Volumetric_Source_Invariant_G}}
  \psfrag{xo}{$\mathbf{x}_{\mathrm{origin}}$}
  \psfrag{Lx}{$L_x$}
  \psfrag{Ly}{\rotatebox{-90}{$L_y$}}
  \psfrag{BC1}{\rotatebox{90}{$c_F^{\mathrm{p}}(\mathbf{x}) = 0$ , 
    $c_G^{\mathrm{p}}(\mathbf{x}) = \cos \left( \frac{\pi y}{2 L_y} \right)$}}
  \psfrag{BC2}{$c_F^{\mathrm{p}}(\mathbf{x}) = \sin \left( \frac{\pi x}{2 L_x} 
    \right)$ , $c_G^{\mathrm{p}}(\mathbf{x}) = 0$}
  \psfrag{BC3}{\rotatebox{-90}{$c_F^{\mathrm{p}}(\mathbf{x}) = \sin \left( 
    \frac{\pi y}{2 L_y} \right)$ , $c_G^{\mathrm{p}}(\mathbf{x}) = 0$}}
  \psfrag{BC4}{$c_F^{\mathrm{p}}(\mathbf{x}) = 0$ , $c_G^{\mathrm{p}}
  (\mathbf{x}) = \cos \left( \frac{\pi x}{2 L_x} \right)$}
  \includegraphics[scale=0.9]{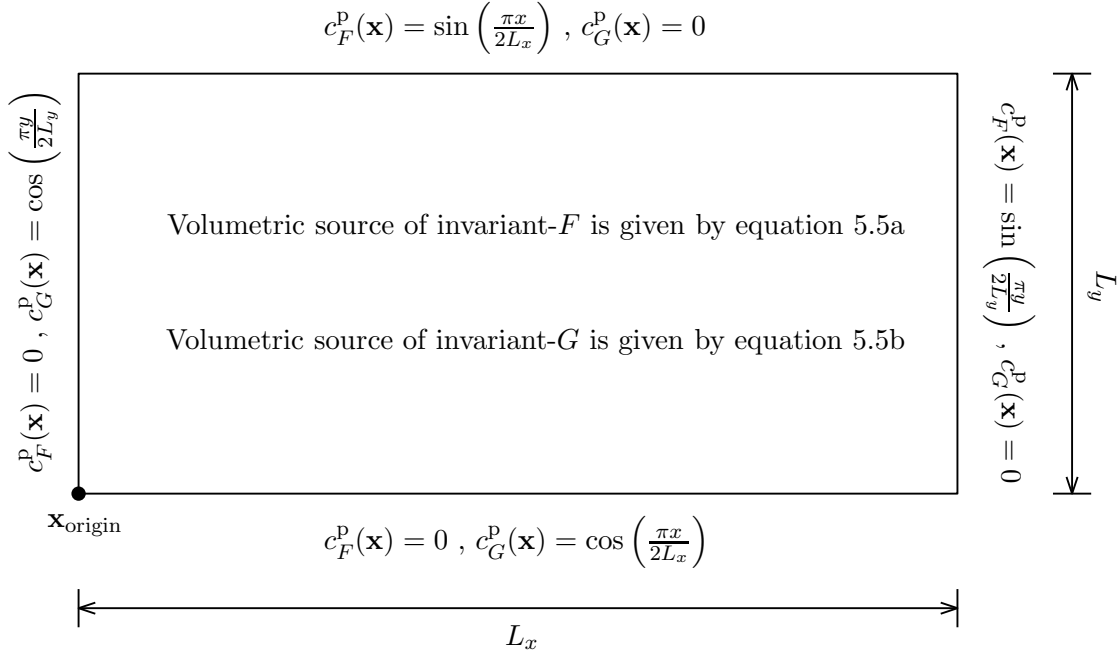}
  \caption{Numerical $h$-convergence study: A pictorial description of the 
    boundary value problem. The computational domain is a rectangle with 
    origin at $\mathbf{x}_{\mathrm{origin}} = (0,0)$. The length and width 
    of the rectangular domain is $L_x = 2$ and $L_y = 1$. The volumetric 
    sources for invariants $F$ and $G$ are prescribed. The Dirichlet boundary 
    conditions $c_F^{\mathrm{p}}(\mathbf{x})$ and $c_G^{\mathrm{p}}(\mathbf{x})$ 
    are prescribed directly by evaluating the expressions given by equations 
    \eqref{Eqn:Analytical_Invariant_F} and \eqref{Eqn:Analytical_Invariant_G} 
    on the boundary.  
    \label{Fig:Convergence_Study_SteadyState}}
\end{figure}

\begin{figure}
  \subfigure{\includegraphics[scale=0.33]
    {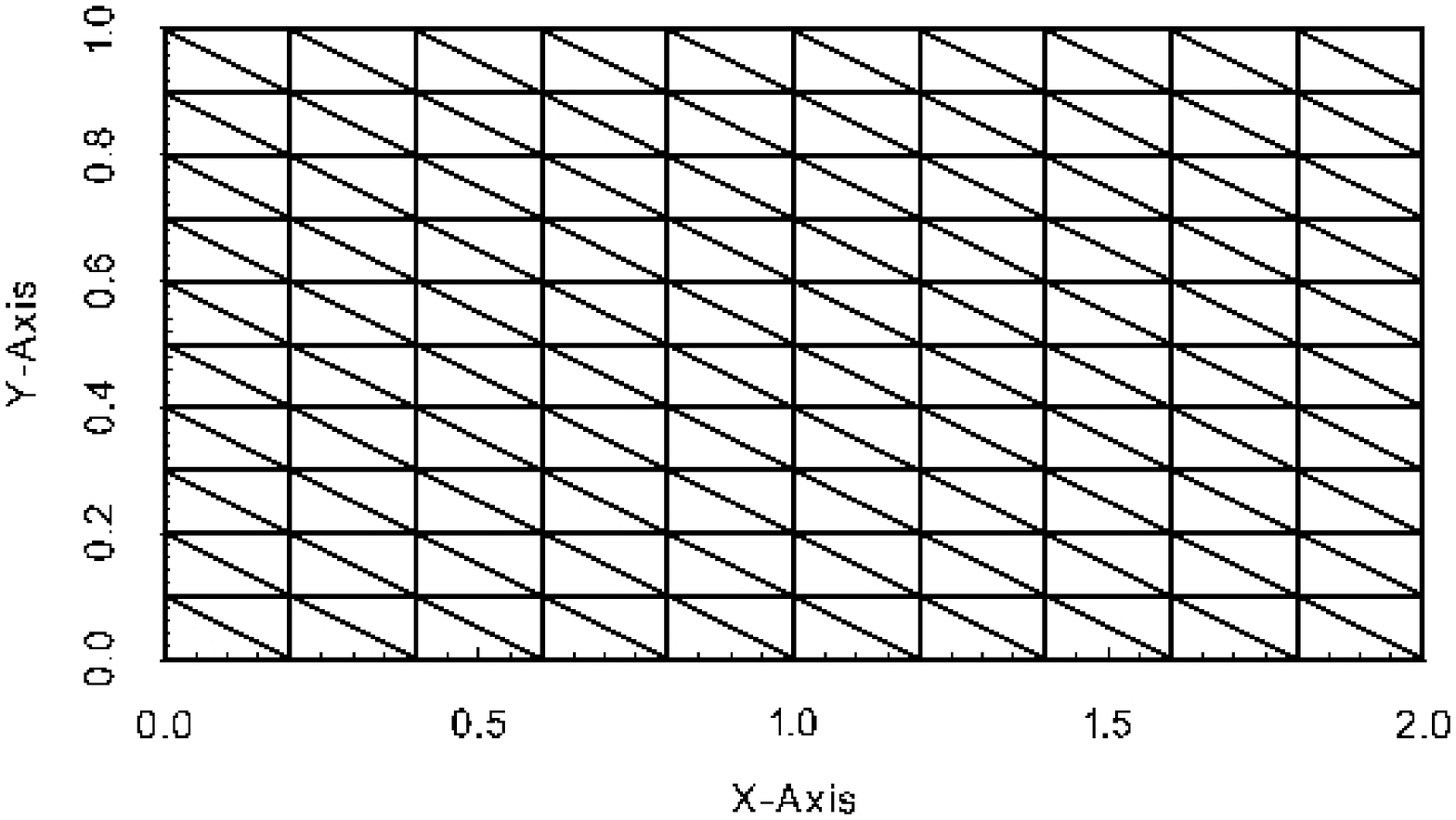}}
  \subfigure{\includegraphics[scale=0.33]
    {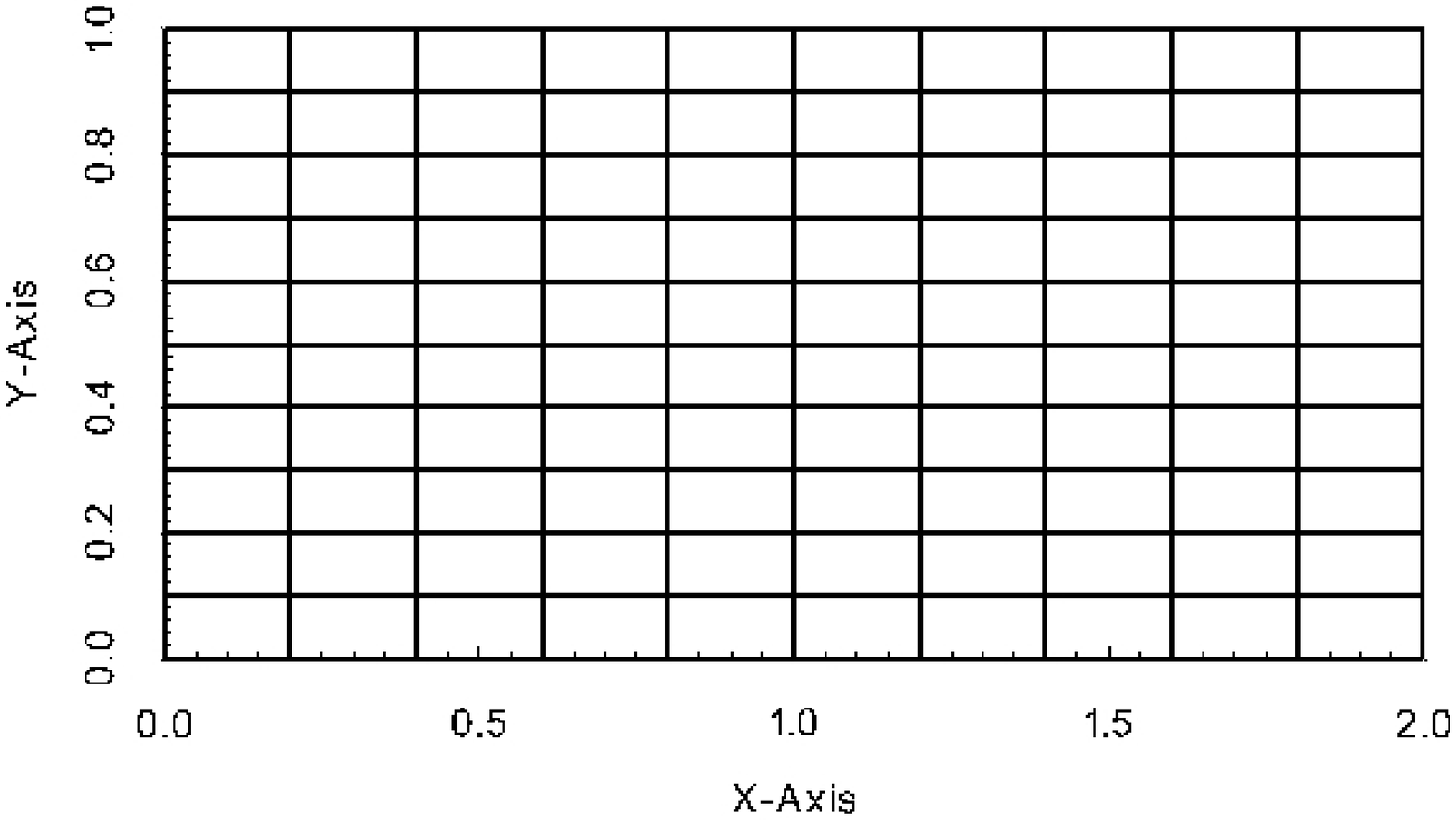}
    \label{Fig:11_Node_Q4_Mesh}}
  \caption{Numerical $h$-convergence study: This figure shows the 
  typical computational meshes (using three-node triangular and 
  four-node quadrilateral elements) employed in the numerical 
  convergence study. Hierarchical meshes are employed in the 
  study. The meshes shown in the figure have $11$ nodes along 
  each side of the domain. Convergence analysis is performed 
  using $11\times11$, $21\times21$, $41\times41$, and $81\times
  81$ meshes. \label{Fig:Hierarchical_Meshes_Convergence}}
\end{figure}

\begin{figure}
  \subfigure{\includegraphics[scale=0.22]
    {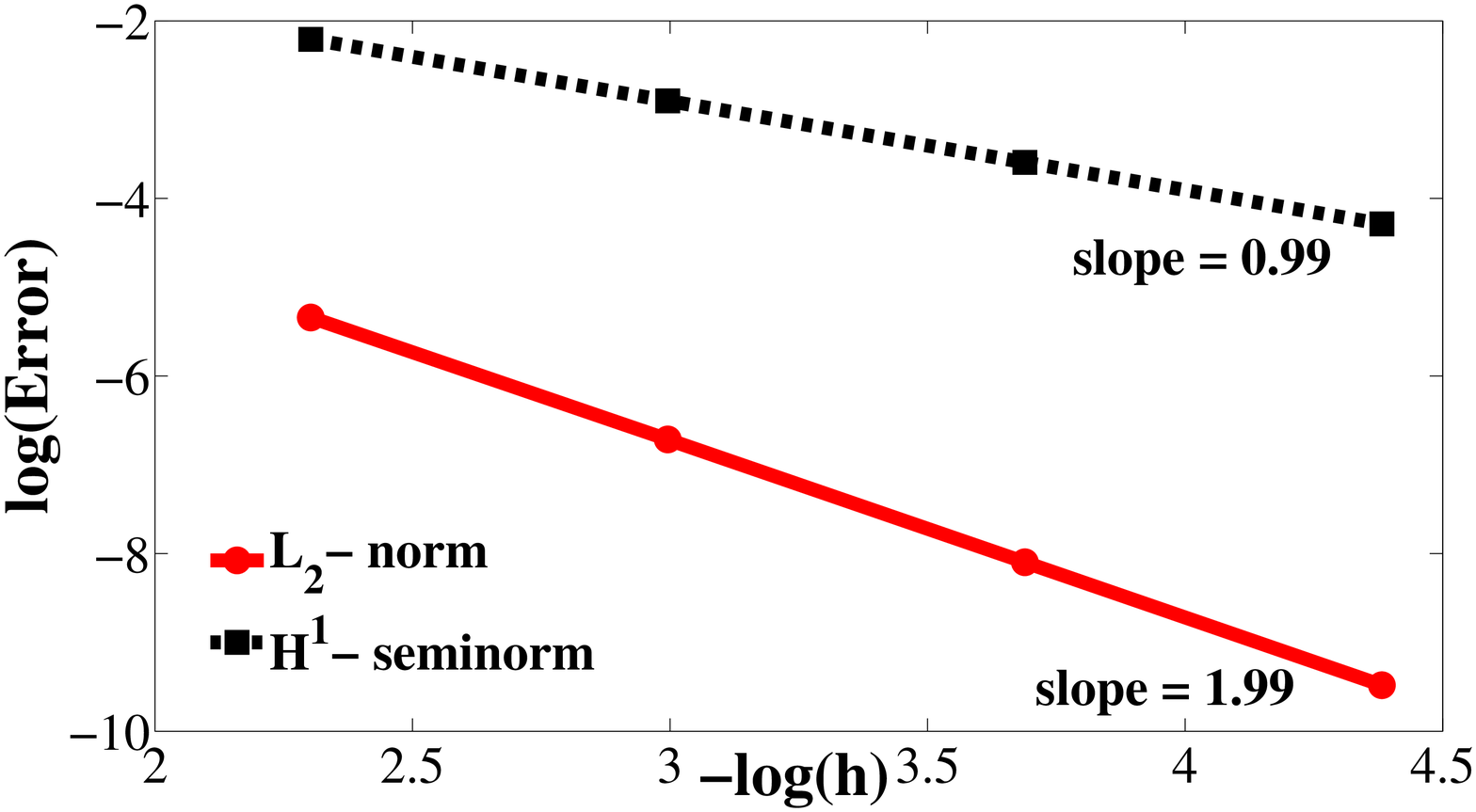}}
  \subfigure{\includegraphics[scale=0.22]
    {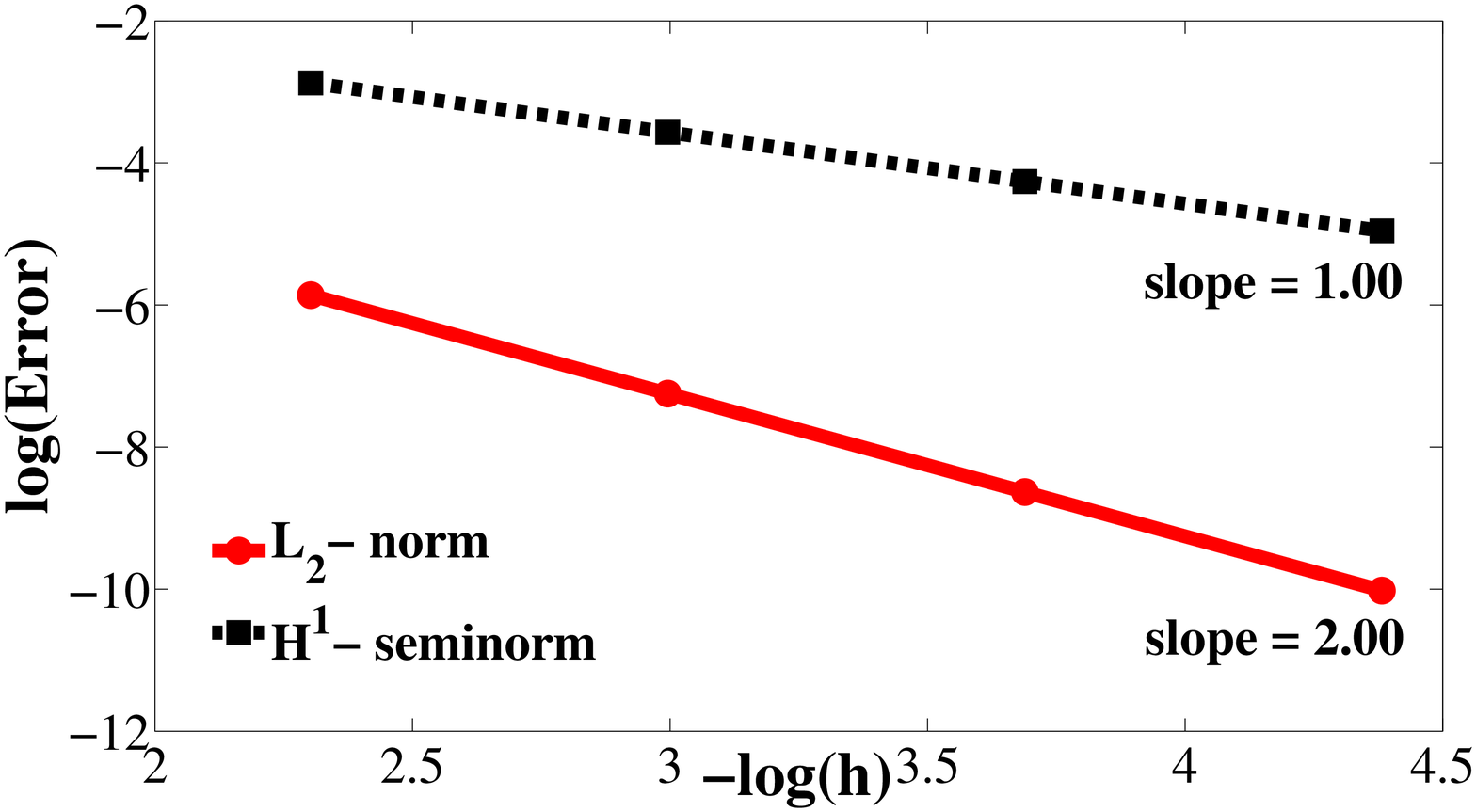}}
  \caption{Numerical $h$-convergence study: This figure shows the variation 
    of error in the concentrations of invariants $F$ and $G$ with respect 
    to mesh refinement. The left figure is for three-node triangular 
    element, and the right figure is for four-node quadrilateral 
    element. The error is calculated in both $L_2$ norm and $H^{1}$ 
    seminorm. Since the obtained values for the error are identical 
    for both $F$ and $G$, only the error for invariant $F$ is shown in the 
    figure. \label{Fig:Bimolecular_h_convergence_of_F_G}}
\end{figure}

\begin{figure}
  \subfigure{\includegraphics[scale=0.22]
    {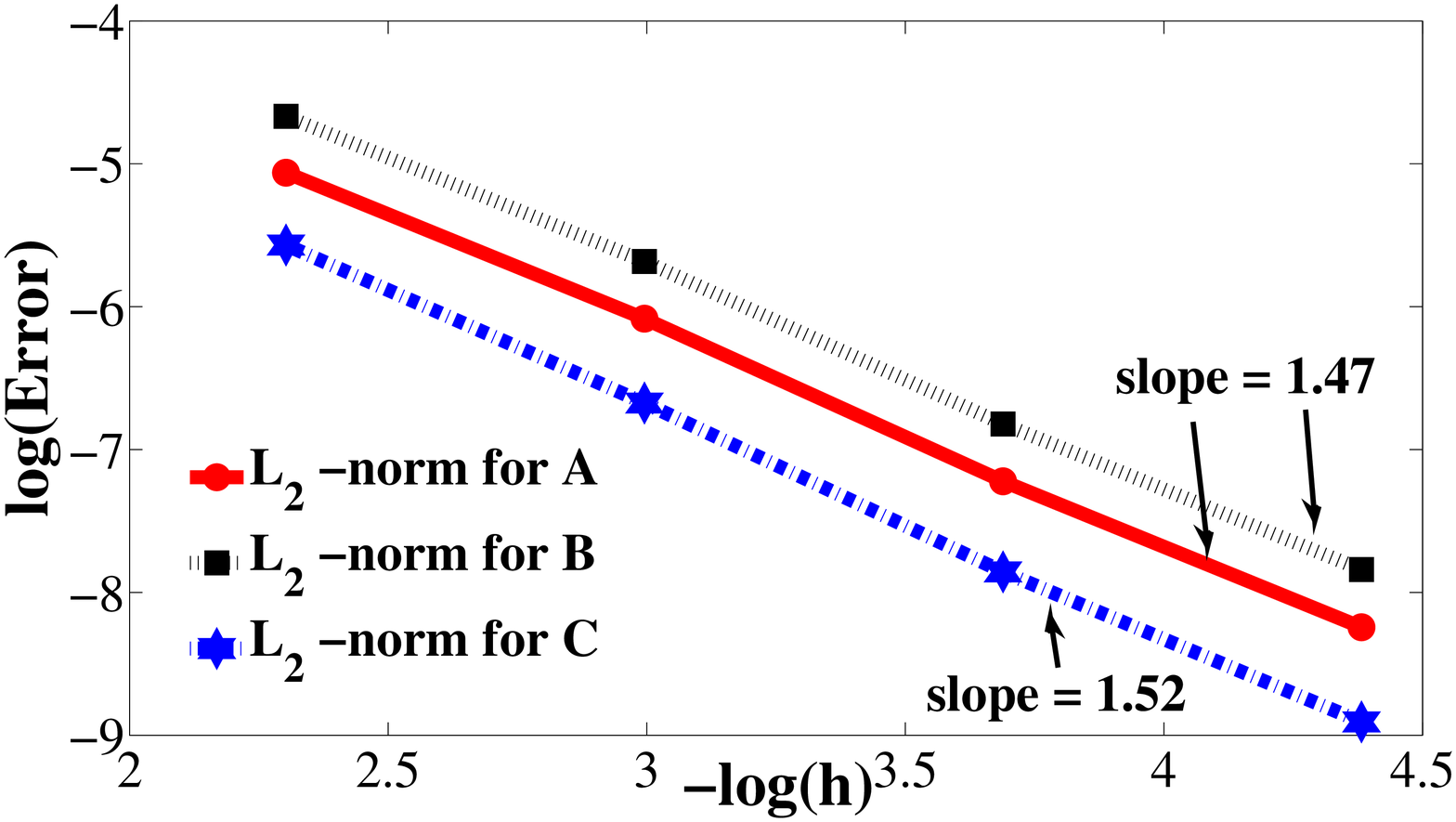}}
  \subfigure{\includegraphics[scale=0.22]
    {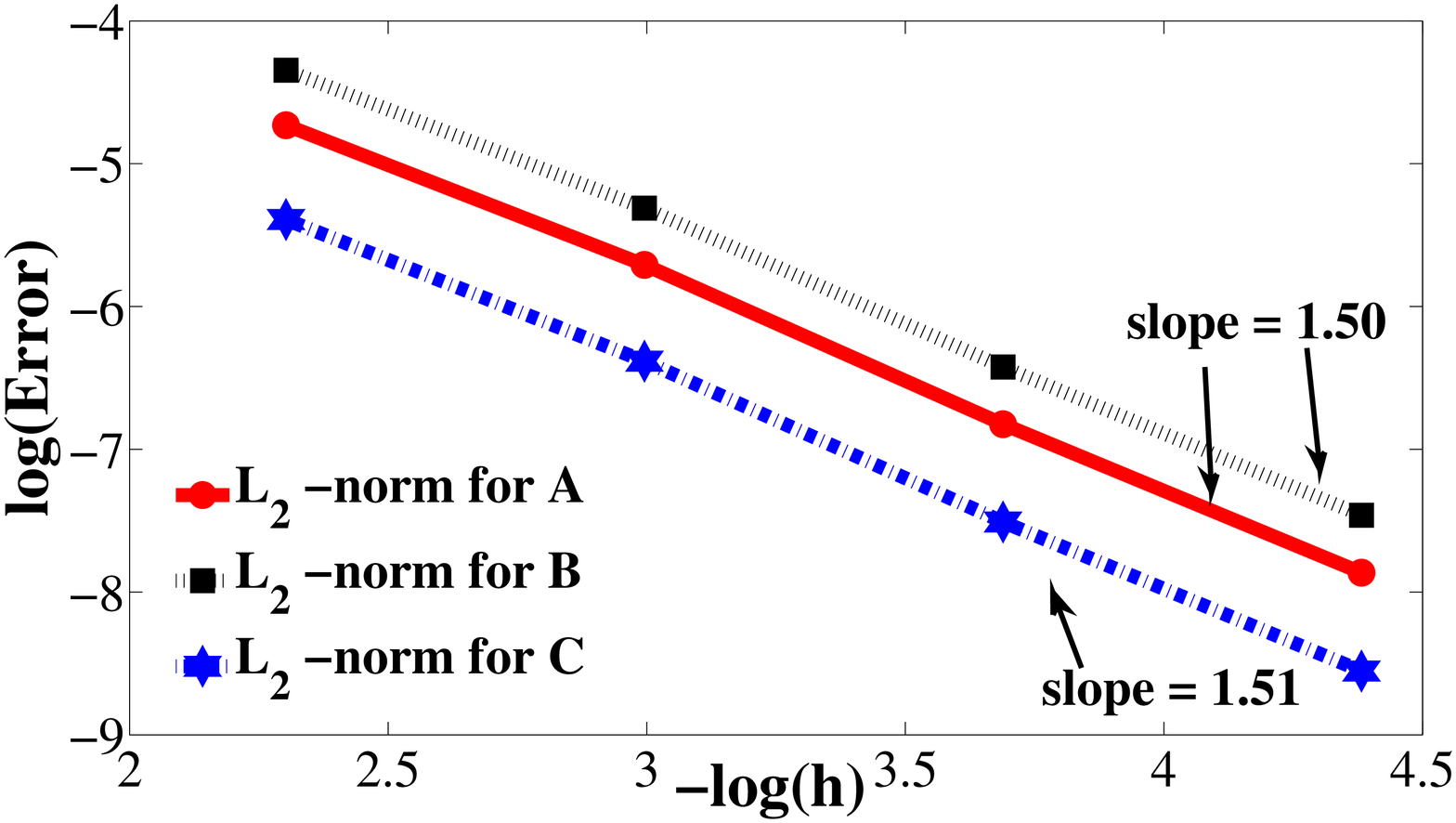}}
  \caption{Numerical $h$-convergence study: This figure 
    shows the variation of error in the concentrations 
    of $A$, $B$, and $C$ with respect to mesh refinement. 
    The left figure is for three-node triangular element, 
    and the right figure is for four-node quadrilateral 
    element. It is evident that the numerical solution 
    obtained from the proposed computational framework 
    converges with respect to mesh refinement. 
    \label{Fig:Bimolecular_h_convergence_of_A_B_C}}
\end{figure}


\begin{figure}
  \psfrag{A}{$A$}
  \psfrag{B}{$B$}
  \psfrag{P}{$\mathrm{P}$}
  \psfrag{Q}{$\mathrm{Q}$}
  \psfrag{cAp}{$c_A^{\mathrm{p}}$}
  \psfrag{cBp}{$c_B^{\mathrm{p}}$}
  \psfrag{Lx}{$L_x$}
  \psfrag{Ly}{$L_y$}
  \psfrag{Ly6}{$L_y/6$}
  \psfrag{x}{$\boldsymbol{x}$}
  \psfrag{y}{$\boldsymbol{y}$}
  \psfrag{zero}{zero flux}
  \psfrag{D}{Dirichlet BC}
  \includegraphics[scale=0.6]{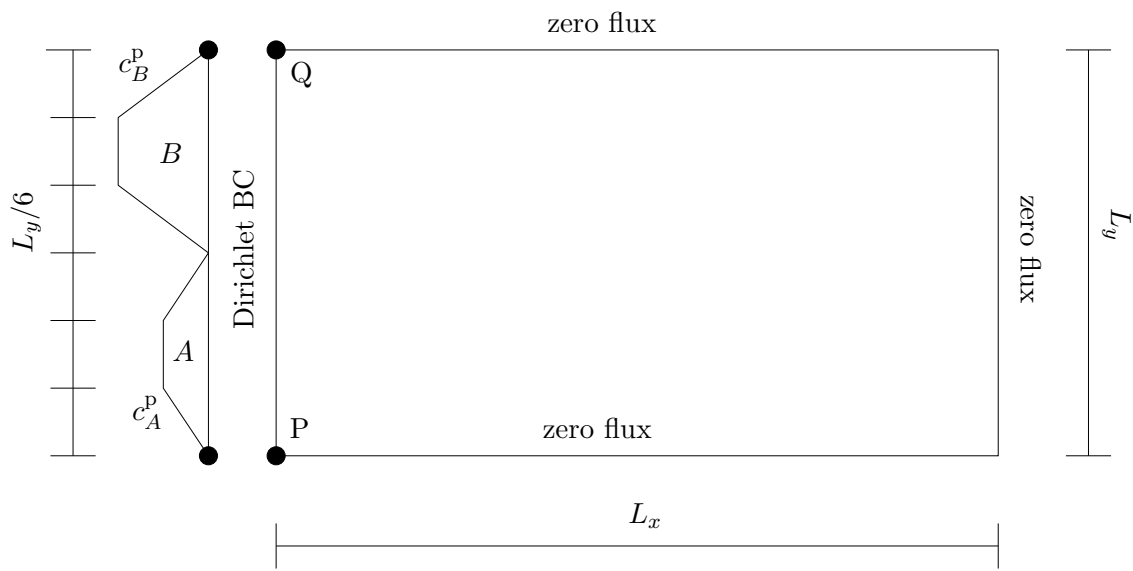}
  \caption{Plume development from boundary in a reaction tank: A 
    pictorial description of boundary value  problem. The computational 
    domain is a rectangle with reactants $A$ and $B$ are pumped on the 
    left face of the domain, which are modeled as Dirichlet boundary 
    conditions. On the remaining faces of the domain, no flux (i.e., 
    Neumann) boundary condition is applied. Dirichlet boundary 
    conditions for $A$ and $B$ are also indicated in the figure. 
    \label{Fig:Reaction_tank}}
\end{figure}

\begin{figure}
  \includegraphics[scale=0.75]{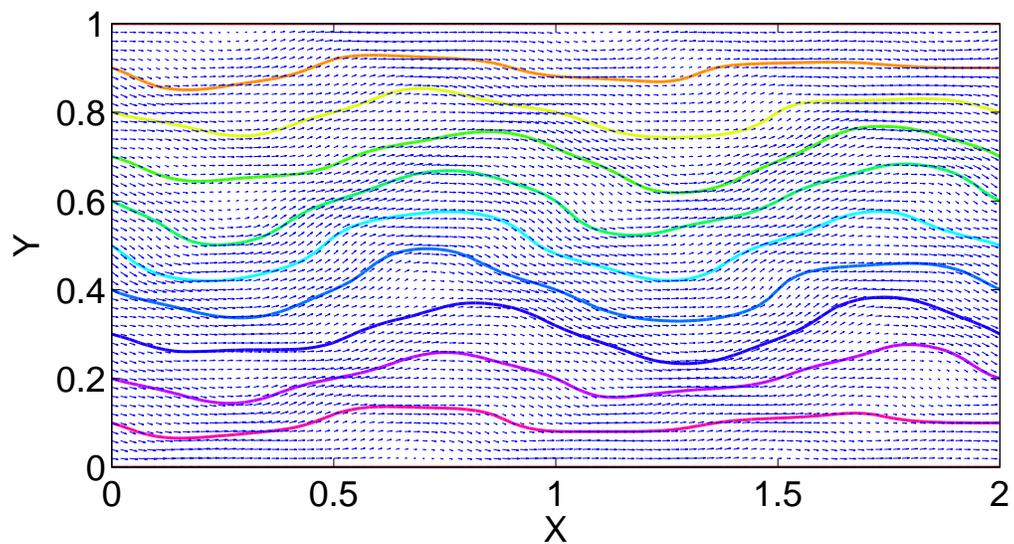}
  \caption{This figure shows the contours of the multi-mode stream function 
  given by equation \eqref{Eqn:NN_Plume_subsurface_stream_function}, 
  and corresponding vector field given by equation
  \eqref{Eqn:NN_Plume_subsurface_velocity_vector_field}. It should 
  be emphasized that the velocity should not be associated with the 
  advection velocity, as advection is neglected in this paper. 
  Herein, the velocity (given by equation 
  \eqref{Eqn:NN_Plume_subsurface_velocity_vector_field}) should 
  be interpreted as the principal directions (or eigenvectors) 
  of the diffusivity tensor (which is given by equation 
  \eqref{Eqn:NN_subsurface_D}). Since the streamlines do 
  not intersect, the nature of the vector fields can 
  provide insights on the plume of the reactants and 
  the product. \label{Fig:NN_multi_mode_stream_function}}
\end{figure}

\clearpage
\newpage

\begin{figure}
  \centering
  \subfigure{
    \includegraphics[clip,scale=0.8]{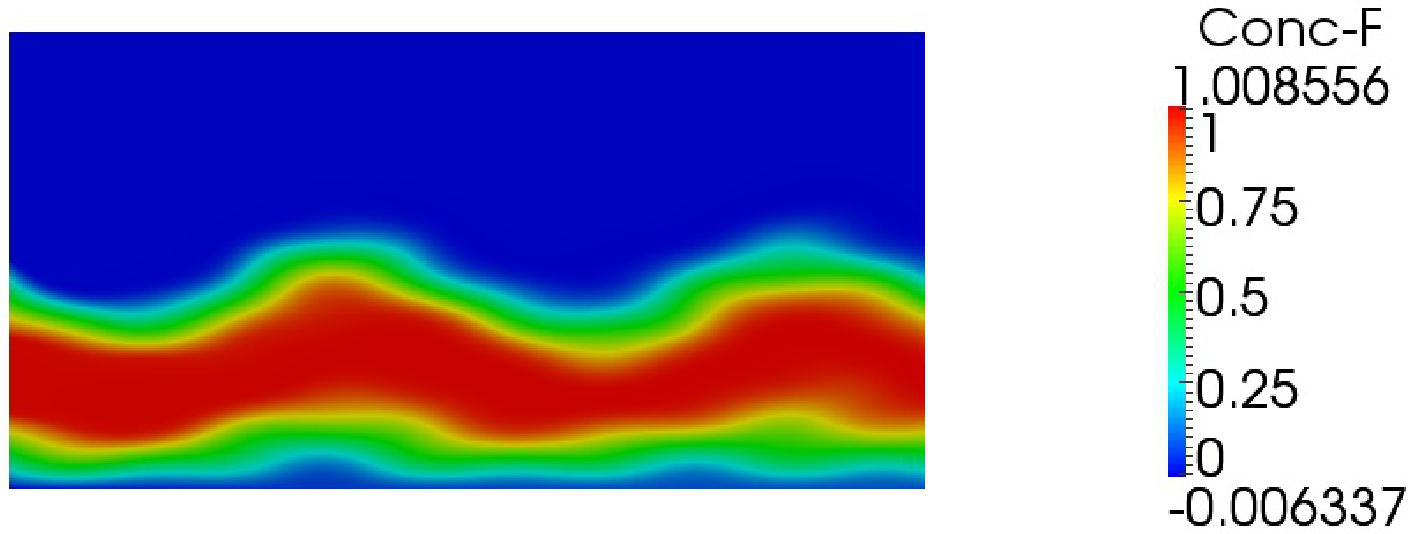}}
  \subfigure{
    \includegraphics[clip,scale=0.8]{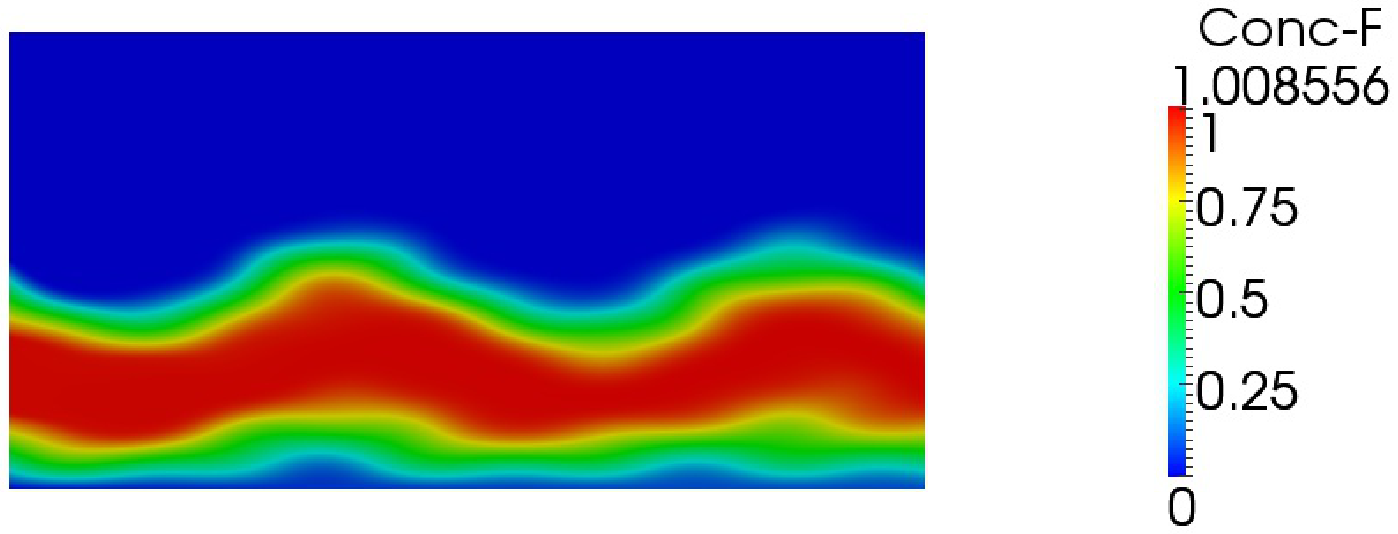}}
  \subfigure{
    \includegraphics[clip,scale=0.8]{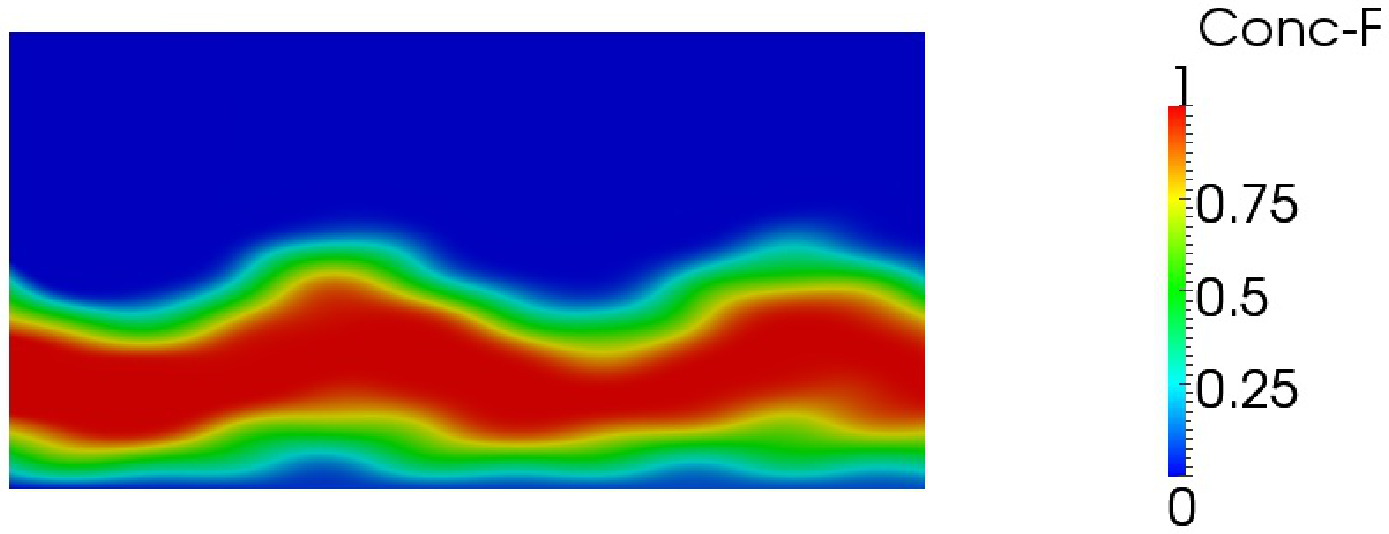}}
  \caption{Plume formation from boundary in a reaction tank: 
    This figure shows the contours of the invariant $F$ 
    obtained using the Galerkin formulation (top figure), 
    the clipping procedure (middle figure), and the proposed 
    computational framework (bottom figure). The concentration 
    of the invariant $F$ should be between $0$ and $1$. The 
    Galerkin formulation violates the non-negative constraint 
    and the maximum principle. 
    On the other hand, by explicitly enforcing the lower and 
    upper bounds, the proposed computational framework satisfies 
    the non-negative constraint and the maximum principle. However, 
    it is noteworthy that the contours of the invariant $F$ under 
    the Galerkin formulation, the clipping procedure and the proposed 
    computational framework all look similar. This is not the 
    case with the contours of the product $C$ (see Figure 
    \ref{Fig:NN_Plume_subsurface_C_Q4_contours}).
    \label{Fig:NN_Plume_subsurface_F_Q4_contours}}
\end{figure}

\begin{figure}
  \centering
  \subfigure{
    \includegraphics[clip,scale=0.8]{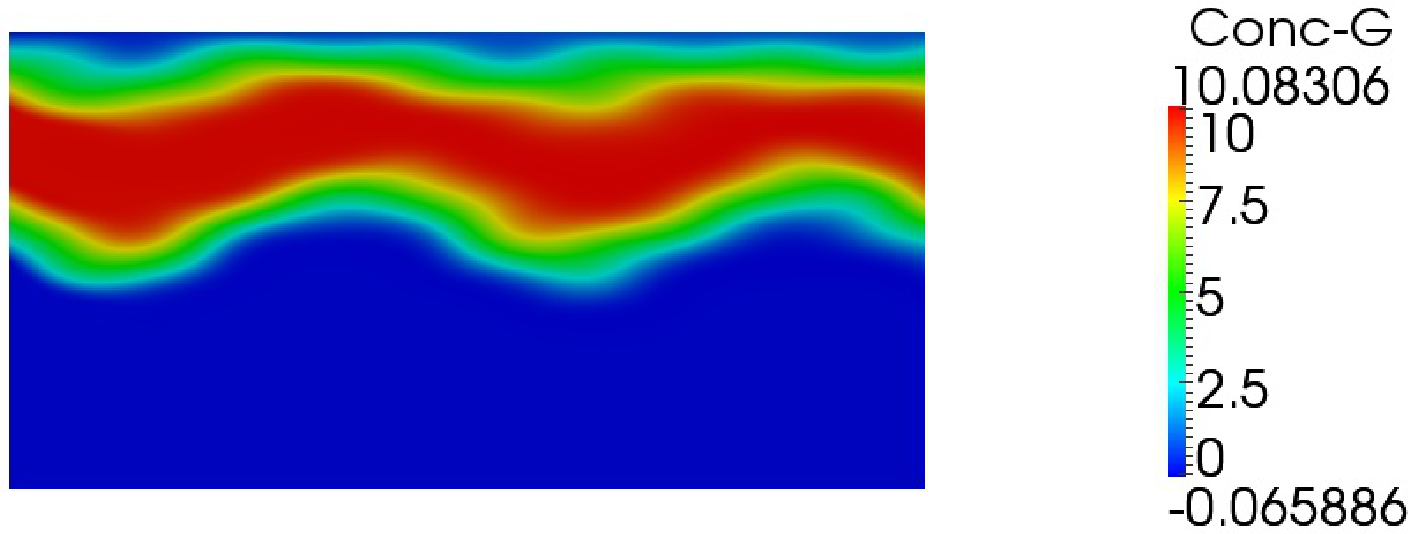}}
  \subfigure{
    \includegraphics[clip,scale=0.8]{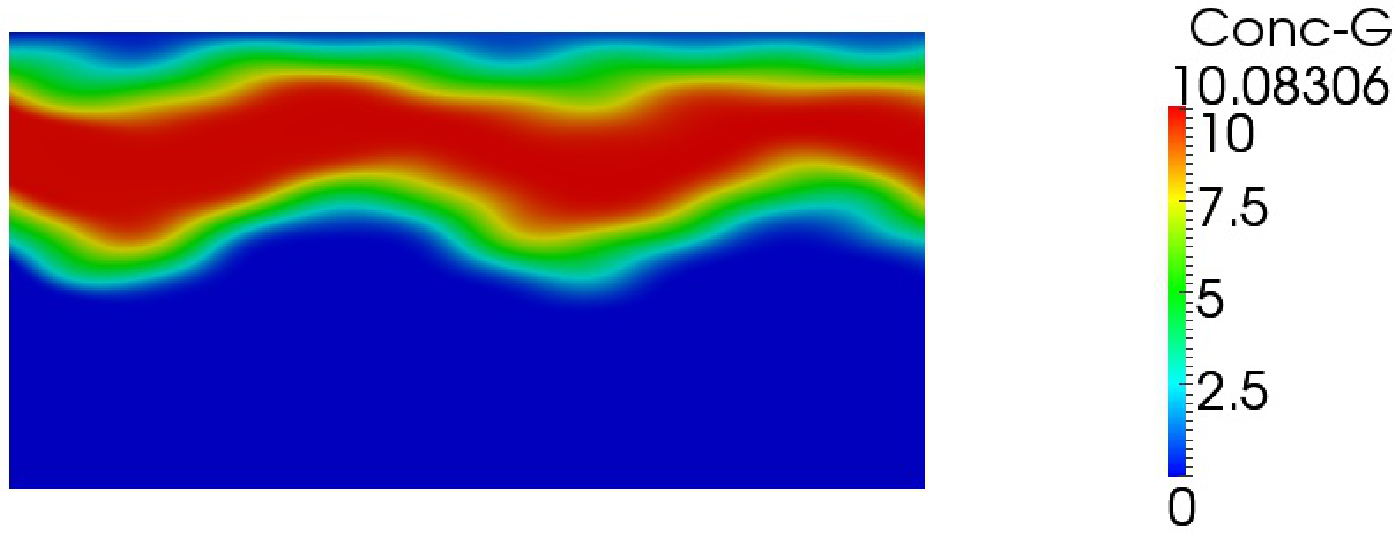}}
  \subfigure{
    \includegraphics[clip,scale=0.8]{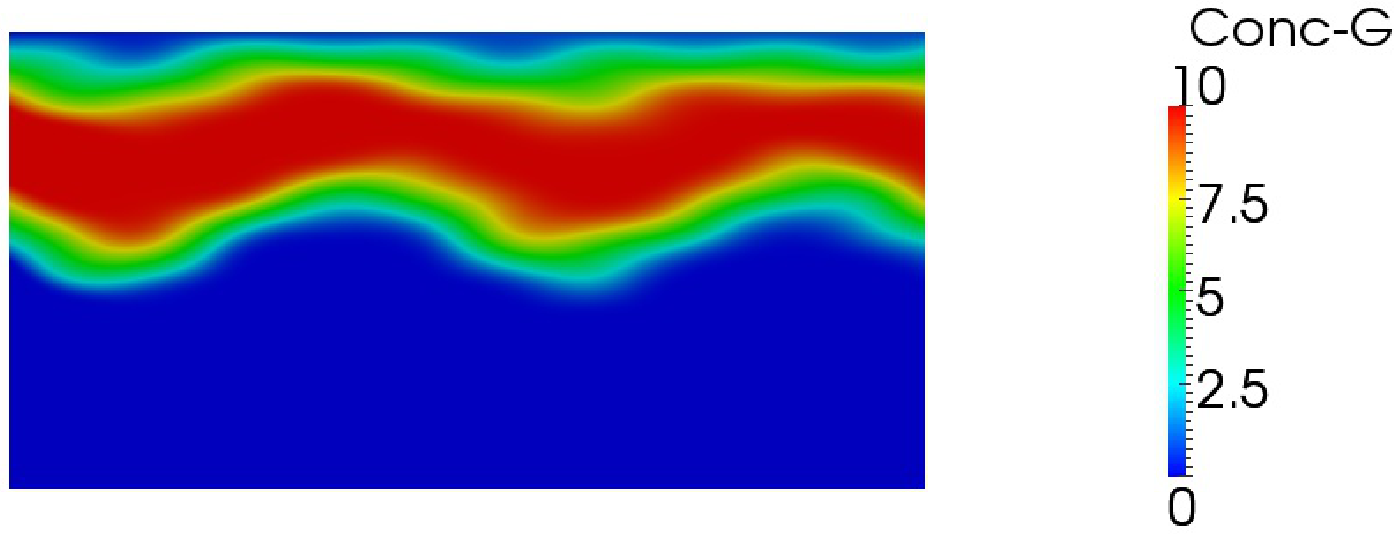}}
  \caption{Plume formation from boundary in a reaction tank: 
    This figure shows the contours of the invariant $G$ 
    obtained using the Galerkin formulation (top figure), 
    the clipping procedure (middle figure), and the proposed 
    computational framework (bottom figure). The concentration 
    of the invariant $G$ should be between $0$ and $10$. The 
    Galerkin formulation violates the non-negative constraint 
    and the maximum principle. 
    On the other hand, by explicitly enforcing the lower and 
    upper bounds, the proposed computational framework satisfies 
    the non-negative constraint and the maximum principle. However, 
    it is noteworthy that the contours of the invariant $G$ under 
    the Galerkin formulation, the clipping procedure and the proposed 
    computational framework all look similar. This is not the 
    case with the contours of the product $C$ (see Figure 
    \ref{Fig:NN_Plume_subsurface_C_Q4_contours}).
    \label{Fig:NN_Plume_subsurface_G_Q4_contours}}
\end{figure}

\begin{figure}
  \centering
  \subfigure{
    \includegraphics[clip,scale=0.8]{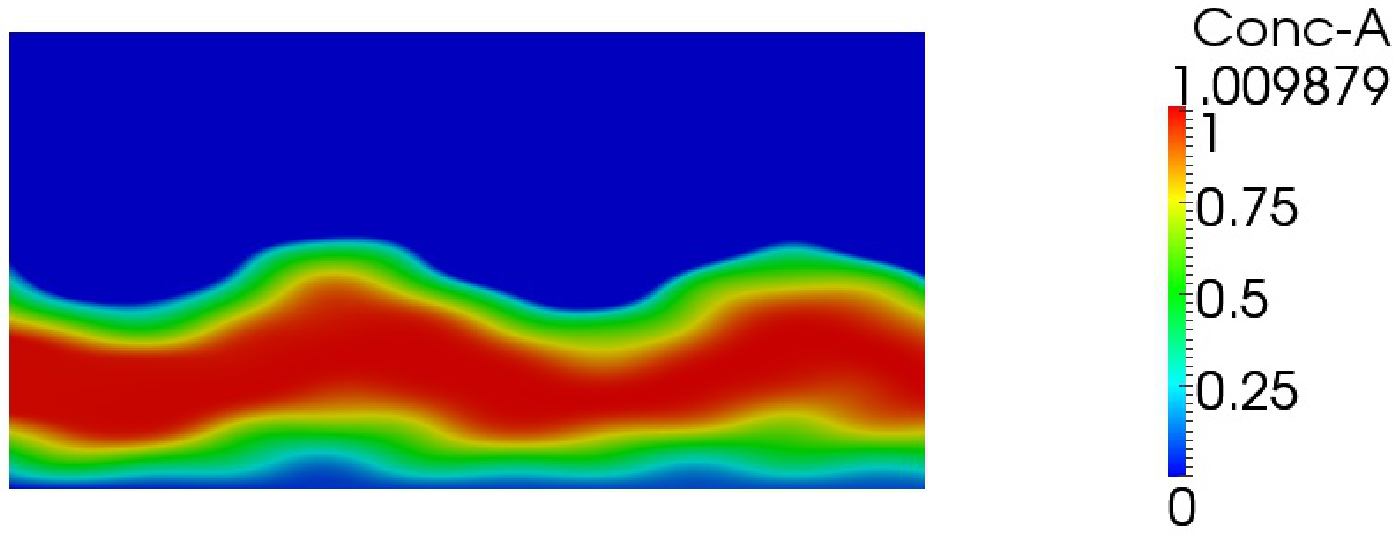}}
  \subfigure{
    \includegraphics[clip,scale=0.8]{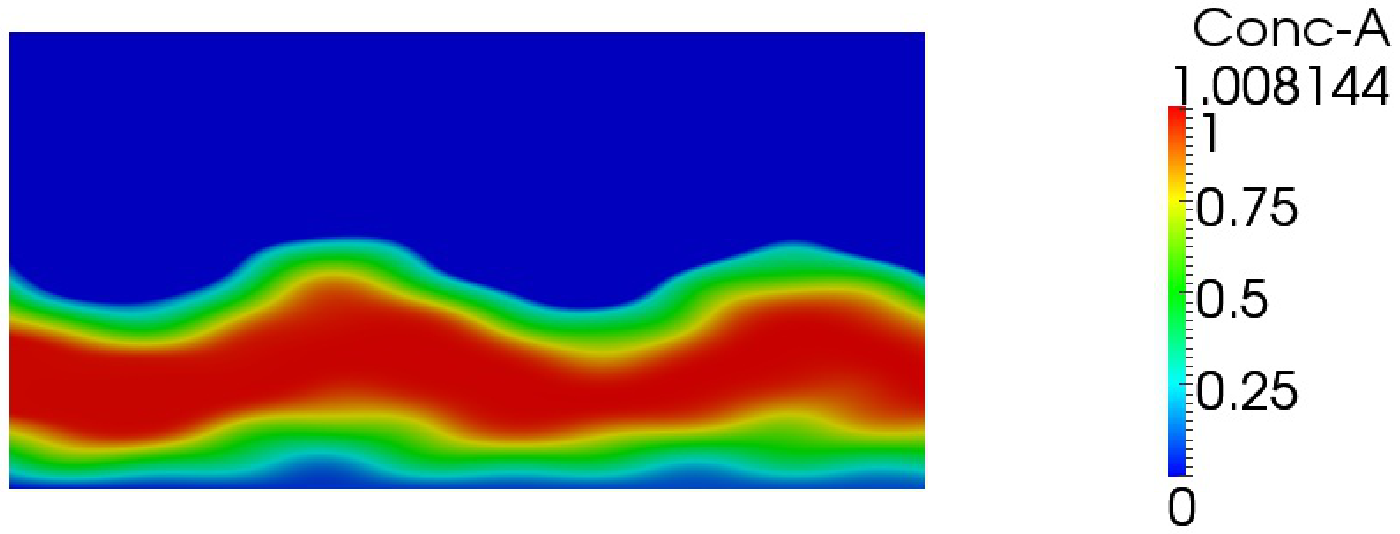}}
  \subfigure{
    \includegraphics[clip,scale=0.8]{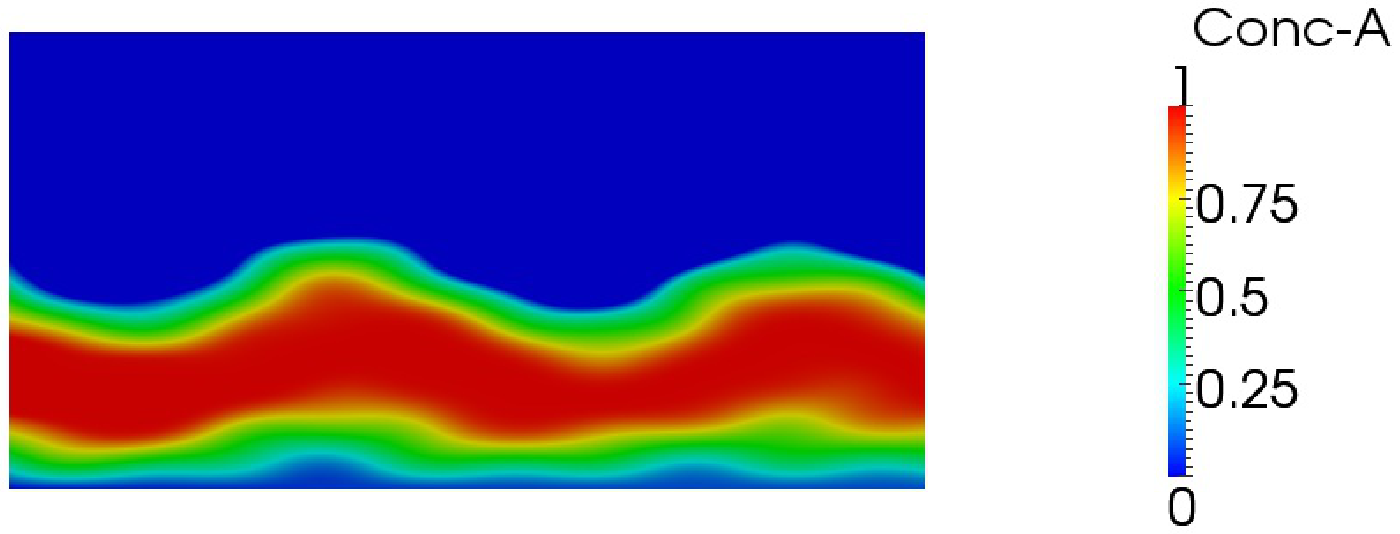}}
  \caption{Plume formation from boundary in a reaction tank: 
    This figure shows the contours of the reactant $A$ obtained 
    using the Galerkin formulation (top figure), the clipping 
    procedure (middle figure), and the proposed computational 
    framework (bottom figure). The concentration of the reactant 
    $A$ should be less than or equal to $1$. 
    The Galerkin formulation violates the maximum principle. 
    On the other hand, the proposed computational framework 
    satisfies the upper bound due to the maximum principle. 
    However, it is noteworthy that the contours of the 
    reactant $A$ under the Galerkin formulation, the 
    clipping procedure and the proposed computational 
    framework all look similar. This is not the case 
    with the contours of the product $C$ (see Figure 
    \ref{Fig:NN_Plume_subsurface_C_Q4_contours}). 
    \label{Fig:NN_Plume_subsurface_A_Q4_contours}}
\end{figure}

\begin{figure}
  \centering
  \subfigure{
    \includegraphics[clip,scale=0.8]{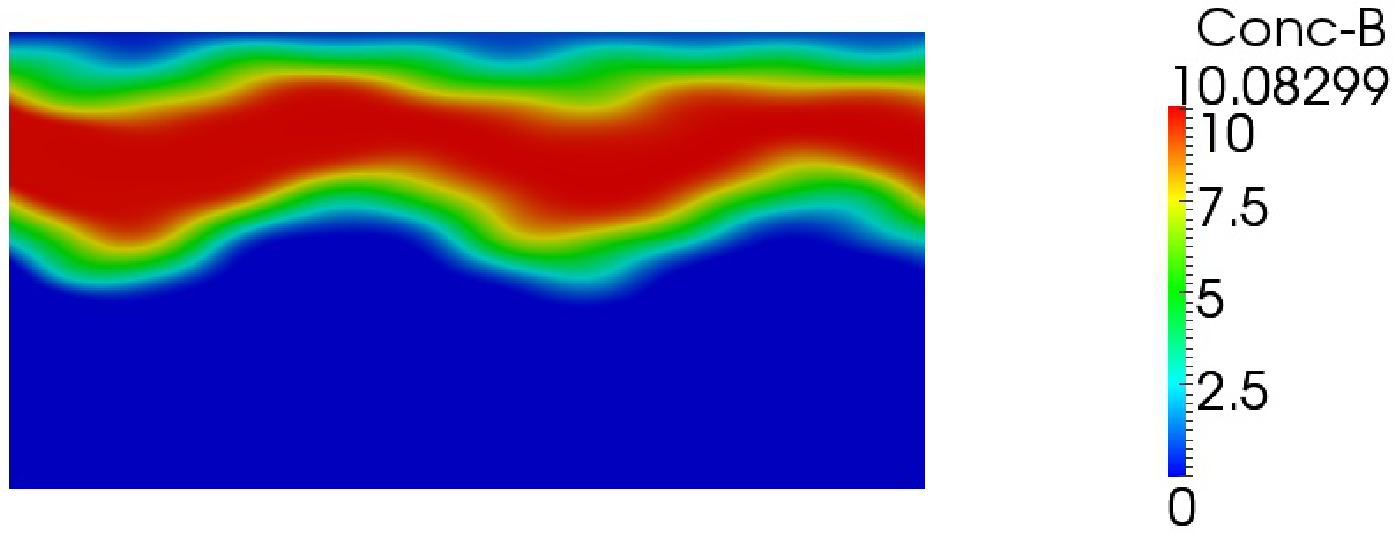}}
  \subfigure{
    \includegraphics[clip,scale=0.8]{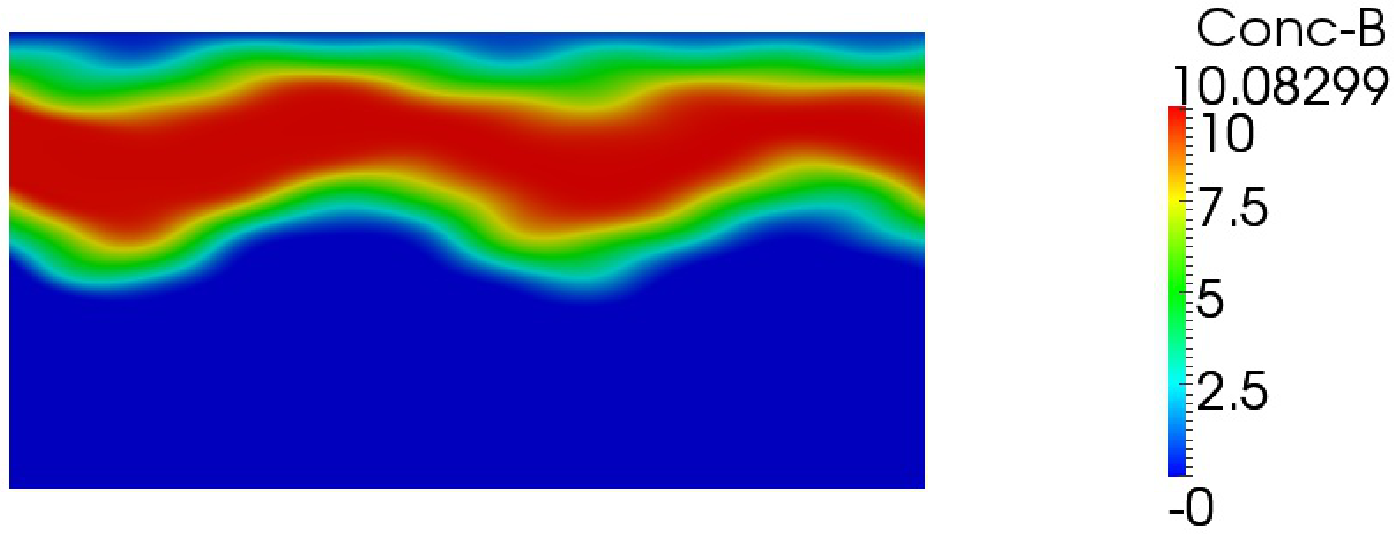}}
  \subfigure{
    \includegraphics[clip,scale=0.8]{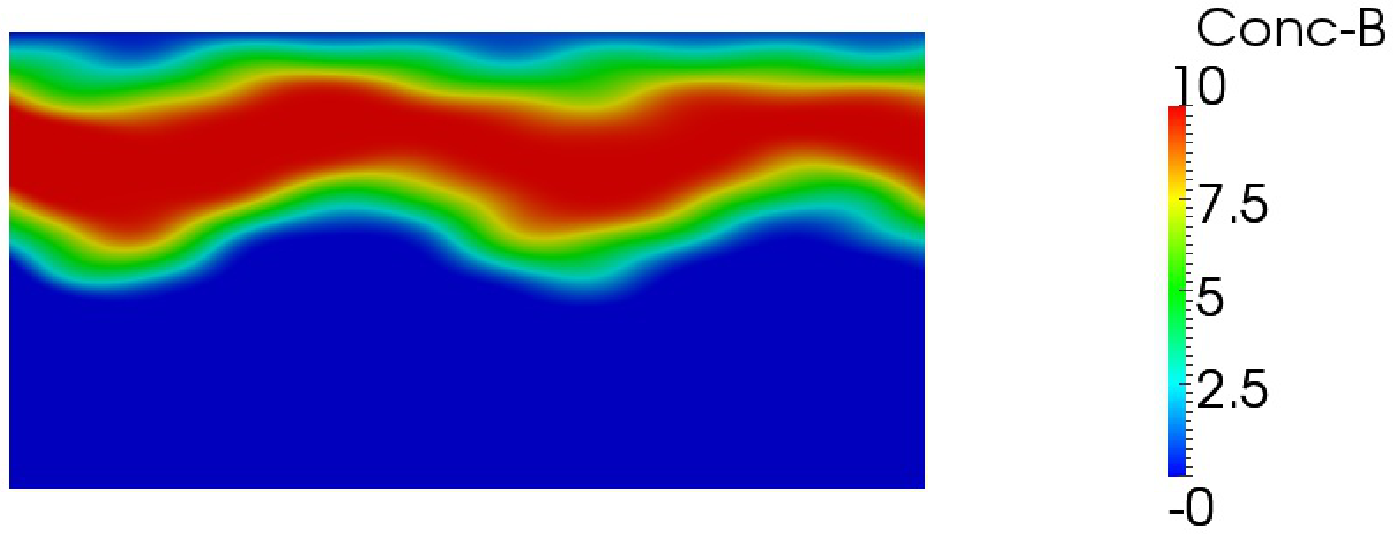}}
  \caption{Plume formation from boundary in a reaction tank: 
    This figure shows the contours of the reactant $B$ obtained 
    using the Galerkin formulation (top figure), the clipping 
    procedure (middle figure), and the proposed computational 
    framework (bottom figure). The concentration of the reactant 
    $B$ should be less than or equal to $10$. 
    The Galerkin formulation violates the maximum principle. 
    On the other hand, the proposed computational framework 
    satisfies the upper bound due to the maximum principle. 
    However, it is noteworthy that the contours of the 
    reactant $B$ under the Galerkin formulation, the 
    clipping procedure and the proposed computational 
    framework all look similar. This is not the case 
    with the contours of the product $C$ (see Figure 
    \ref{Fig:NN_Plume_subsurface_C_Q4_contours}). 
    \label{Fig:NN_Plume_subsurface_B_Q4_contours}}
\end{figure}

\begin{figure}
  \centering
  \subfigure{
    \includegraphics[clip,scale=0.8]{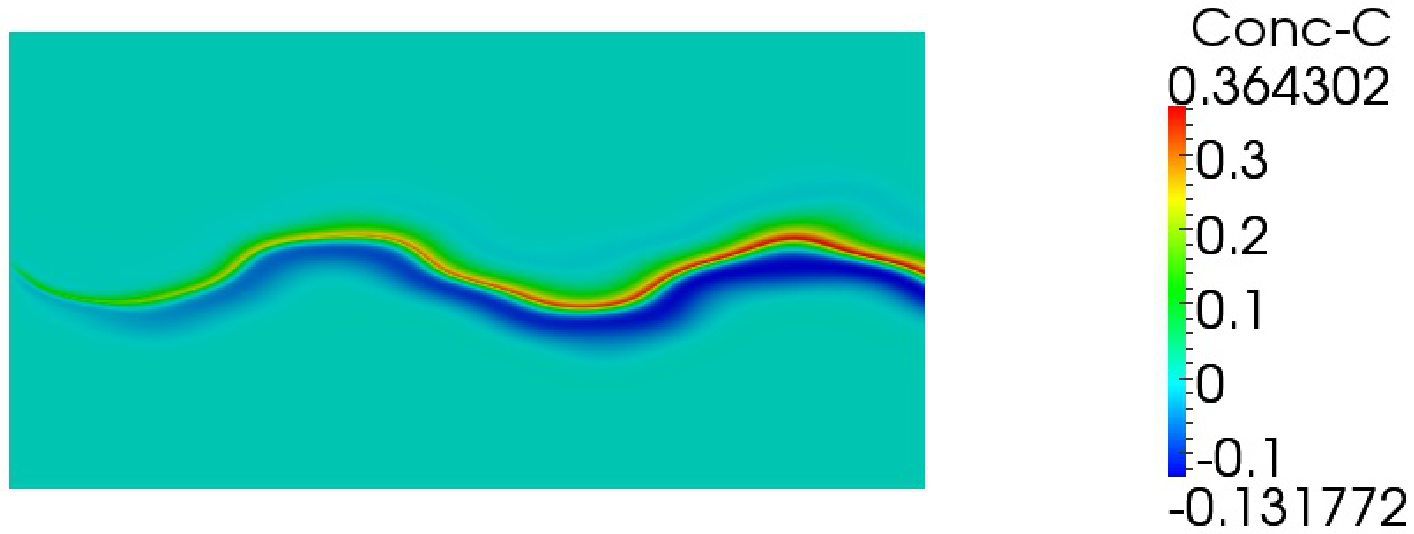}}
  \subfigure{
    \includegraphics[clip,scale=0.8]{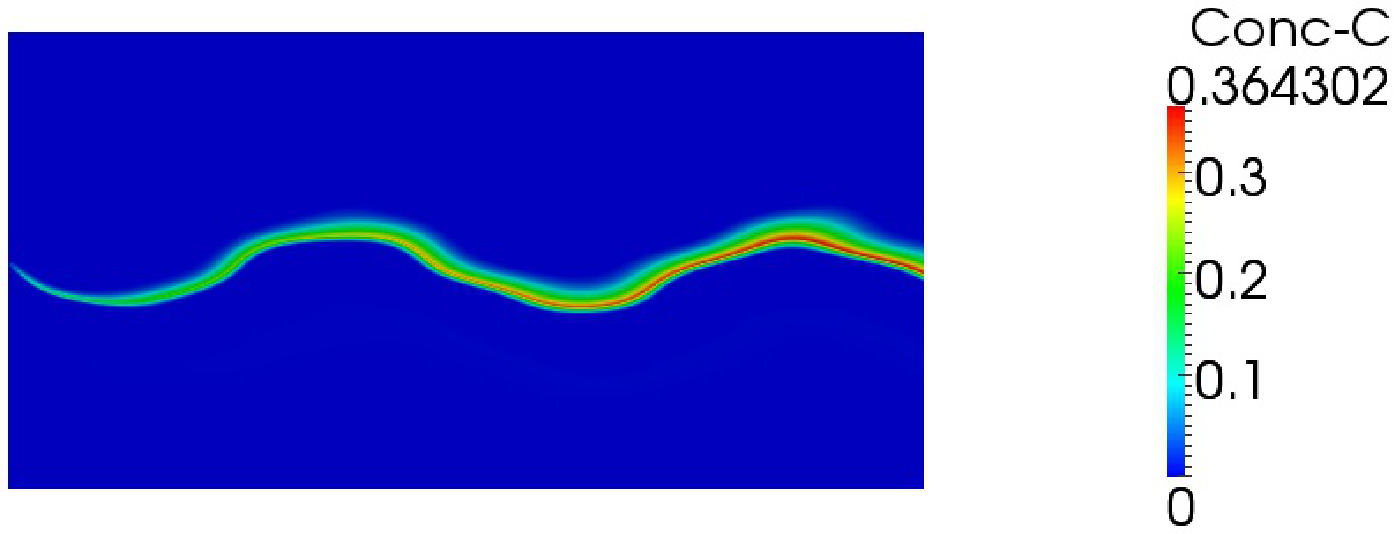}}
  \subfigure{
    \includegraphics[clip,scale=0.8]{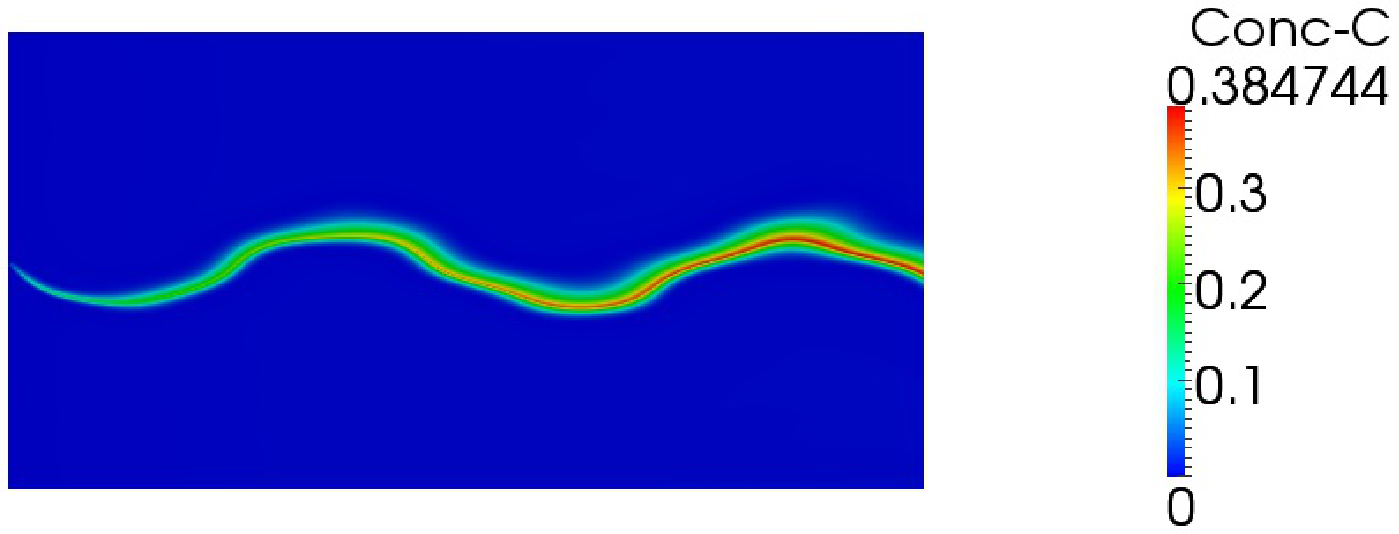}}
  \caption{Plume formation from boundary in a reaction tank: 
    This figure shows the contours of the concentration of 
    the product $C$ obtained using the Galerkin formulation 
    (top figure), the clipping procedure (middle figure), 
    and the proposed computational framework (bottom figure). 
    Mathematically, the concentration of the product $C$ should 
    be non-negative. There is no upper bound on the concentration 
    of the product due to the maximum principle, as there is a 
    source for the product $C$ due to the underlying reaction.
    As one can see from the figure, the Galerkin formulation 
    violates the non-negative constraint. On the other hand, 
    the proposed computational framework satisfies the 
    non-negative constraint for the concentration of the 
    product $C$. More importantly, the contours of the 
    product under the Galerkin formulation, the clipping 
    procedure and the proposed computational framework
    differ considerably. 
    \emph{
    Using this figure, one can draw an important conclusion 
    of the paper: A small violation of the non-negative 
    constraint and the maximum principle in the calculation 
    of the invariants can result in significant errors in the concentration of the 
    product for diffusive-reactive systems.} 
    \label{Fig:NN_Plume_subsurface_C_Q4_contours}}
\end{figure}

\newpage
\clearpage

\begin{figure}
  \centering
  \includegraphics[clip,scale=0.9]{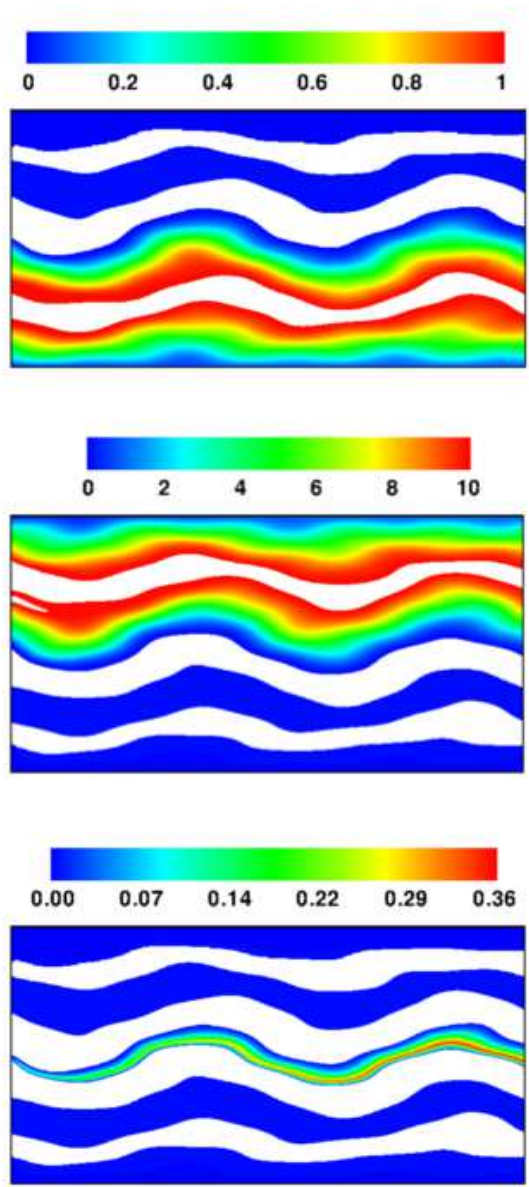}
  \caption{Plume formation from boundary in a reaction tank: This 
  figure shows the contours of invariant $F$ (top figure), invariant 
  $G$ (middle figure), and product $C$ (bottom figure) under the 
  Galerkin formulation. For the invariant $F$, the concentrations 
  outside the interval $[0, 1]$ are cut-off, and for the invariant 
  $G$ the concentrations outside $[0, 10]$ are cut-off. For the 
  product $C$, the concentrations below zero are cut-off. The 
  lower bounds are due to the non-negative constraint, and the 
  upper bounds are due to the maximum principle. The regions with 
  cut-off concentration values are indicated in white color. It 
  is evident that the Galerkin formulation violated the bounds 
  in significant portions of the computational domain.
  \label{Fig:NN_Plume_subsurface_FGC_Q4_contours}}
\end{figure}

\begin{figure}
  \centering
  \psfrag{cC_y_dot45}{$c_C(\mathrm{x},\mathrm{y}=0.45)$}
  \psfrag{cC_x_1dot75}{$c_C(\mathrm{x}=1.75,\mathrm{y})$}
  \psfrag{x}{$\mathrm{x}$}
  \psfrag{y}{$\mathrm{y}$}
  \subfigure{
    \includegraphics[clip,scale=0.45]{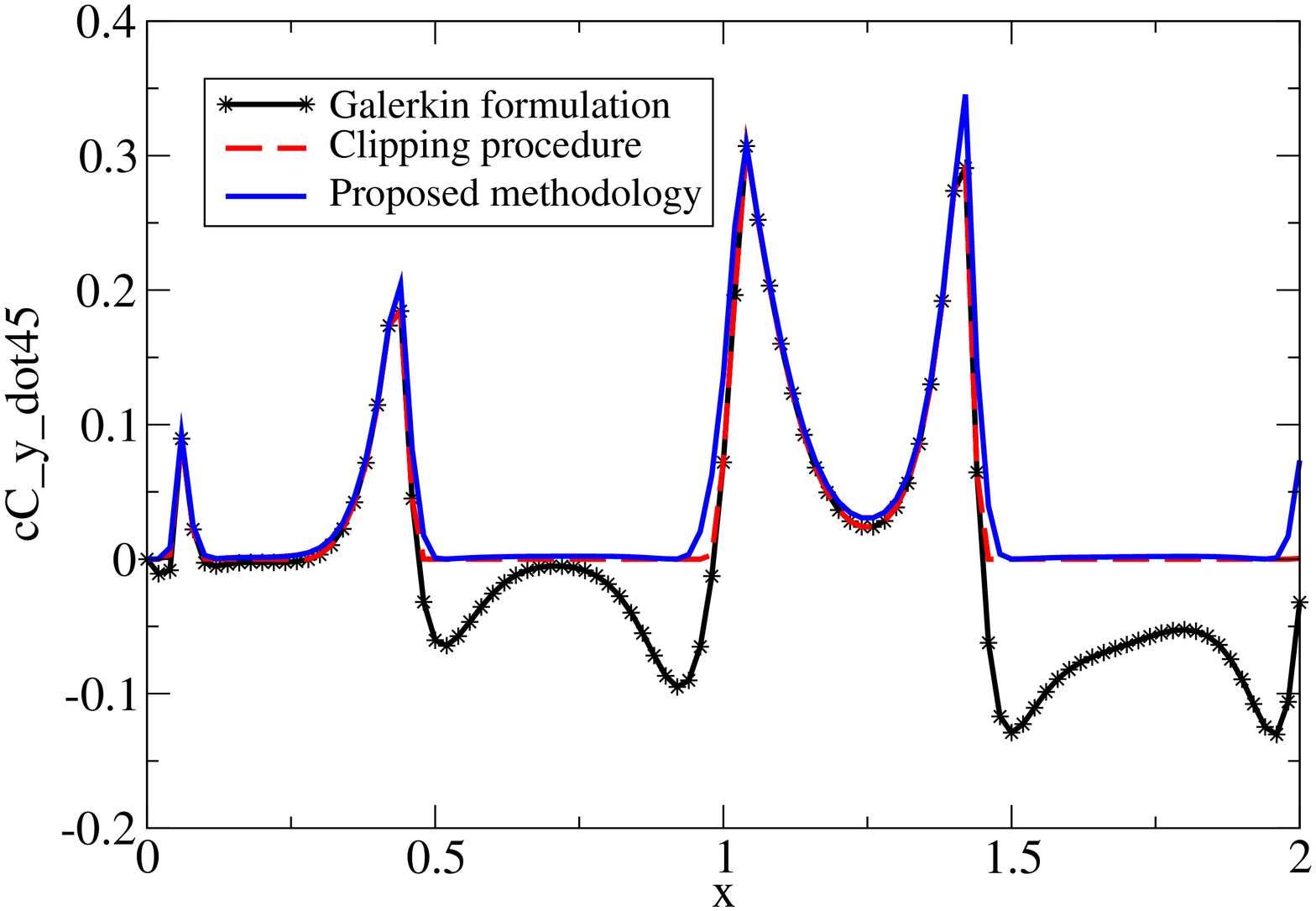}}
  \subfigure{
    \includegraphics[clip,scale=0.45]{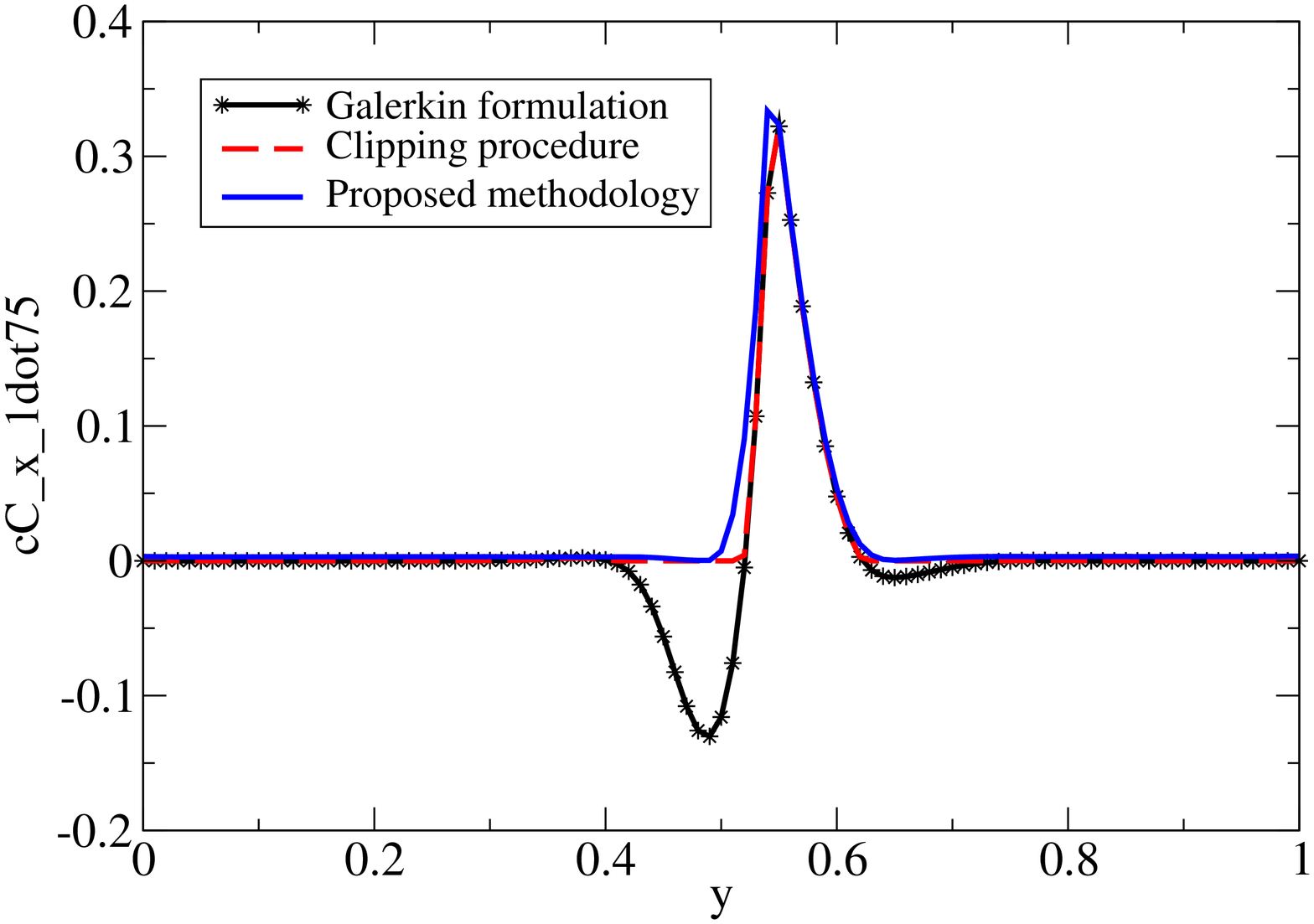}}
  \caption{Plume formation from boundary in a reaction tank: 
    This figure shows the variation of the concentration of 
    the product at $\mathrm{y} = 0.45$ (top figure) and $\mathrm{x} 
    = 1.75$ (bottom figure). The Galerkin formulation produces negative 
    concentrations for the product that is extended over significant 
    region, and the violation is not mere numerical noise. It should 
    also be noted the concentration of the product under the clipping 
    procedure is different from that of the proposed computational 
    framework.
    \label{Fig:NN_Plume_subsurface_cC_sectional_views}}
\end{figure}

\begin{figure}
  \centering
  \psfrag{intcF}{$\int c_F(\mathrm{x},\mathrm{y}) \;\mathrm{dy}$}
  \psfrag{intcG}{$\int c_G(\mathrm{x},\mathrm{y}) \;\mathrm{dy}$}
  \psfrag{x}{$\mathrm{x}$}
  \subfigure{
    \includegraphics[clip,scale=0.45]{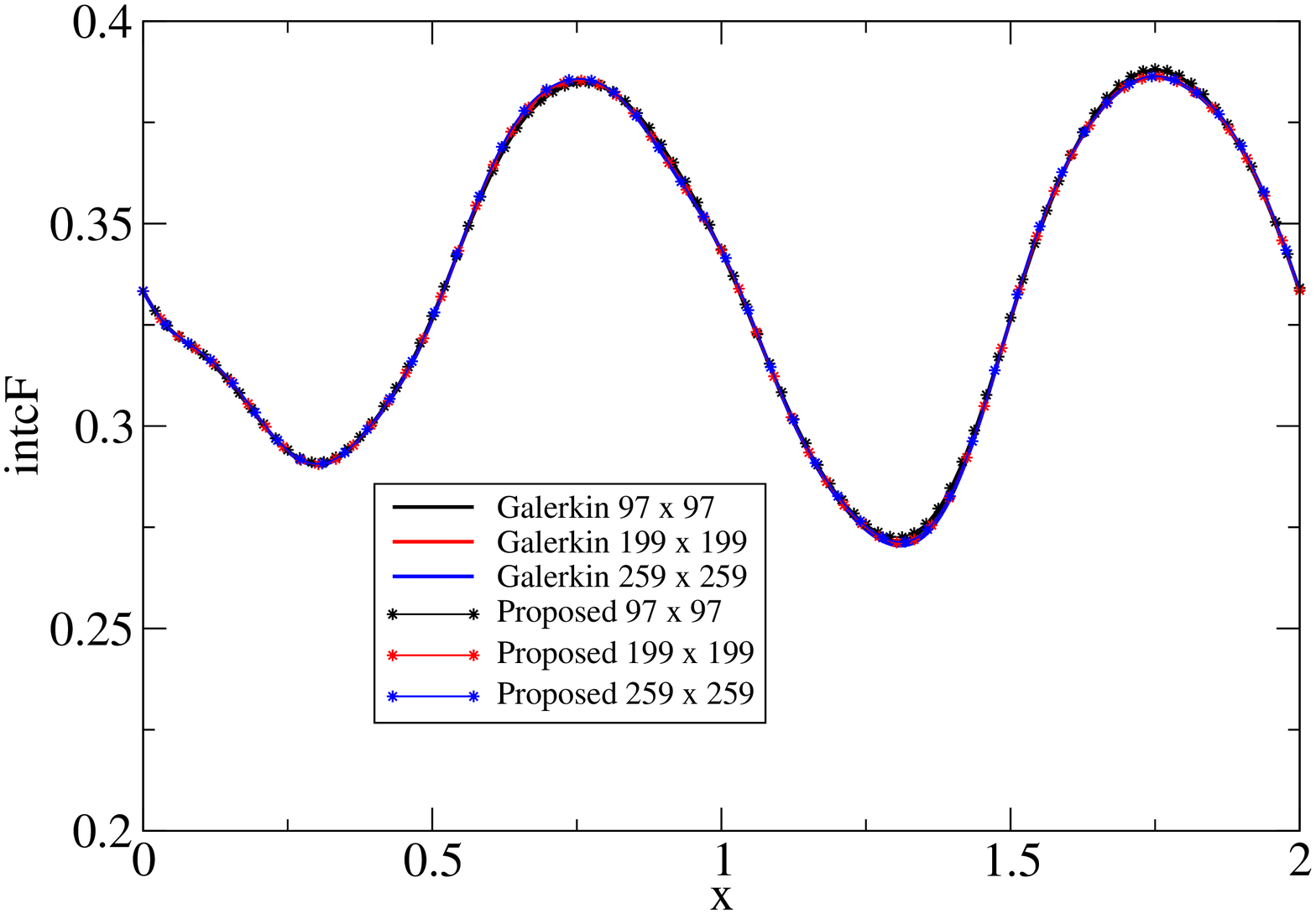}}
  \subfigure{
    \includegraphics[clip,scale=0.45]{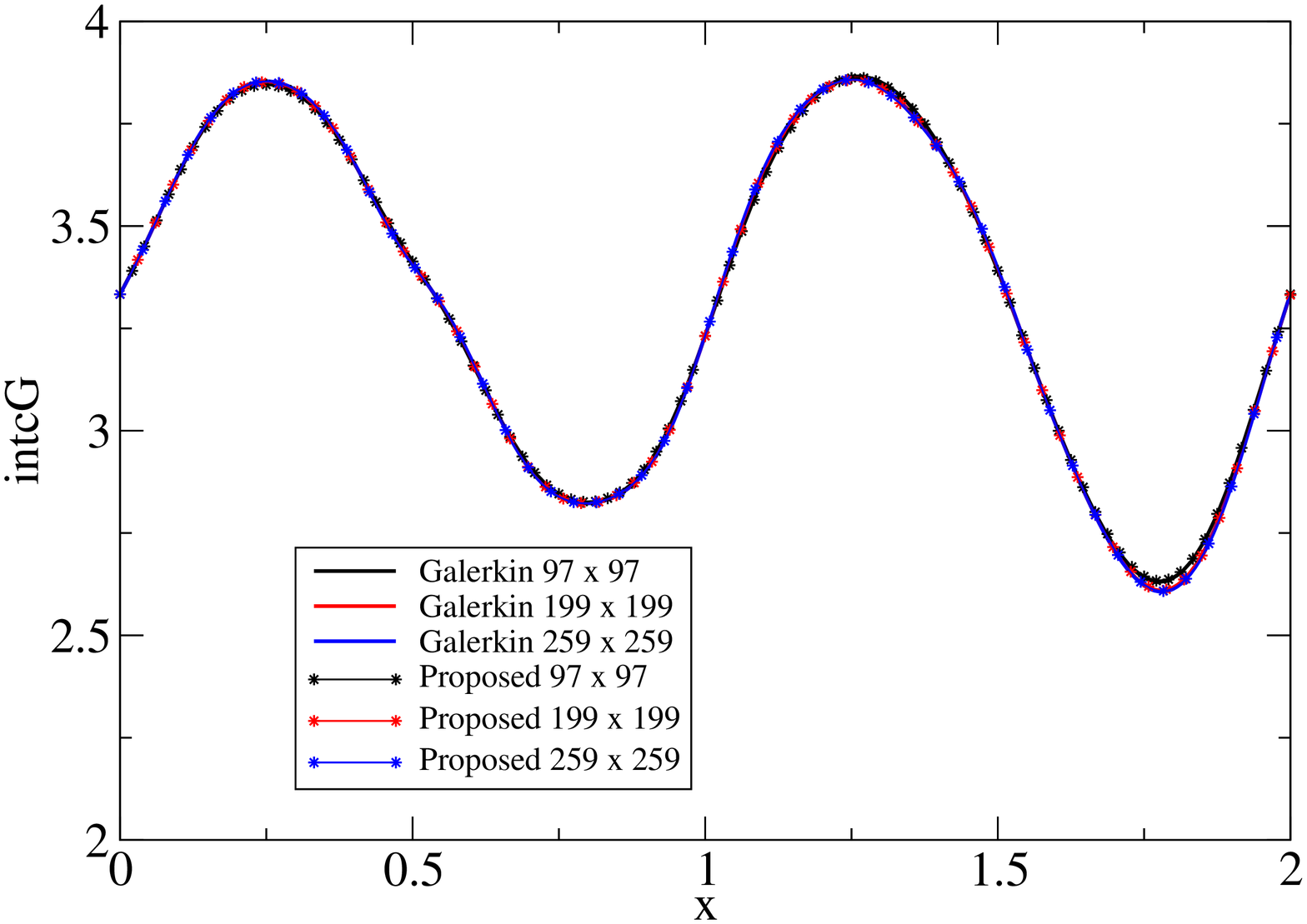}}
  \caption{Plume formation from boundary in a reaction tank: 
    This figure shows the variation of integrated concentration 
    of the invariants $F$ and $G$ with respect to $\mathrm{y}$ 
    along x (i.e., the variations of $\int c_F (\mathrm{x},
    \mathrm{y}) \; \mathrm{dy}$ and $\int c_G (\mathrm{x},
    \mathrm{y}) \; \mathrm{dy}$ along $\mathrm{x}$) for 
    various meshes using the Galerkin formulation, and 
    the proposed computational framework. The results 
    obtained under the clipping procedure are almost 
    the same as that of the Galerkin formulation and 
    the proposed computational framework, and hence 
    not shown in the figure. 
    \label{Fig:NN_Plume_subsurface_integral_FG_over_y}}
\end{figure}

\begin{figure}
  \centering
  \psfrag{intcC}{$\int c_C(\mathrm{x},\mathrm{y}) \;\mathrm{dy}$}
  \psfrag{x}{$\mathrm{x}$}
  \includegraphics[clip,scale=0.45]{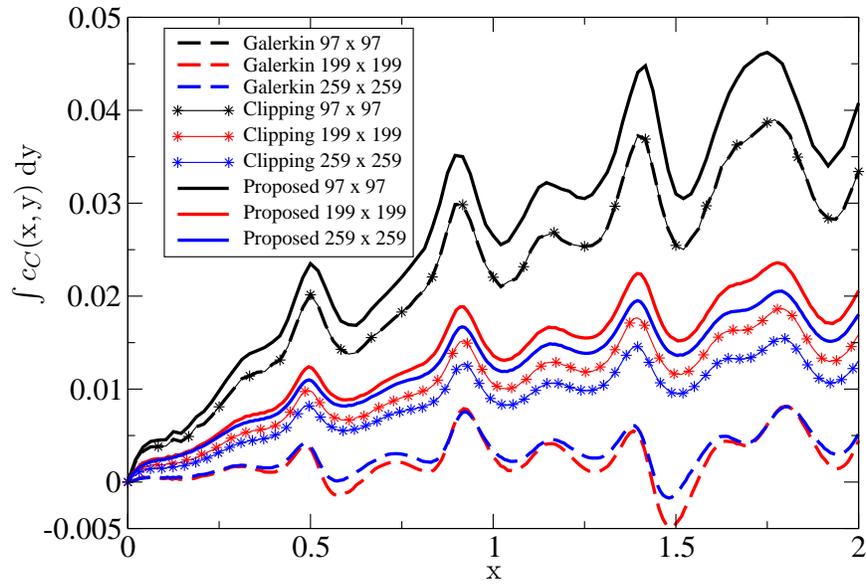}
  \caption{Plume formation from boundary in a reaction tank: 
    This figure shows the variation of integrated concentration 
    of the product $C$ with respect to $\mathrm{y}$ along x 
    (i.e., the variation of $\int c_C (\mathrm{x},\mathrm{y}) 
    \; \mathrm{dy}$ along $\mathrm{x}$) for various meshes using 
    the Galerkin formulation, the clipping procedure, and the 
    proposed computational framework. This figure clearly shows 
    that the Galerkin formulation predicts unphysical negative 
    values even for average quantities. In particular, 
    the Galerkin formulation predicts negative values for the 
    cross-sectional total concentration of the product for 
    significant portions along the $\mathrm{x}$-direction 
    even on fine computational grids. 
    \label{Fig:NN_Plume_subsurface_integral_C_over_y}}
\end{figure}

\begin{figure}
  \centering
  \subfigure{
    \includegraphics[clip,scale=0.9]{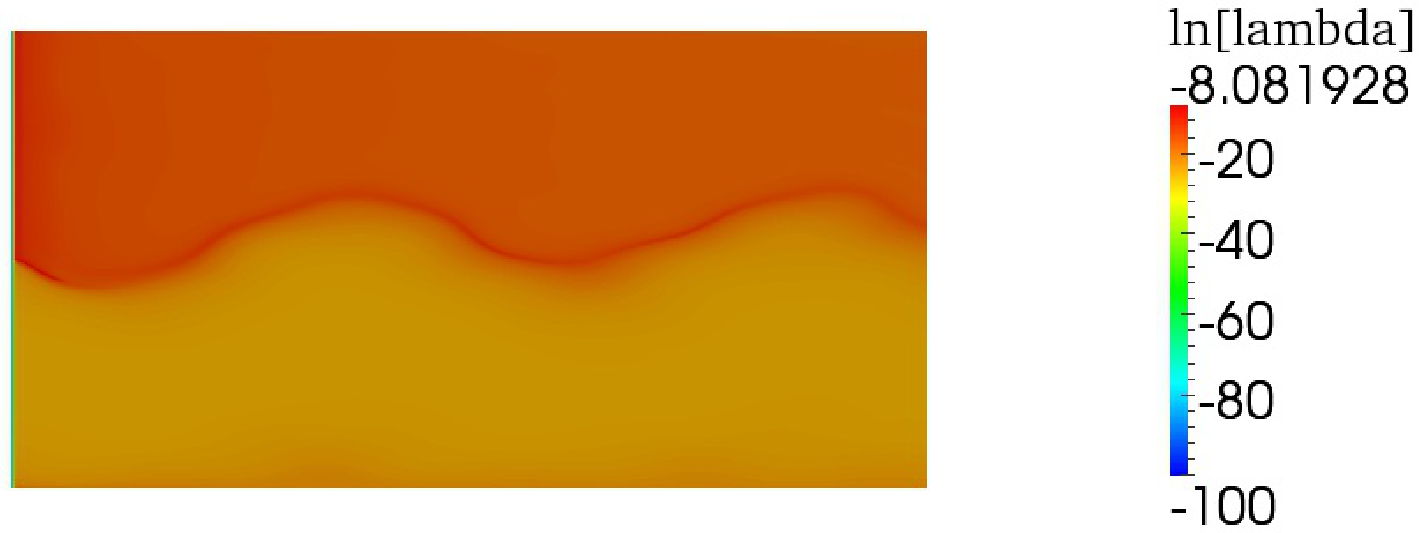}}
  \subfigure{
    \includegraphics[clip,scale=0.9]{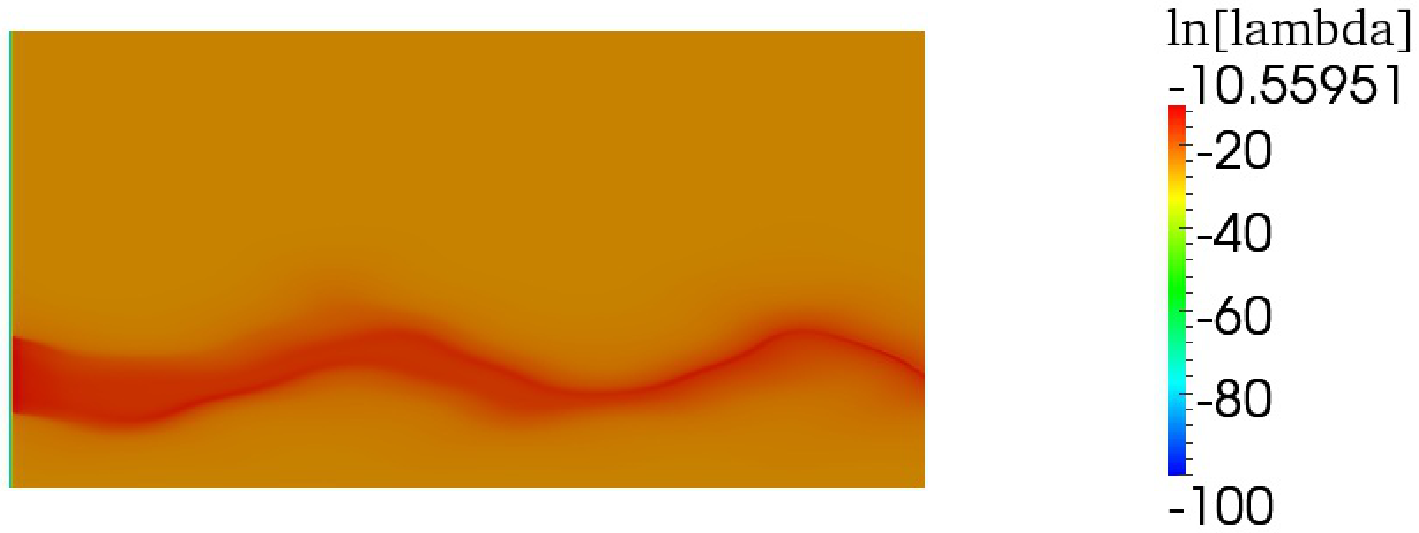}}
  \caption{Plume formation from boundary in a reaction tank: This 
    figure shows the contours of the natural logarithm of the 
    Lagrange multipliers corresponding to the lower bound 
    constraints (top) and the upper bound constraints (bottom) 
    for the invariant $F$. The proposed computational framework 
    is employed, which solves a quadratic programming problems to 
    obtain the concentration of invariant $F$. It should be noted 
    that the Lagrange multipliers will always be non-negative, and 
    hence the natural logarithm of the Lagrange multipliers will 
    be (extended) real numbers. For visualization purposes, we 
    have set the natural logarithm of zero to be $-100$. 
    \label{Fig:NN_Plume_subsurface_F_Lagrange_contours}}
\end{figure}

\begin{figure}
  \centering
  \subfigure{
    \includegraphics[clip,scale=0.9]{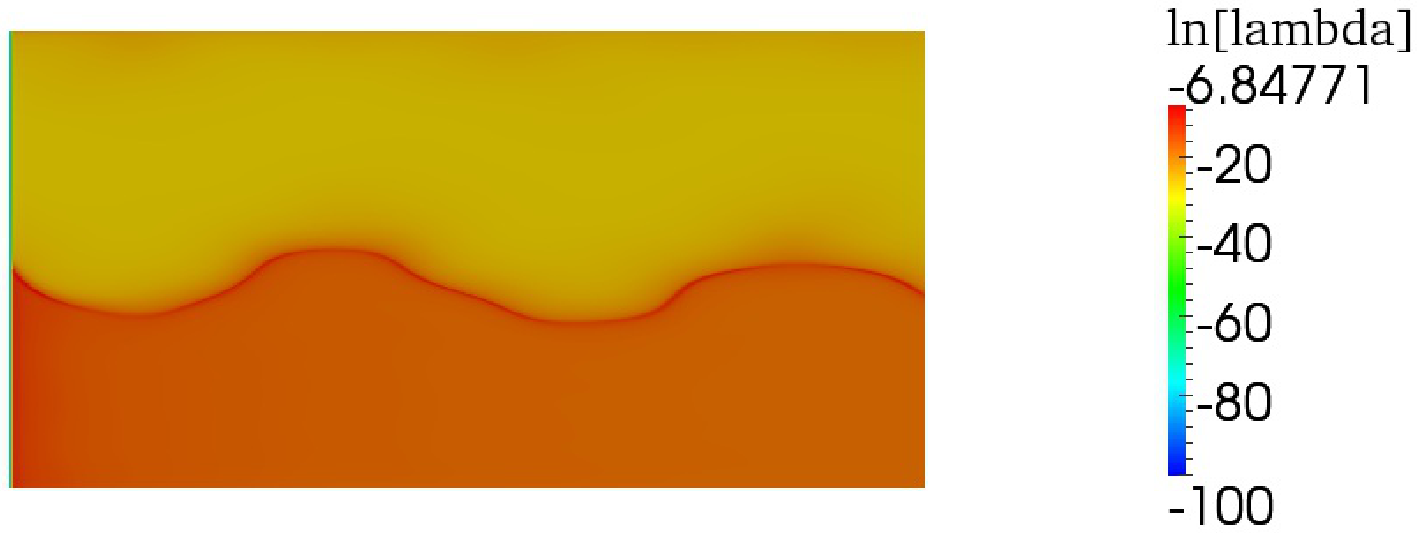}}
  \subfigure{
    \includegraphics[clip,scale=0.9]{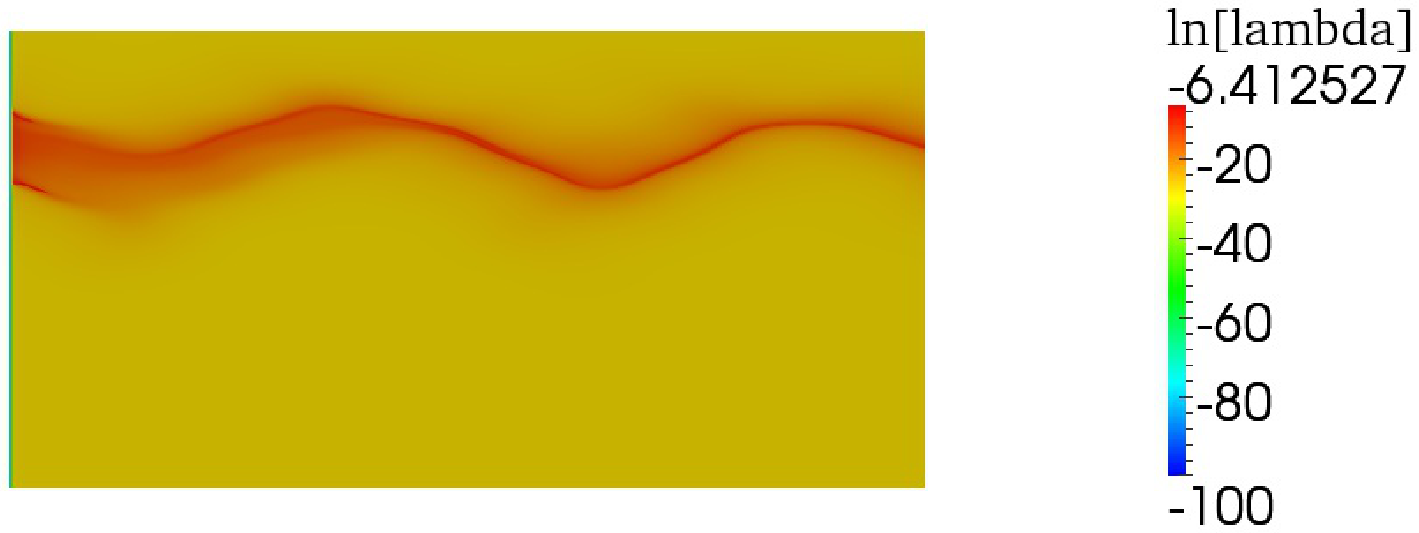}}
  \caption{Plume formation from boundary in a reaction tank: This 
    figure shows the contours of the natural logarithm of the 
    Lagrange multipliers corresponding to the lower bound 
    constraints (top) and the upper bound constraints (bottom) 
    for the invariant $G$. The proposed computational framework 
    is employed, which solves a quadratic programming problems to 
    obtain the concentration of invariant $G$. It should be noted 
    that the Lagrange multipliers will always be non-negative, and 
    hence the natural logarithm of the Lagrange multipliers will 
    be (extended) real numbers. For visualization purposes, we 
    have set the natural logarithm of zero to be $-100$. 
    \label{Fig:NN_Plume_subsurface_G_Lagrange_contours}}
\end{figure}


\begin{figure}
  \psfrag{A1}{$A \, , \, f^1_A$}
  \psfrag{A2}{$A \, , \, f^2_A$}
  \psfrag{B1}{$B \, , \, f^1_B$}
  \psfrag{B2}{$B \, , \, f^2_B$}
  \psfrag{xo}{$\mathbf{x}_{\mathrm{origin}}$}
  \psfrag{x1}{$\mathbf{x}_{1} = (\mathrm{x}_{1},\mathrm{y}_{1})$}
  \psfrag{x2}{$\mathbf{x}_{2} = (\mathrm{x}_{2},\mathrm{y}_{2})$}
  \psfrag{x3}{$\mathbf{x}_{3} = (\mathrm{x}_{3},\mathrm{y}_{3})$}
  \psfrag{x4}{$\mathbf{x}_{4} = (\mathrm{x}_{4},\mathrm{y}_{4})$}
  \psfrag{Lx}{$L_x$}
  \psfrag{Ly}{\rotatebox{-90}{$L_y$}}
  \psfrag{BC1}{\rotatebox{90}{$c_A^{\mathrm{p}}(\mathbf{x}) = 0$ , 
    $c_B^{\mathrm{p}}(\mathbf{x}) = 0$, $c_C^{\mathrm{p}}(\mathbf{x}) = 0$}}
  \psfrag{BC2}{$c_A^{\mathrm{p}}(\mathbf{x}) = 0$ , 
    $c_B^{\mathrm{p}}(\mathbf{x}) = 0$, $c_C^{\mathrm{p}}(\mathbf{x}) = 0$}
  \psfrag{BC3}{\rotatebox{-90}{$c_A^{\mathrm{p}}(\mathbf{x}) = 0$ , 
    $c_B^{\mathrm{p}}(\mathbf{x}) = 0$, $c_C^{\mathrm{p}}(\mathbf{x}) = 0$}}
  \psfrag{BC4}{$c_A^{\mathrm{p}}(\mathbf{x}) = 0$ , 
    $c_B^{\mathrm{p}}(\mathbf{x}) = 0$, $c_C^{\mathrm{p}}(\mathbf{x}) = 0$}
  \includegraphics[scale=0.9]{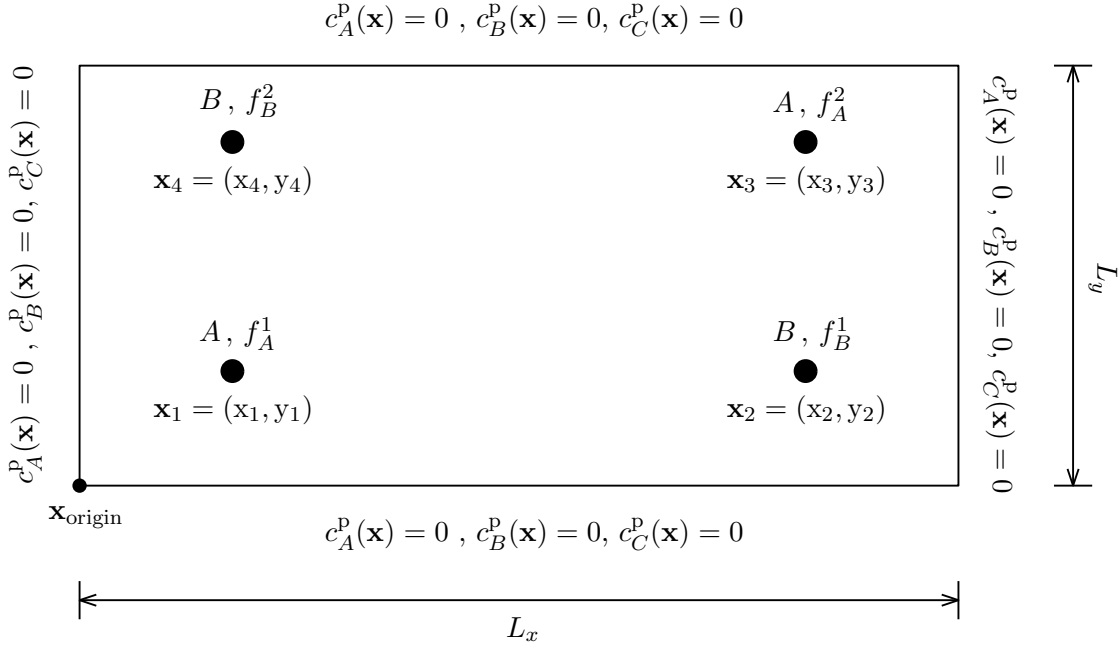}
  \caption{Plume formation due to multiple stationary point sources: 
    A pictorial description of boundary value problem. The 
    computational domain is a rectangle with origin at $\mathbf{x}_
    {\mathrm{origin}} = (0,0)$. The length and width of the rectangular 
    domain is $L_x = 2$ and $L_y = 1$. Continuous point sources of 
    chemical species $A$ and $B$ are located at $\mathbf{x}_{1} = 
    (0.2,0.3)$, $\mathbf{x}_{2} = (1.6,0.3)$, $\mathbf{x}_{3} = (1.6,0.7)$, 
    and $\mathbf{x}_{4} = (0.2,0.7)$ release reactants at a constant rate 
    of $f^1_A = 0.1$, $f^2_A = 0.05$, $f^1_B = 0.1$, and $f^2_B = 0.1$ 
    moles. The Dirichlet boundary conditions $c_A^{\mathrm{p}}(\mathbf{x})$ , 
    $c_B^{\mathrm{p}}(\mathbf{x})$, and $c_C^{\mathrm{p}}(\mathbf{x})$ 
    are prescribed to be zero on the boundary.   
    \label{Fig:Plume_Multiple_PointSources}}
\end{figure}

\begin{figure}
  \centering
  \includegraphics[clip,scale=0.75]{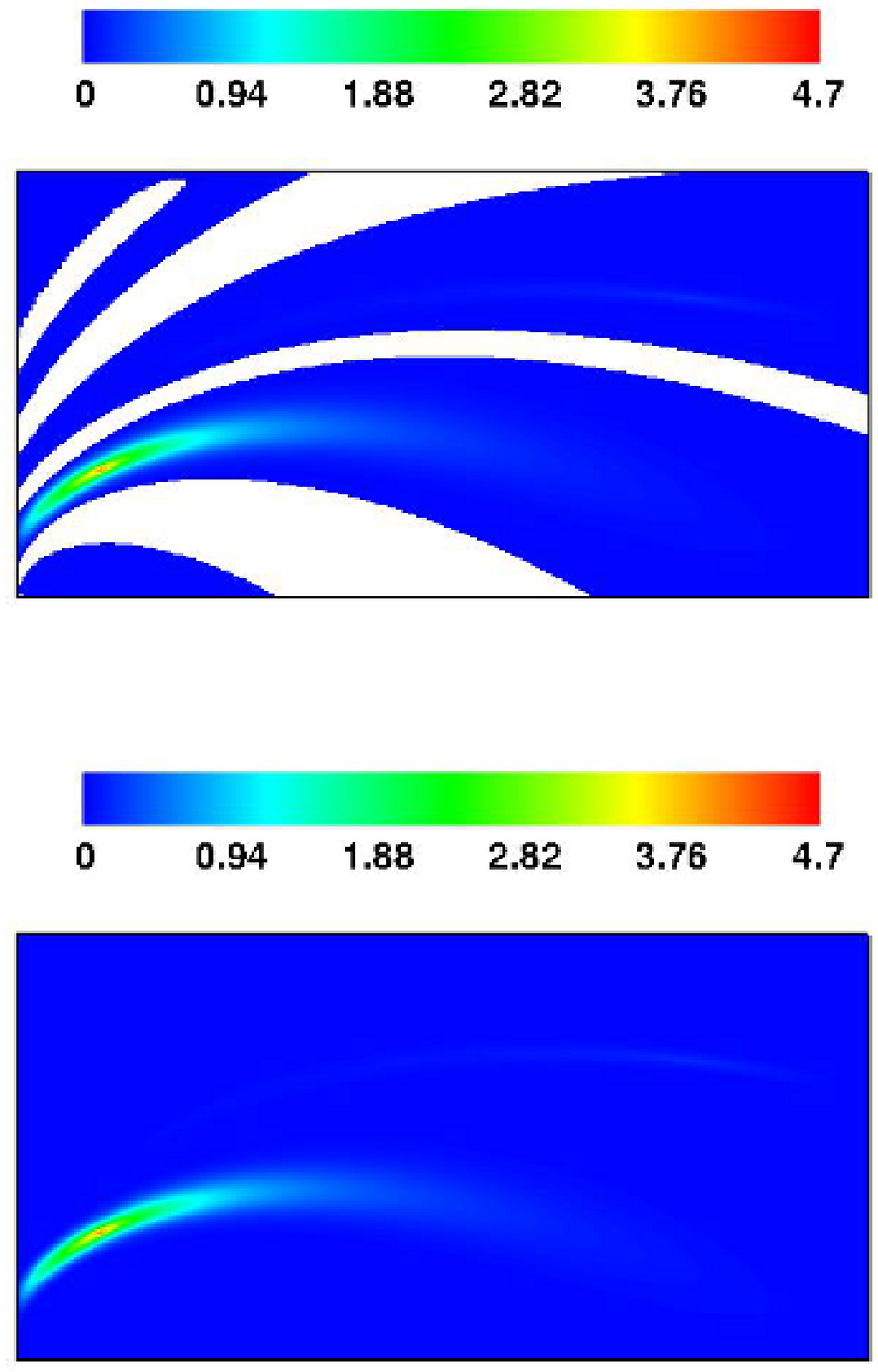}
  \caption{Plume formation due to multiple stationary point sources: 
    This figure shows the \emph{steady-state} contours of concentration 
    of the invariant $F$ under the Galerkin formulation (top figure), 
    and the non-negative formulation (bottom figure). The minimum and 
    maximum concentrations of $F$ under the Galerkin formulation are 
    respectively $-0.0498$ and $4.7100$, and the corresponding values 
    under the non-negative formulation are respectively $0$ and $4.7088$. 
    The rotated heterogeneous diffusivity tensor is employed in the 
    numerical simulation. The results are generated using $201 \times 
    201$ structured mesh using four-node quadrilateral elements. The 
    regions with negative concentration are indicated in white color. 
    \label{Fig:NN_Plume_PointSources_contour_F}}
\end{figure}

\begin{figure}
  \centering
  \includegraphics[clip,scale=0.75]{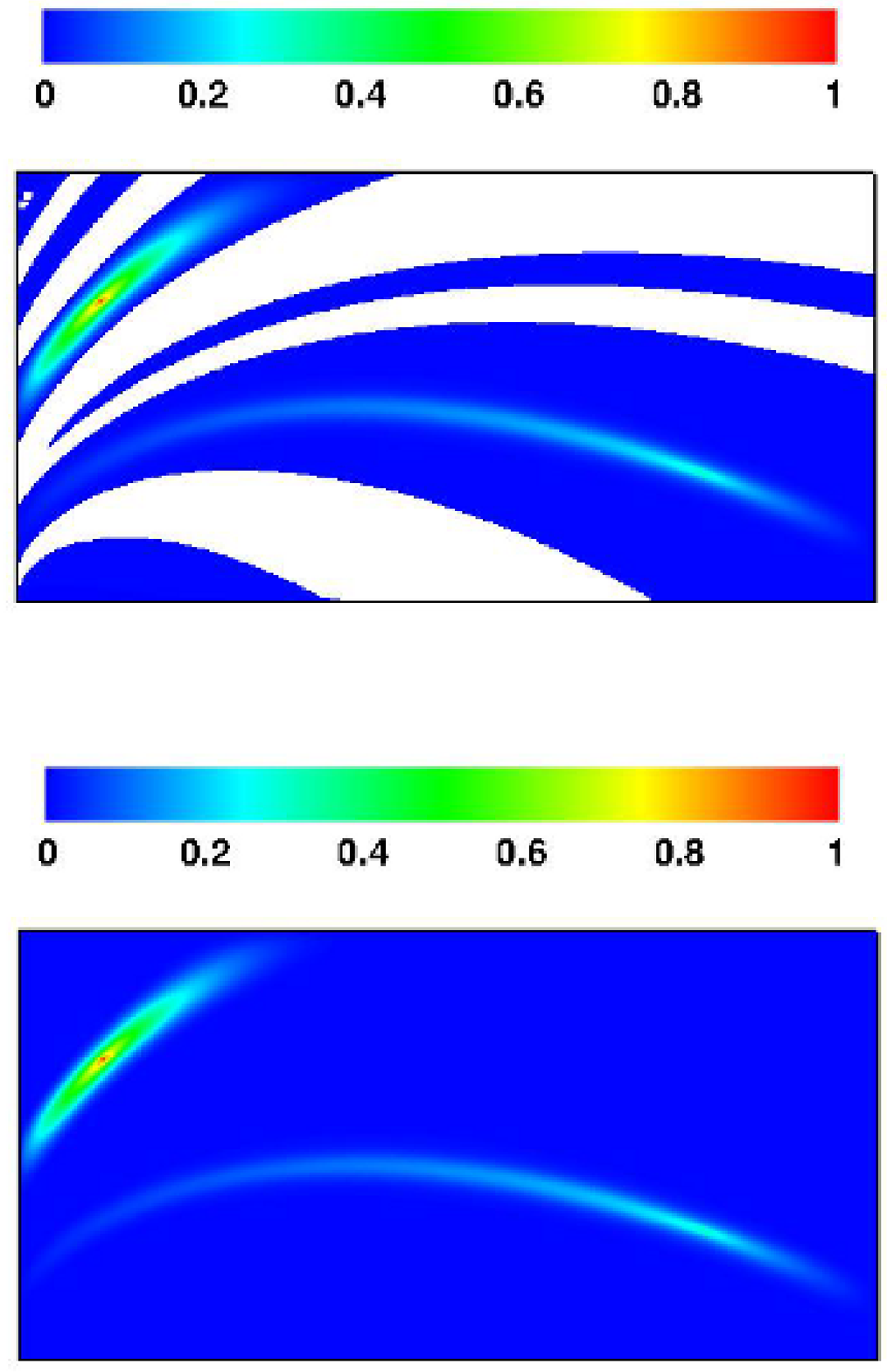}
  \caption{Plume formation due to multiple stationary point sources: 
    This figure shows the \emph{steady-state} contours of concentration 
    of the invariant $G$ under the Galerkin formulation (top figure), 
    and the non-negative formulation (bottom figure). The minimum and 
    maximum concentrations of $G$ under the Galerkin formulation are 
    respectively $-0.0158$ and $0.9977$, and the corresponding values 
    under the non-negative formulation are respectively $0$ and $0.9969$. 
    The rotated heterogeneous diffusivity tensor is employed in the 
    numerical simulation. The results are generated using $201 \times 
    201$ structured mesh using four-node quadrilateral elements. 
    The regions with negative concentration are indicated in white 
    color. \label{Fig:NN_Plume_PointSources_contour_G}}
\end{figure}

\begin{figure}
  \centering
  \includegraphics[clip,scale=0.75]{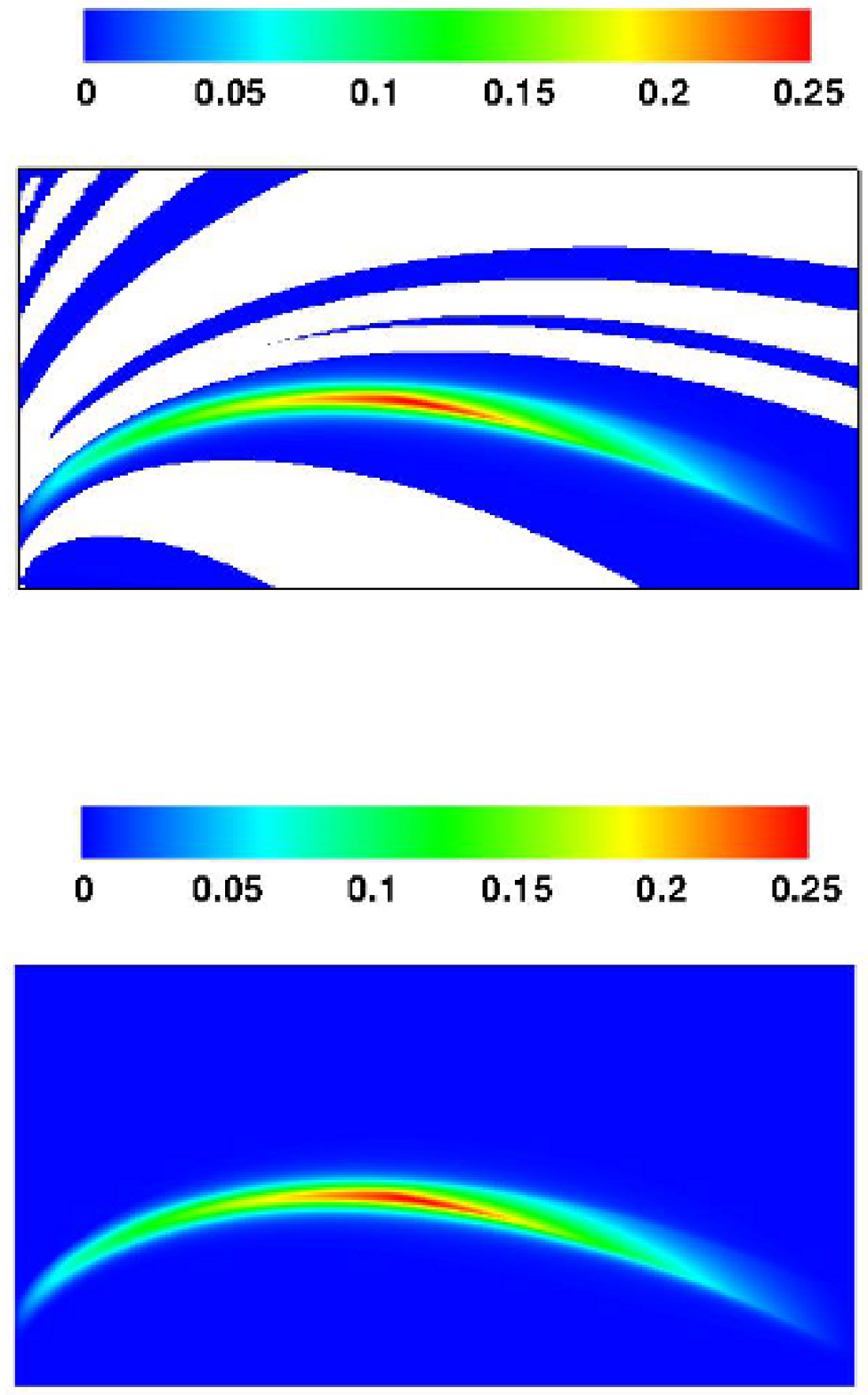}
  \caption{Plume formation due to multiple stationary point sources: 
    This figure shows the \emph{steady-state} contours of concentration 
    of product $C$ under the Galerkin formulation (top figure), and the 
    non-negative formulation (bottom figure). The minimum and maximum 
    concentrations of $C$ under the Galerkin formulation are respectively 
    $-0.0996$ and $0.2560$, and the corresponding values under the 
    non-negative formulation are respectively $0$ and $0.2564$. The 
    rotated heterogeneous diffusivity tensor is employed in the 
    numerical simulation. The results are generated using $201 
    \times 201$ structured mesh using four-node quadrilateral 
    elements. The regions with negative concentration are 
    indicated in white color. 
    \label{Fig:NN_Plume_PointSources_contour_C}}
\end{figure}

\begin{figure}
  \centering
  \psfrag{intcC}{$\int c_C(\mathrm{x},\mathrm{y}) \; \mathrm{d} \mathrm{y}$}
  \includegraphics[clip,scale=0.5]{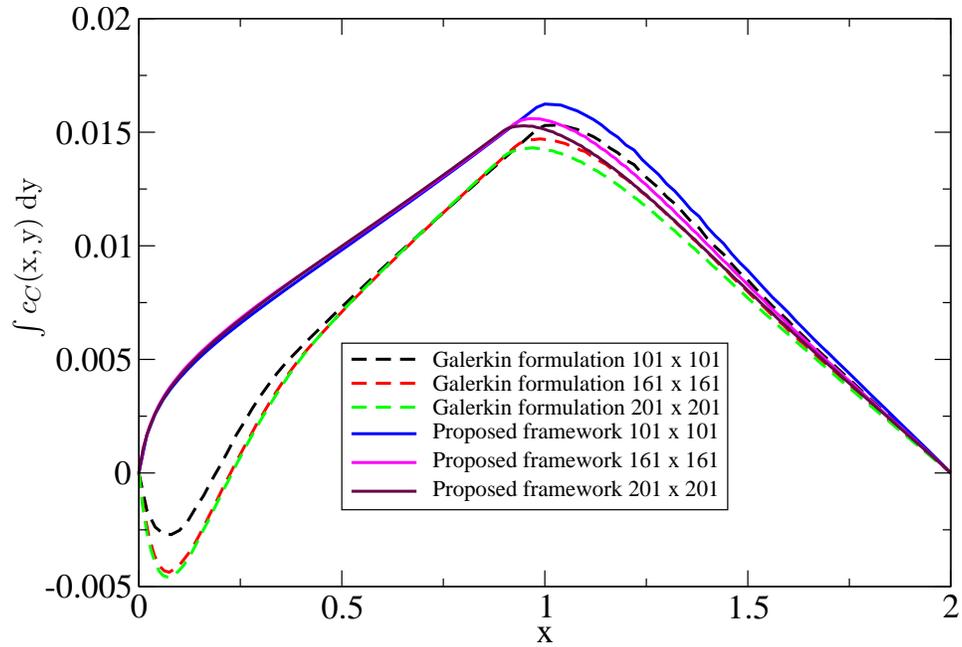}
  \caption{Plume formation due to multiple stationary point sources: 
    This figure shows the variation of integrated steady-state 
    concentration of the product $C$ with respect to $\mathrm{y}$ 
    along x (i.e., the variation of $\int c_C (\mathrm{x},\mathrm{y}) 
    \; \mathrm{dy}$ along $\mathrm{x}$) for various meshes using 
    \emph{the Galerkin formulation} and \emph{the proposed computational 
      framework}. Three different structured meshes using four-node 
    quadrilateral elements are employed, which are indicated as $101 
    \times 101$, $161 \times 161$ and $201 \times 201$. The rotated 
    heterogeneous diffusivity tensor is employed in the numerical 
    simulation. \label{Fig:NN_Plume_PointSources_integral_C_over_y}}
\end{figure}

\begin{figure}
  \centering
  \includegraphics[clip,scale=0.5]{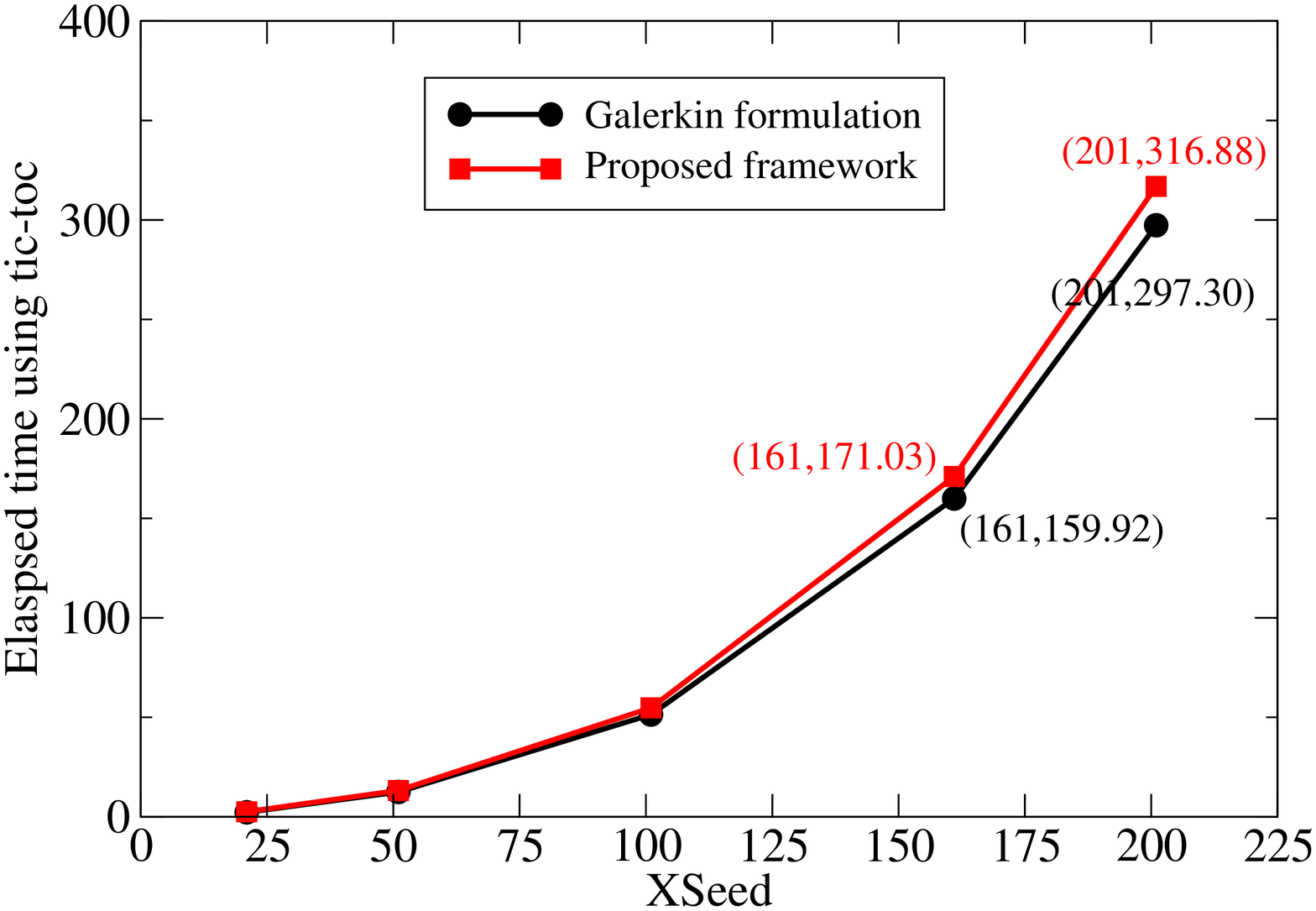}
  \caption{Plume formation due to multiple stationary point sources: 
    This figure compares the CPU timing of the Galerkin formulation 
    and the proposed computational framework. Note that XSeed denotes 
    the number of nodes along the x-direction, and the same number of 
    nodes are employed in the y-direction. In other words, a mesh with 
    XSeed = YSeed = 201 will have over 40,000 total number of 
    nodes. For each node, we have 5 unknowns -- two invariants, 
    two reactants, and the product. Note that we have employed 
    MATLAB's \textsf{interior-point-convex} algorithm to solve 
    the resulting quadratic programming problems. Each point 
    is generated by taking the average of five simulations. 
    \label{Fig:NN_Plume_PointSources_CPU_timing}}
\end{figure}


\begin{figure}
  \psfrag{A}{$A \, , \, c_A^{0}$}
  \psfrag{B1}{$A \, , \, c_A^{0} = 0$}
  \psfrag{B2}{$B \, , \, c_B^{0}$}
  \psfrag{IC}{Initial conditions for chemical species $A$ and $B$}
  \psfrag{xo}{$\mathbf{x}_{\mathrm{origin}}$}
  \psfrag{xs}{$\mathbf{x}_{\mathrm{slug}}$}
  \psfrag{Lx}{$L_x$}
  \psfrag{Ly}{\rotatebox{-90}{$L_y$}}
  \psfrag{dx}{$d_x$}
  \psfrag{dy}{\rotatebox{-90}{$d_y$}}
  \psfrag{BC1}{\rotatebox{90}{$c_A^{\mathrm{p}} = 0$ , $c_B^{\mathrm{p}}$}}
  \psfrag{BC2}{$c_A^{\mathrm{p}} = 0$ , $c_B^{\mathrm{p}}$}
  \psfrag{BC3}{\rotatebox{-90}{$c_A^{\mathrm{p}} = 0$ , $c_B^{\mathrm{p}}$}}
  \psfrag{BC4}{$c_A^{\mathrm{p}} = 0$ , $c_B^{\mathrm{p}}$}
  \includegraphics[scale=0.9]{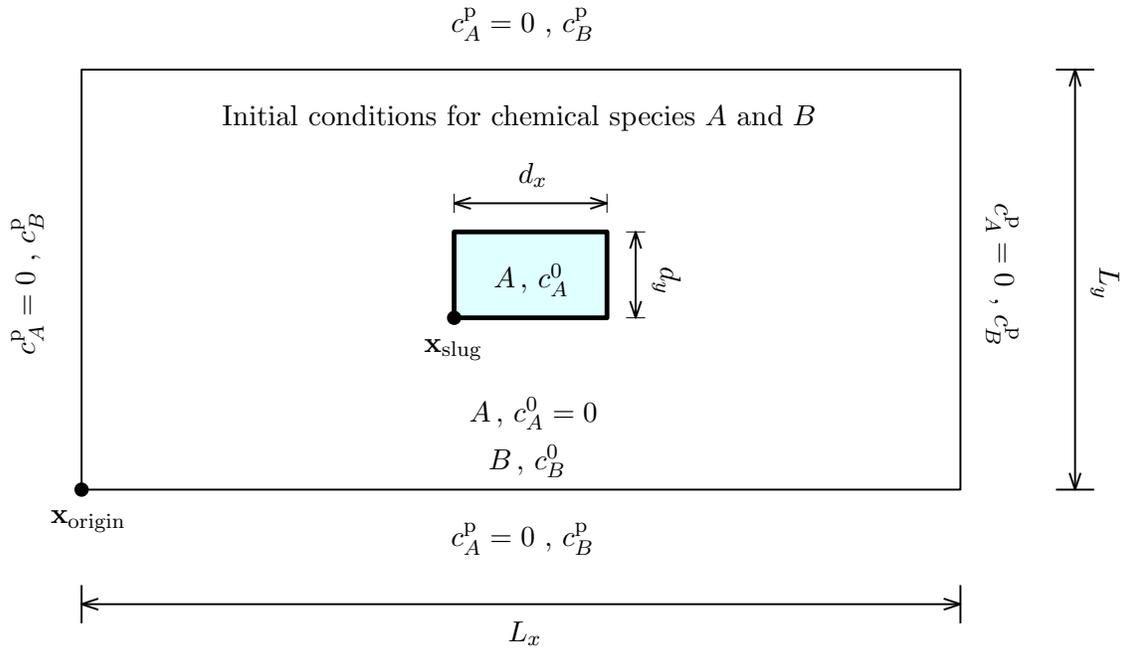}
  \caption{Diffusion and reaction of an initial slug: A pictorial 
  description of the initial boundary value problem. The slug of 
  dimensions $d_{x} = 2$ , $d_{y} = 1$ is located at $\mathbf{x}_{
  \mathrm{slug}} = (4,2)$ in the rectangular domain of length $L_x = 10$
  and breadth $L_y = 5$.
  \label{Fig:Plume_Slug}}
\end{figure}


\begin{figure}
  \centering
  \subfigure[$t = 0.05$]{\includegraphics[scale=0.35]
    {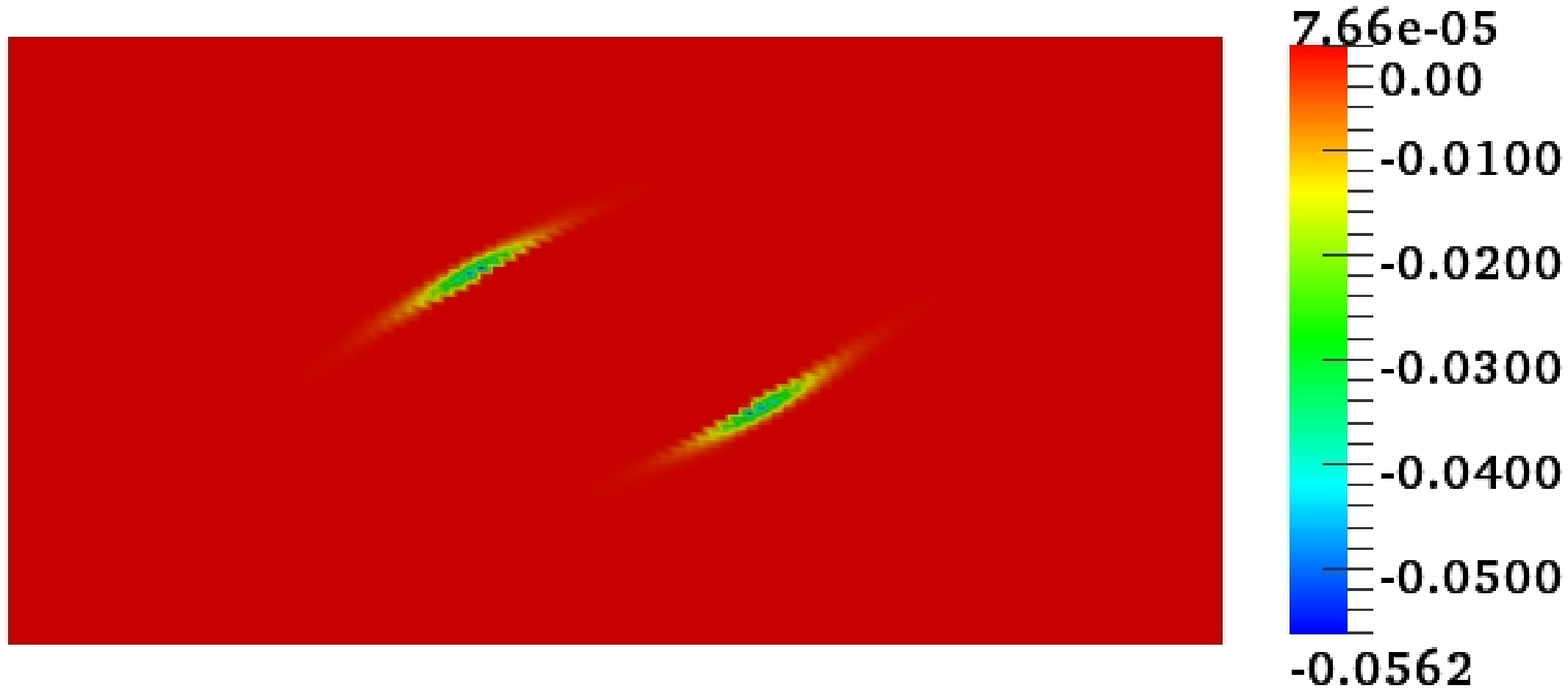}}
  \subfigure[$t = 0.05$]{\includegraphics[scale=0.35]
    {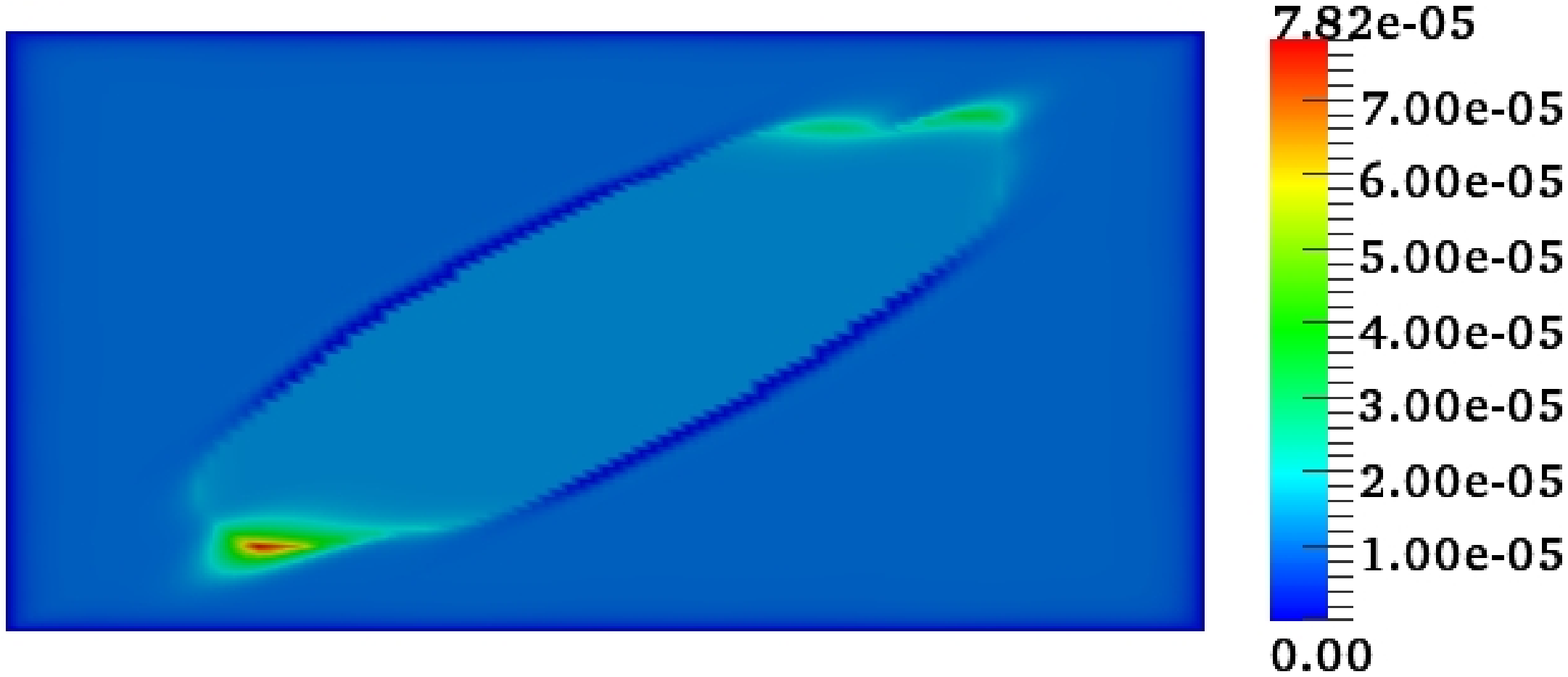}}
  \subfigure[$t = 0.1$]{\includegraphics[scale=0.35]
    {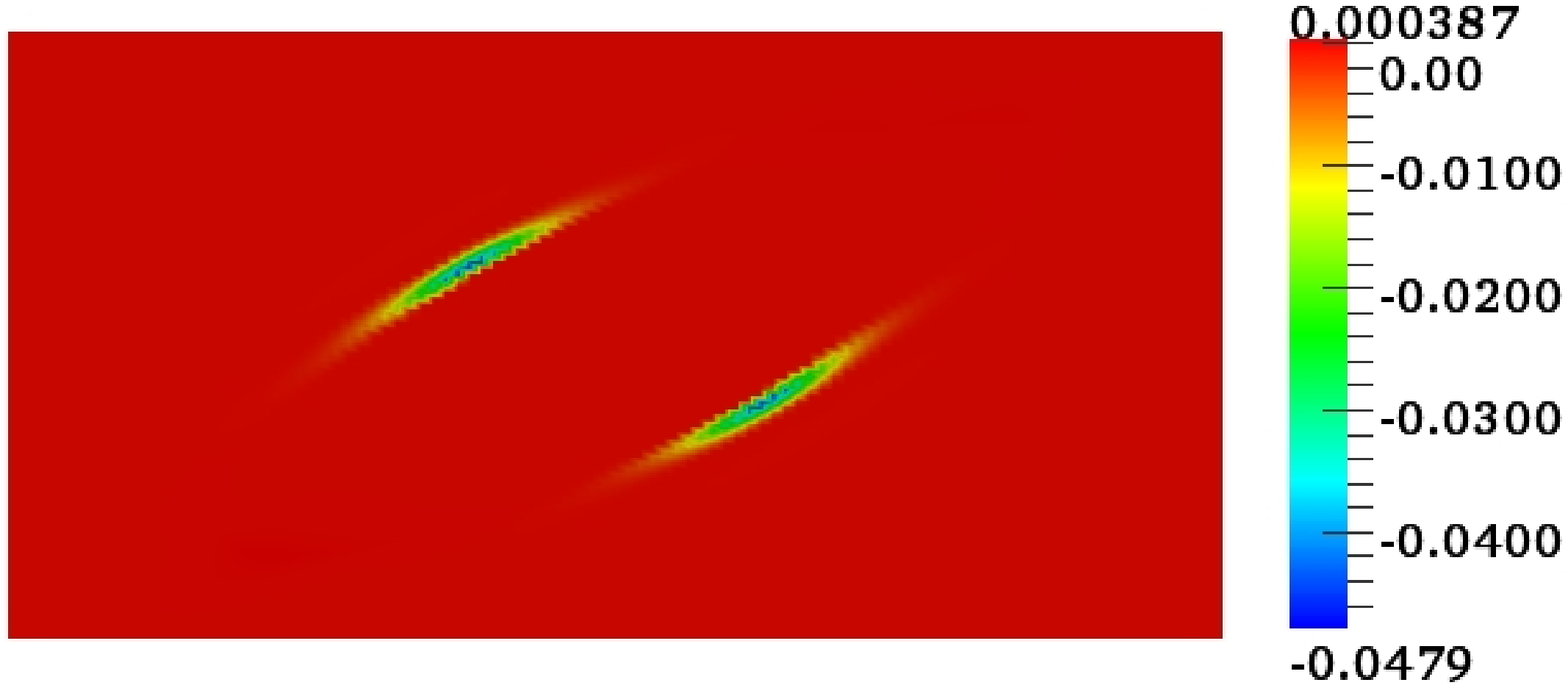}}
  \subfigure[$t = 0.1$]{\includegraphics[scale=0.35]
    {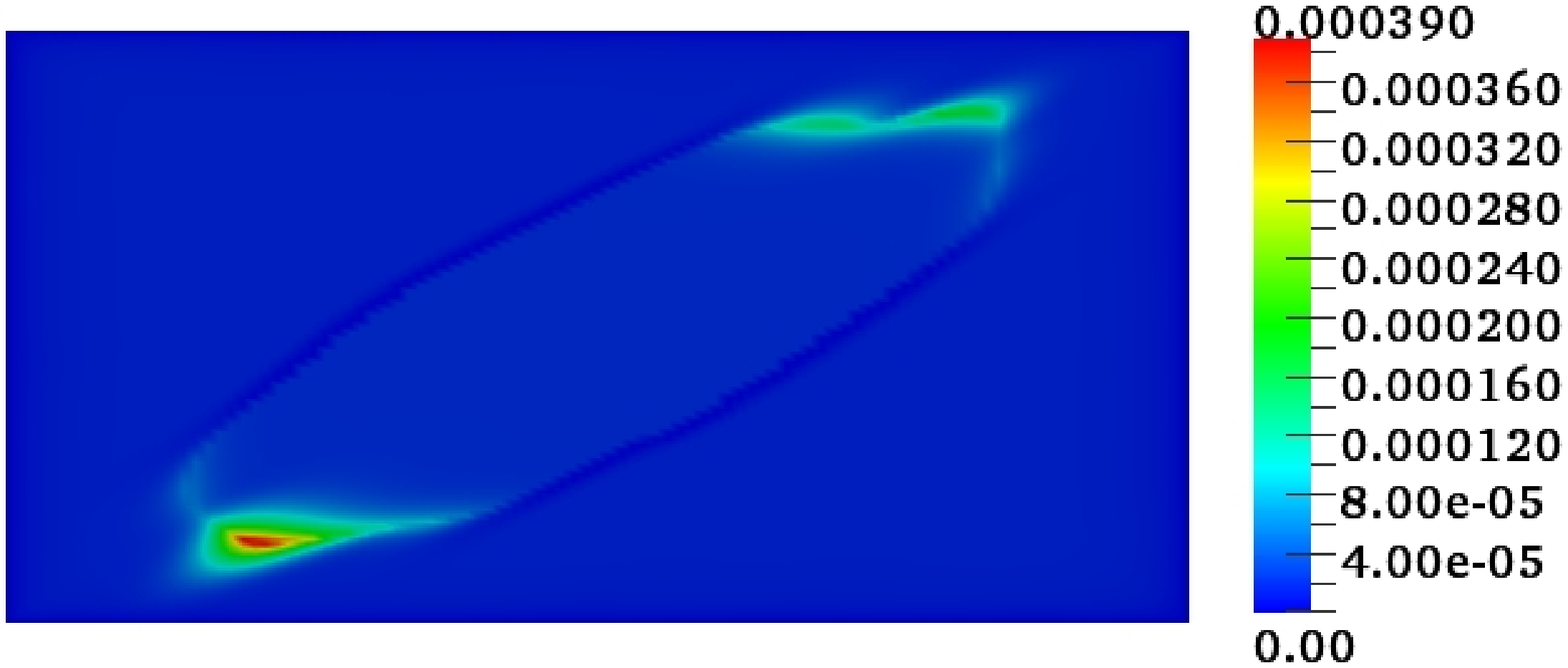}}
  \subfigure[$t = 0.5$]{\includegraphics[scale=0.35]
    {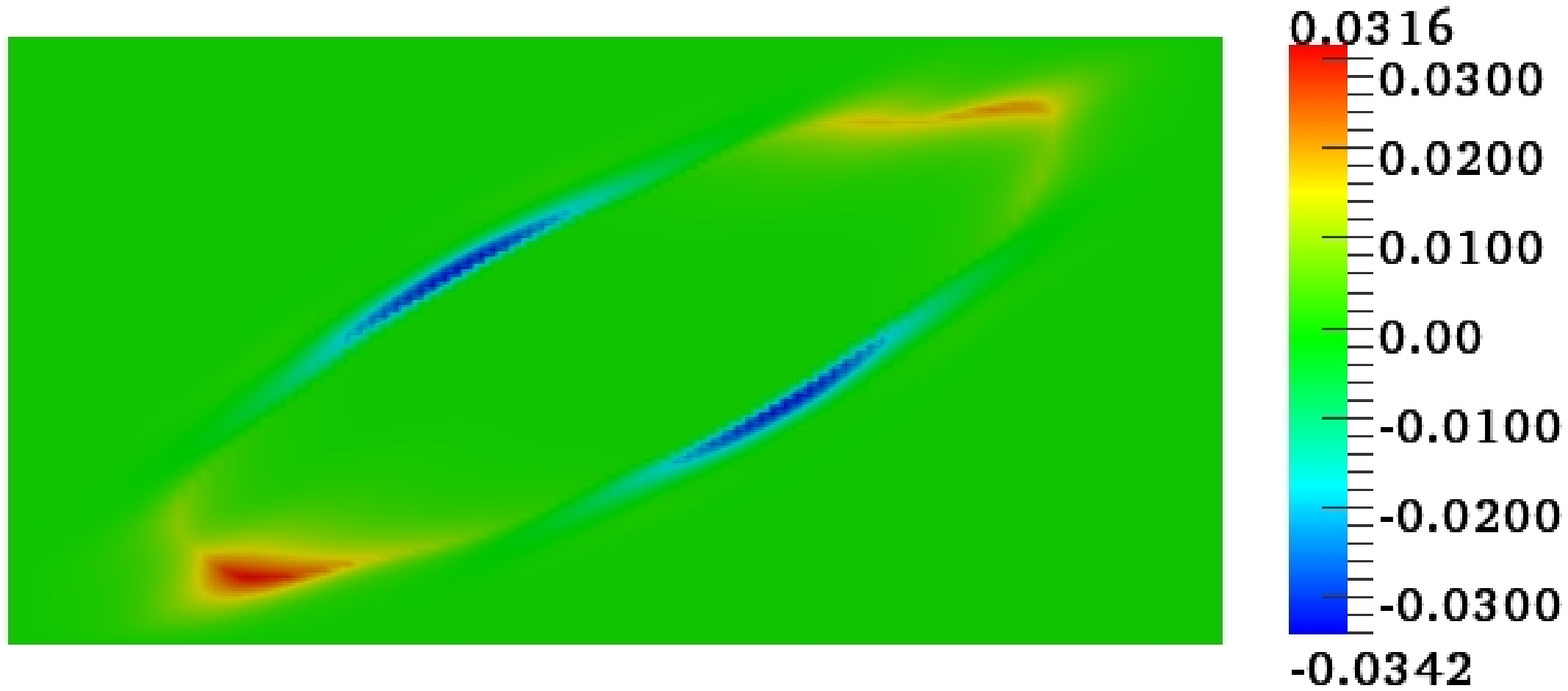}}
  \subfigure[$t = 0.5$]{\includegraphics[scale=0.35]
    {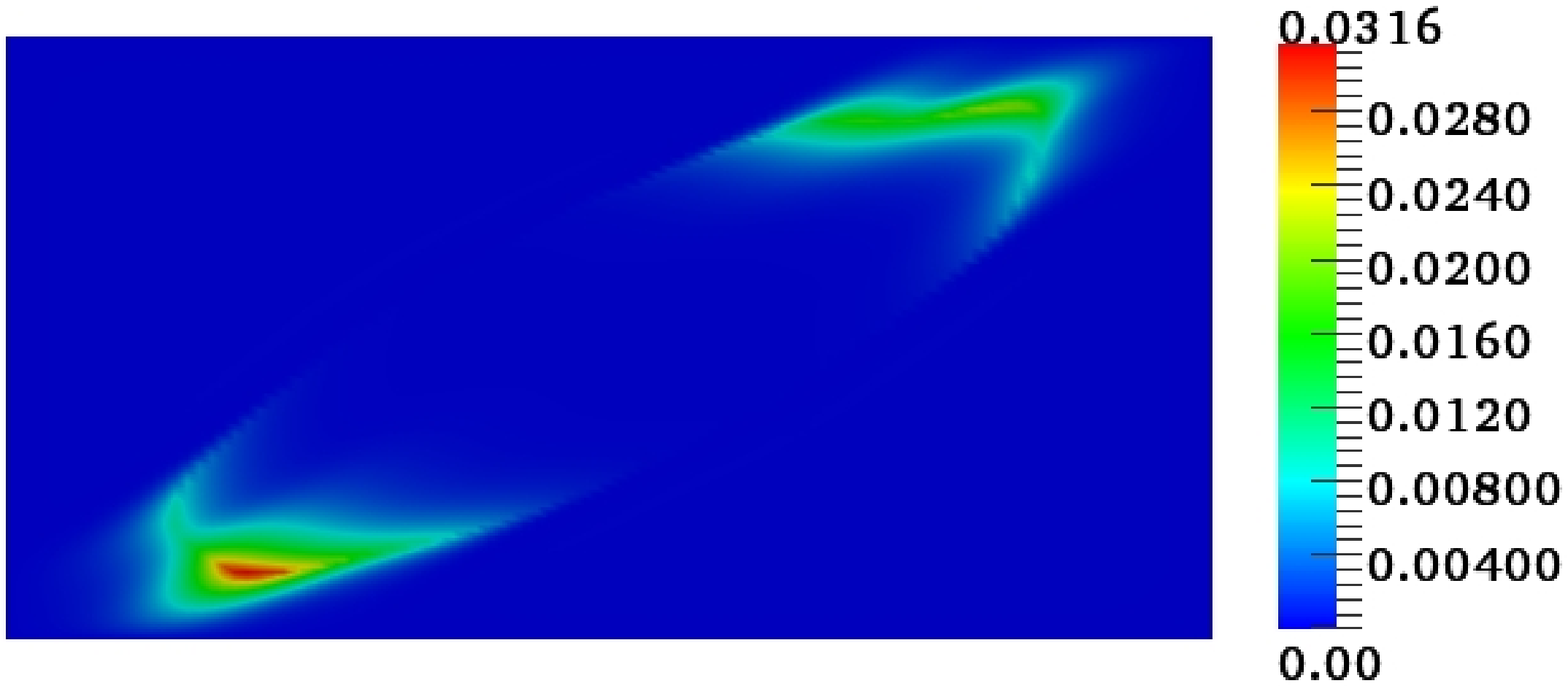}}
  \subfigure[$t = 1.0$]{\includegraphics[scale=0.35]
    {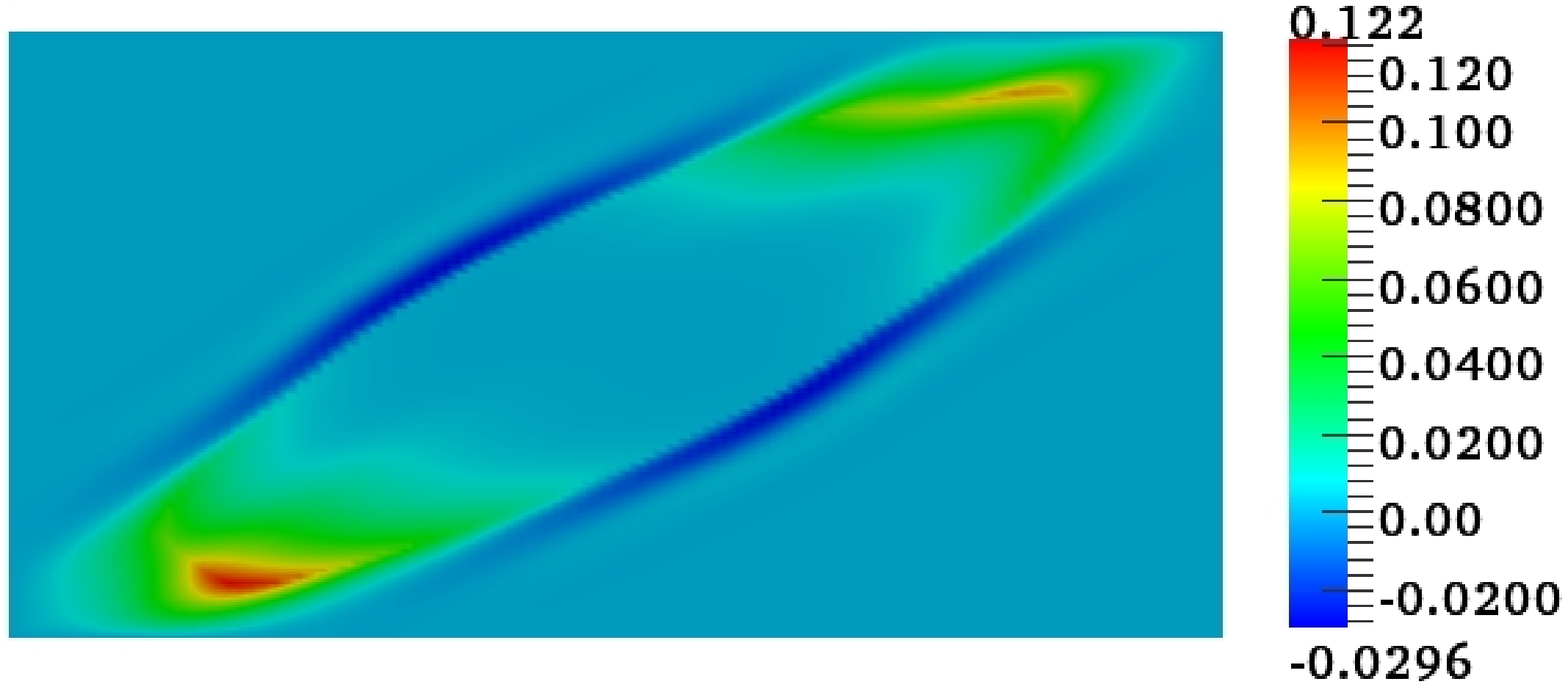}}
  \subfigure[$t = 1.0$]{\includegraphics[scale=0.35]
    {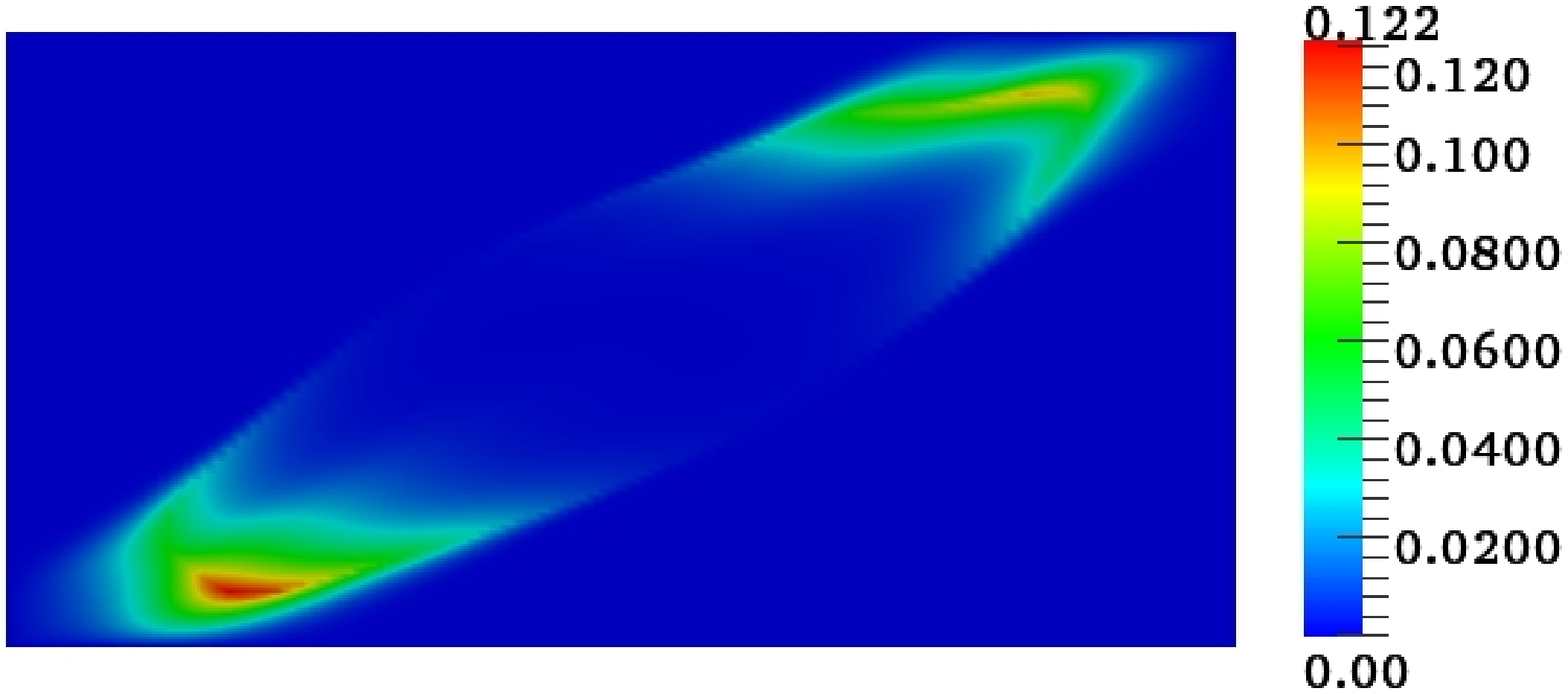}}
  \caption{Diffusion and reaction of an initial slug: This figure 
    shows the contours of the concentration of the product $C$ 
    under the Galerkin single-field formulation (left) and proposed 
    numerical framewotk (right) at various time levels. The time step 
    is taken as $\Delta t = 0.05 \; \mathrm{s}$. It is evident from above 
    figures that the Galerkin single-field formulation violated the non-negative 
    constraint, whereas the proposed numerical framework produced 
    physically meaningful values for concentration of the product $C$.  
    \label{Fig:NN_Slug_C_t_dot05}}
\end{figure}

\begin{figure}
  \centering
  \subfigure[$t = 1.0$]{\includegraphics[scale=0.35]
    {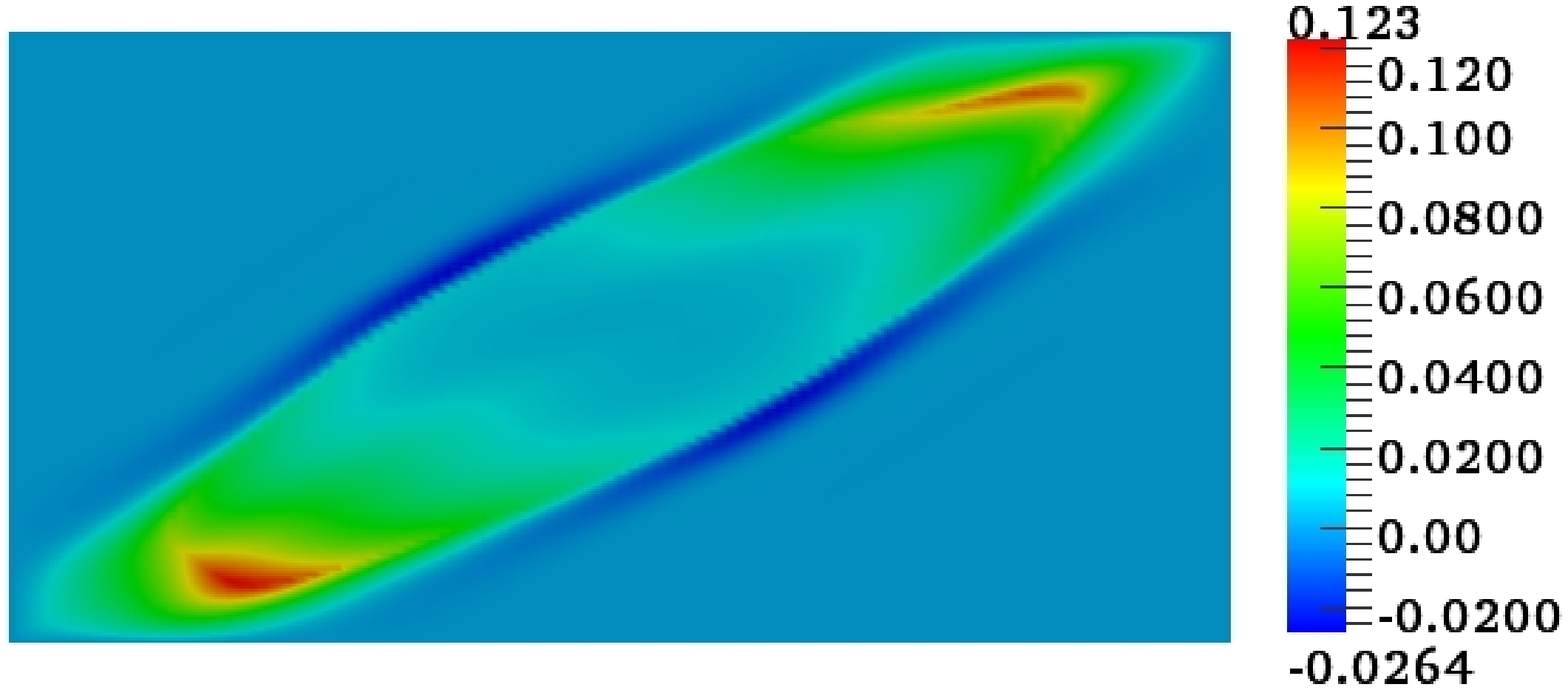}}
  \subfigure[$t = 1.0$]{\includegraphics[scale=0.35]
    {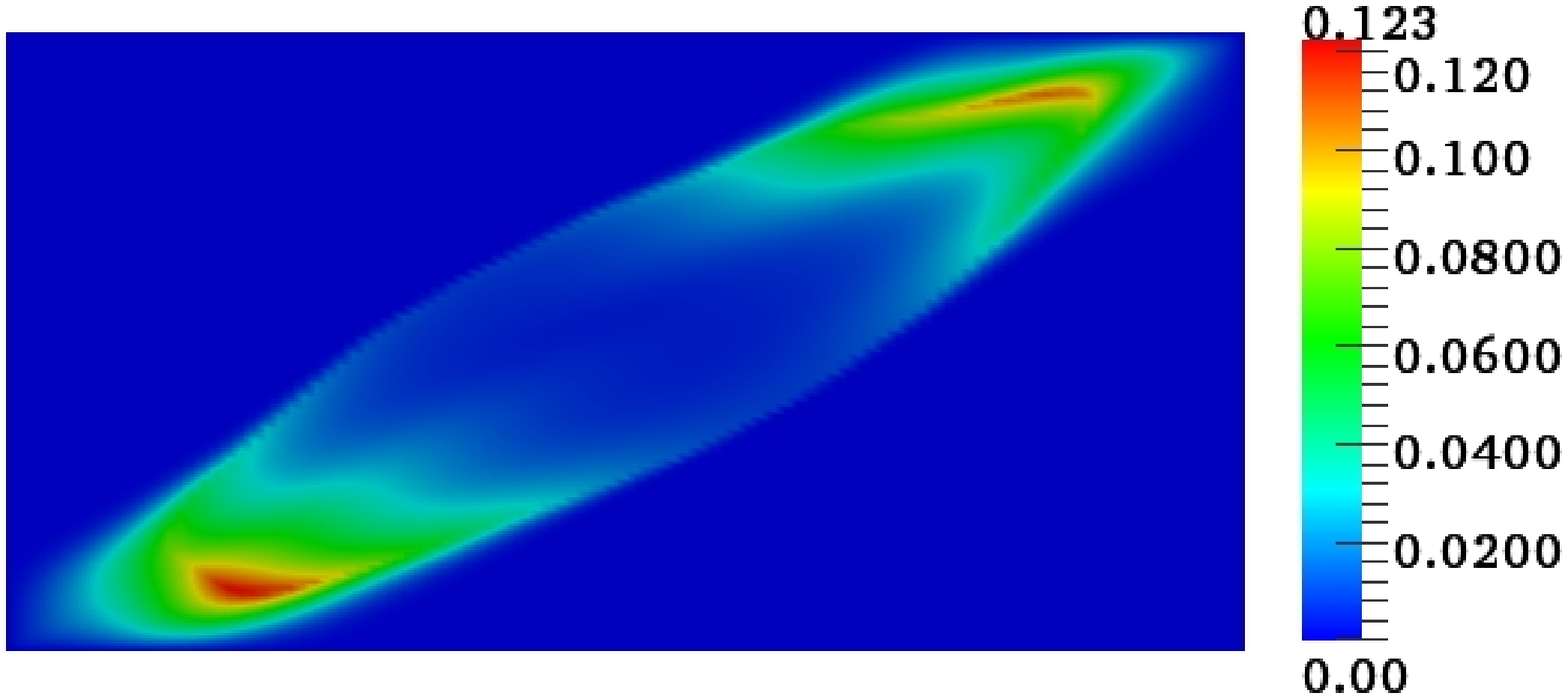}}
  \subfigure[$t = 5.0$]{\includegraphics[scale=0.35]
    {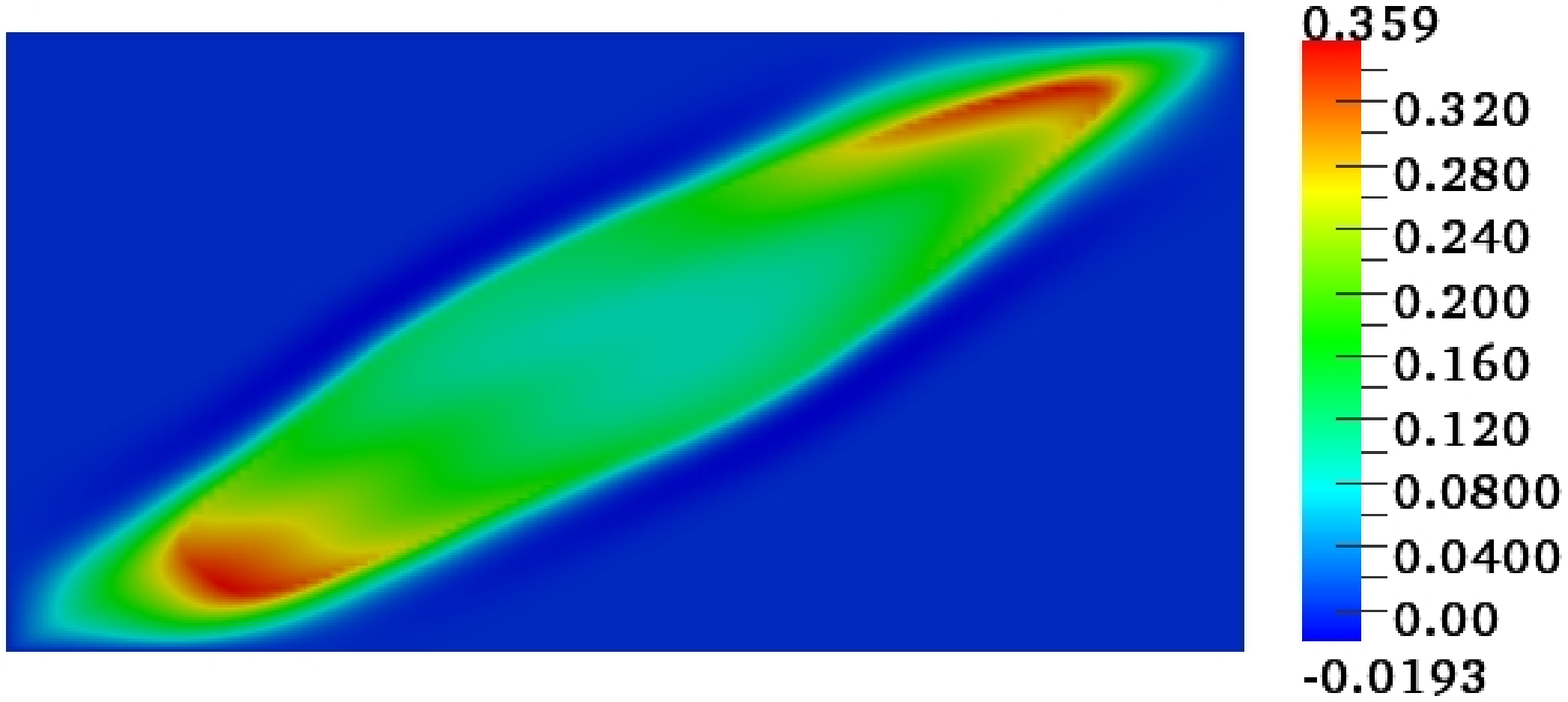}}
  \subfigure[$t = 5.0$]{\includegraphics[scale=0.35]
    {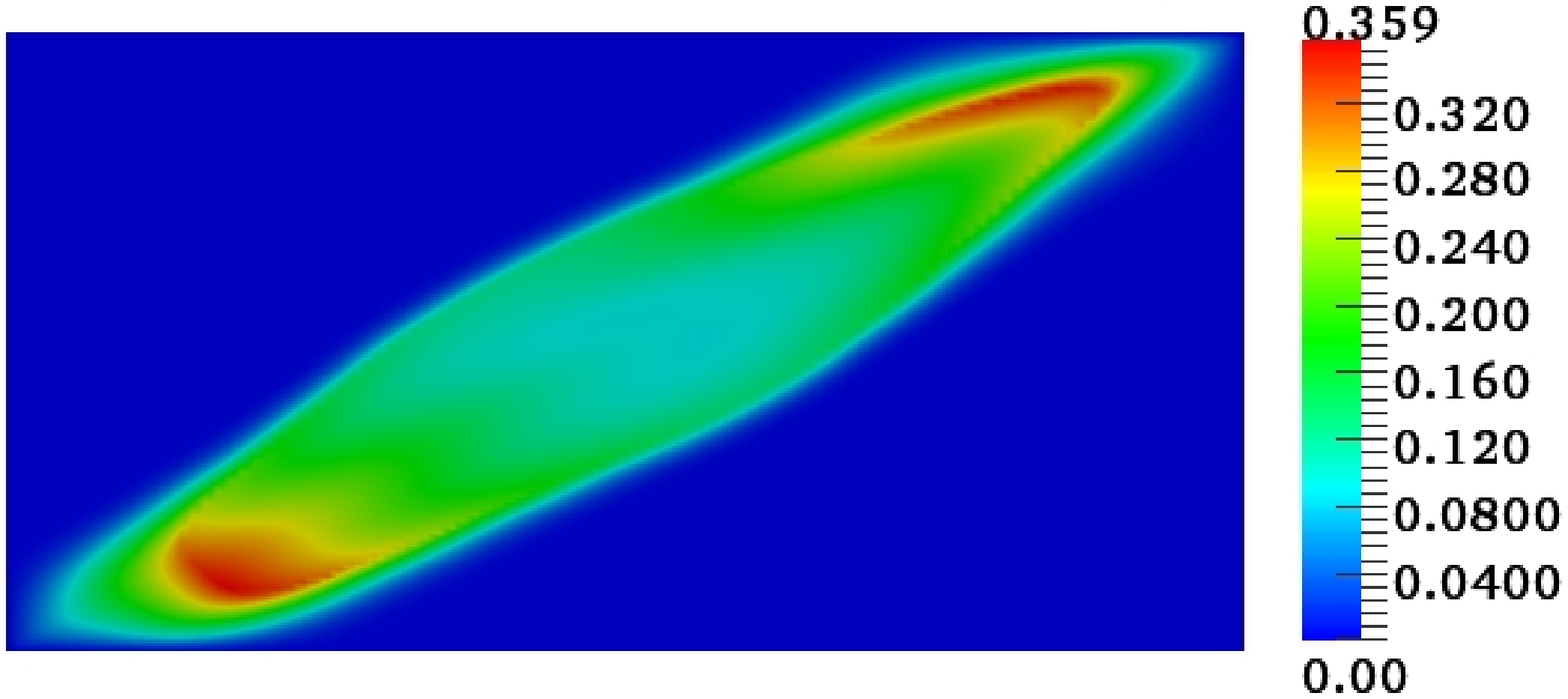}}
  \subfigure[$t = 10.0$]{\includegraphics[scale=0.35]
    {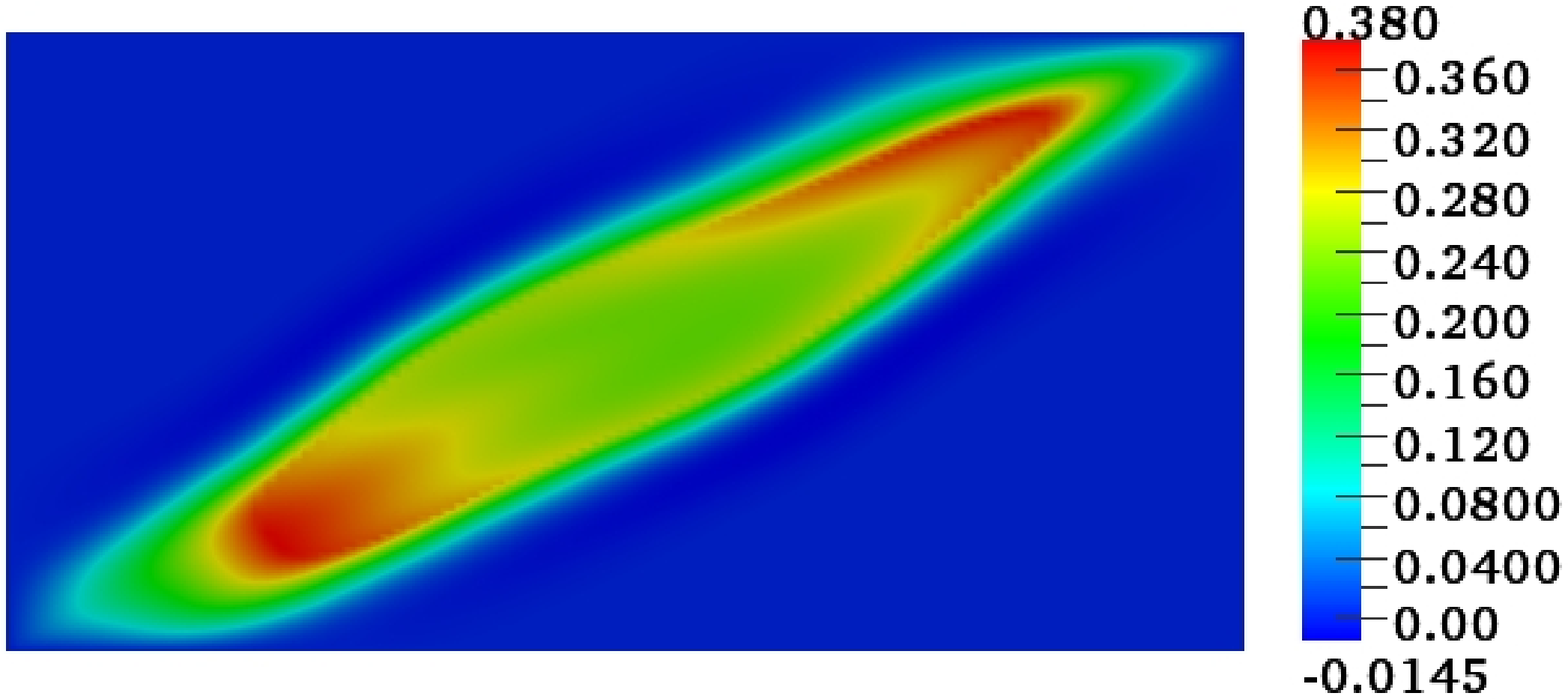}}
  \subfigure[$t = 10.0$]{\includegraphics[scale=0.35]
    {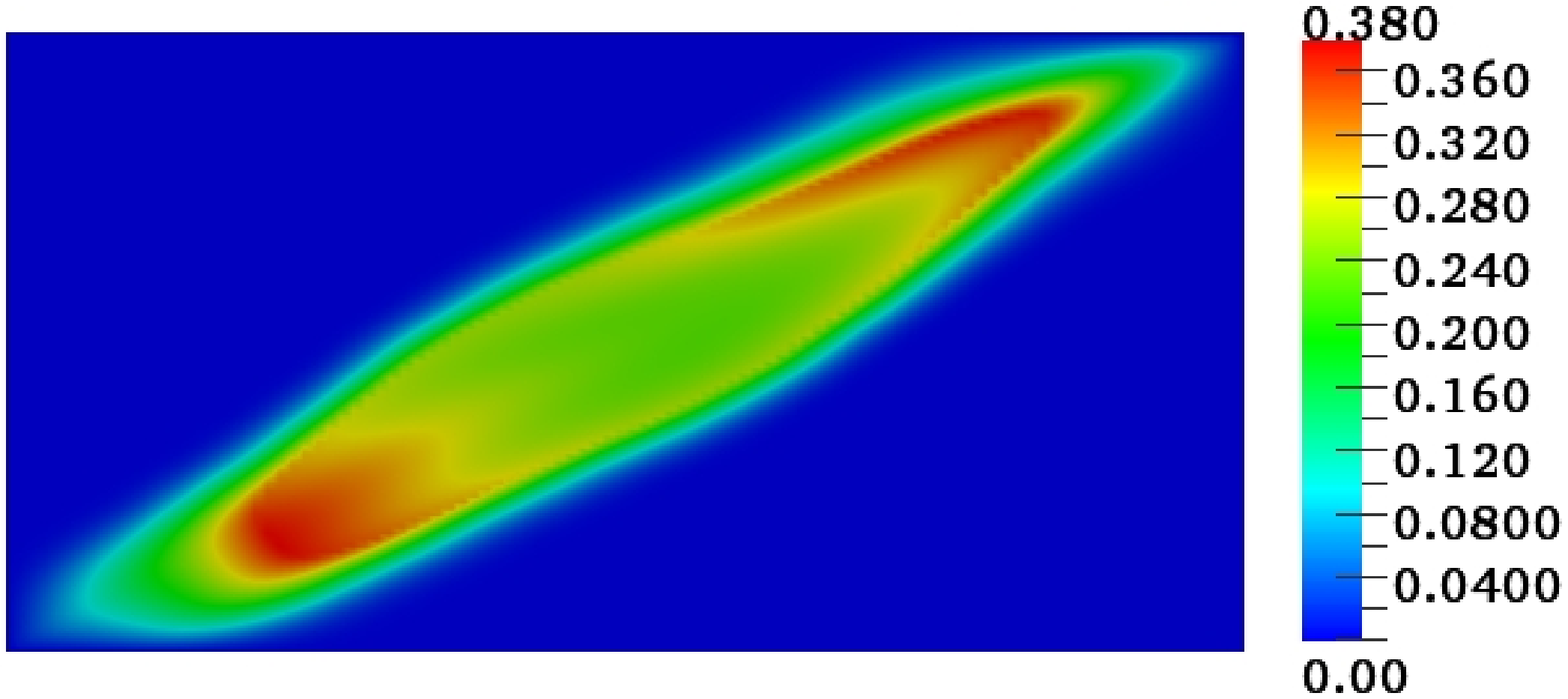}}
  \subfigure[$t = 20.0$]{\includegraphics[scale=0.35]
    {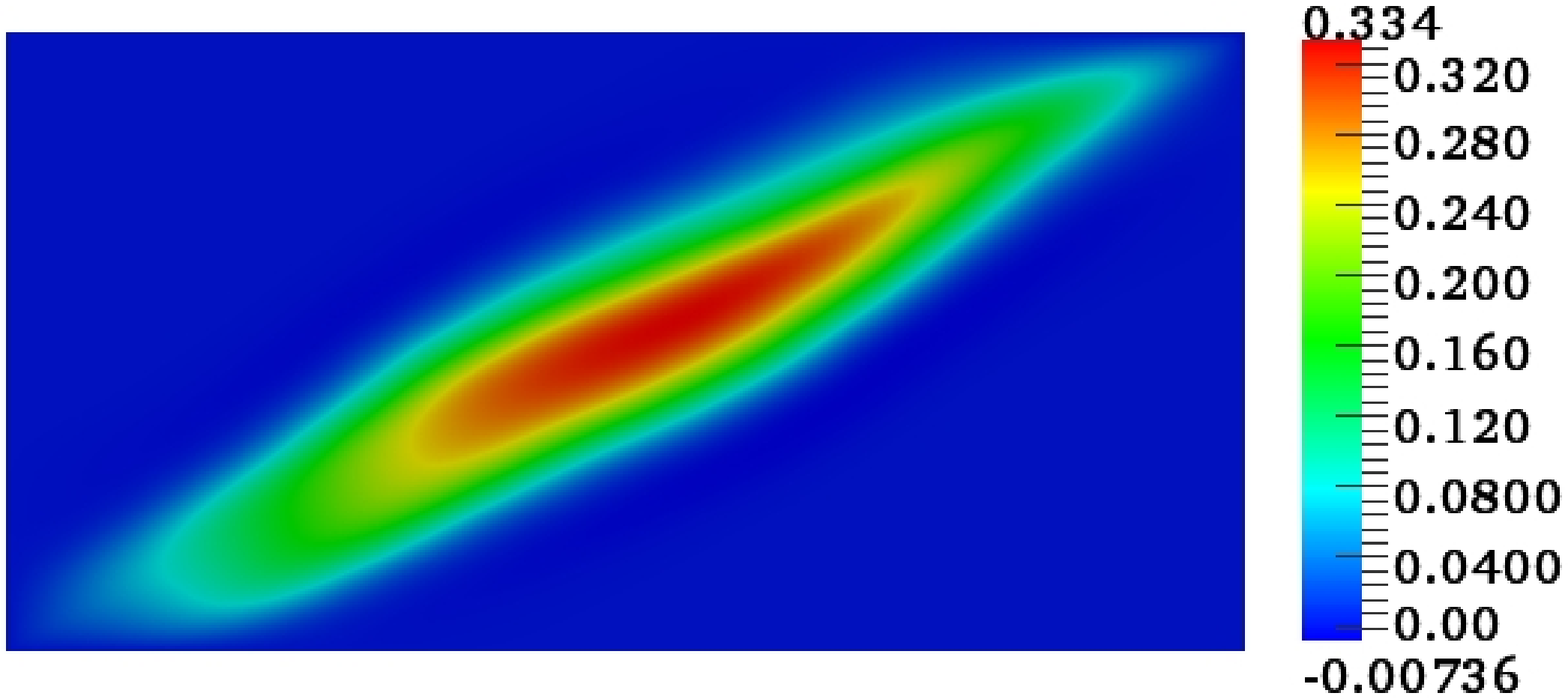}}
  \subfigure[$t = 20.0$]{\includegraphics[scale=0.35]
    {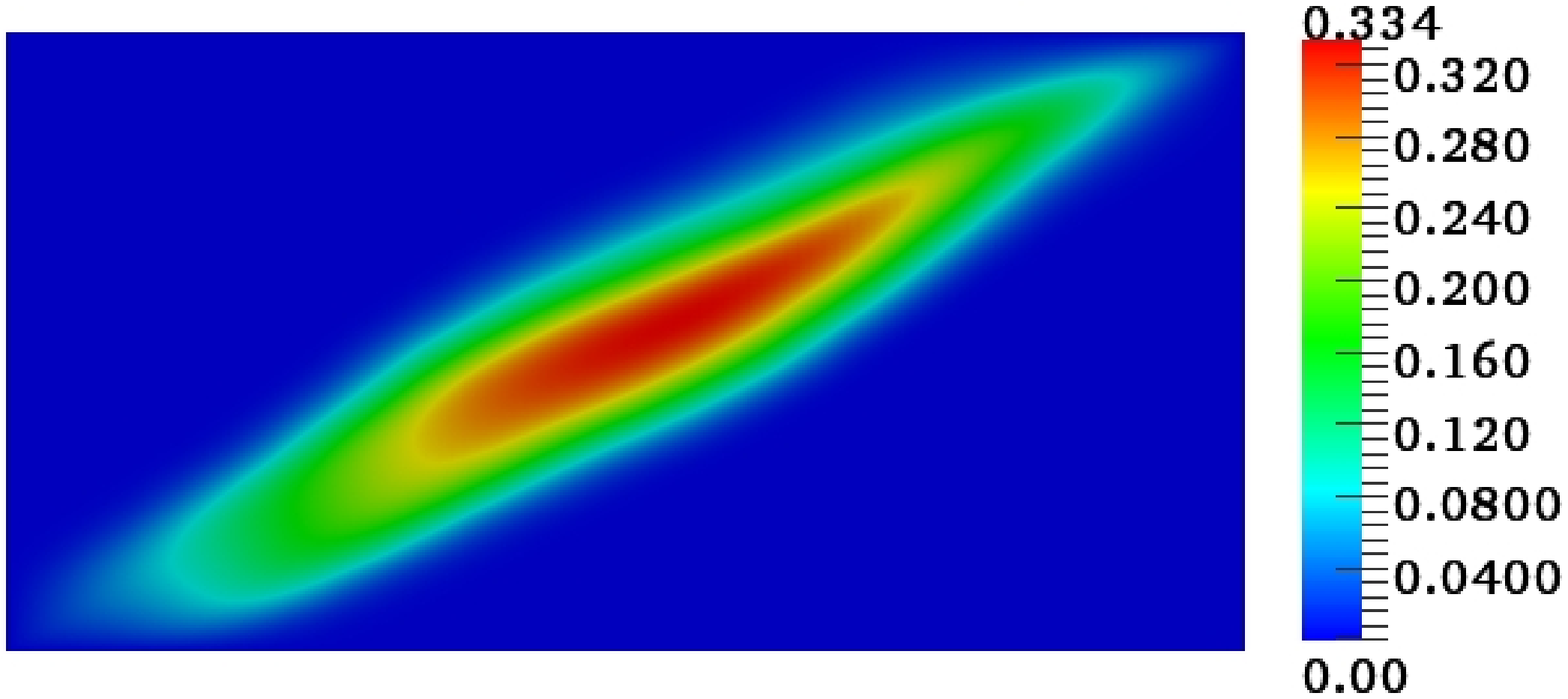}}
 \caption{Diffusion and reaction of an initial slug: This figure 
   shows the contours of the concentration of the product $C$ 
   under the Galerkin single-field formulation (left) and the 
   proposed numerical framework (right) at various time levels. 
   A larger time step of $\Delta t = 1.0 \; \mathrm{s}$ is employed. 
   Again, the Galerkin single-field formulation violated the 
   non-negative constraint, and the proposed numerical framework 
   produced physically meaningful solutions.  \label{Fig:NN_Slug_C_t_1}}
\end{figure}

\begin{figure}
  \centering
  \includegraphics[clip,scale=0.4,angle=-90]  
    {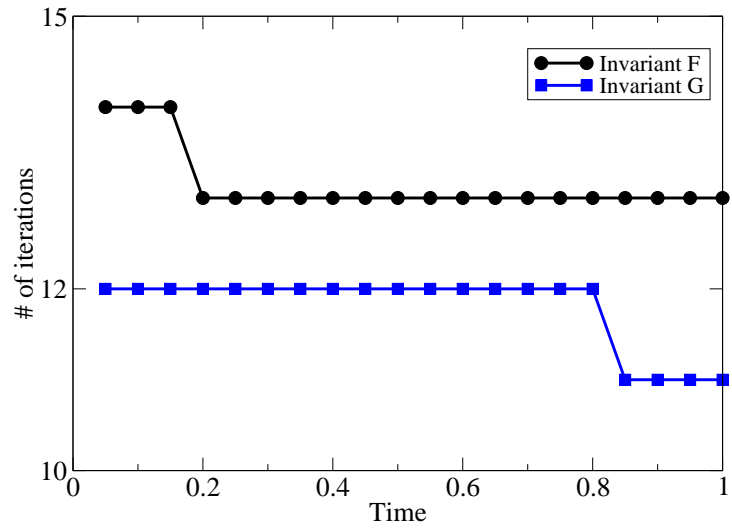}
  \caption{Diffusion and reaction of an initial slug: 
    This figure shows the number of iterations taken by 
    the quadratic programming solver at each time level 
    in obtaining the concentrations for the invariants $F$ 
    and $G$. The time step is taken as $\Delta t = 0.05 \; 
    \mathrm{s}$. \label{Fig:NN_Slug_QP_iterations}}
\end{figure}

\end{document}